\documentclass[11pt,reqno]{amsart}
\usepackage[margin=2.5cm]{geometry}
\allowdisplaybreaks
\usepackage[english]{babel}
\usepackage[latin1]{inputenc}
\usepackage[T1]{fontenc}
\usepackage{amssymb,amsmath,amstext,amsfonts,amsthm,mathrsfs}
\usepackage{mathtools}
\usepackage{relsize}
\usepackage[colorlinks=true, urlcolor=blue, linkcolor=blue, citecolor=blue]{hyperref}
\usepackage{todonotes}
\usepackage[normalem]{ulem}
\usepackage{graphicx}
\usepackage{arydshln}
\usepackage{bm}
\usepackage{cases}
\usepackage{caption}
\usepackage{subcaption}
\usepackage{algorithm,algpseudocode}

\definecolor{OliveGreen}{rgb}{0,0.6,0}

\numberwithin{equation}{section}
\theoremstyle{plain}
\newtheorem*{theorem*}{Theorem}
\newtheorem{theorem}{Theorem}
\numberwithin{theorem}{section}
\newtheorem{proposition}[theorem]{Proposition}
\newtheorem{lemma}[theorem]{Lemma}
\newtheorem{corollary}[theorem]{Corollary}

\newtheorem{remark}[theorem]{Remark}
\newtheorem{assumption}[theorem]{Assumption}

\theoremstyle{definition}

\newcommand{\C}{\mathbb{C}}
\newcommand{\E}{\mathbb{E}}

\newcommand{\Q}{\mathbb{Q}}

\newcommand{\R}{\mathbb{R}}
\newcommand{\sA}{\mathcal{A}}

\newcommand{\sF}{\mathcal{F}}

\newcommand{\sL}{\mathcal{L}}
\newcommand{\sT}{\mathcal{T}}
\newcommand{\sW}{\mathcal{W}}

\newcommand{\sX}{\mathcal{X}}

\newcommand{\pId}{\pmb{\operatorname{Id}}}

\DeclareMathOperator{\Tr}{Tr}
\DeclareMathOperator{\nF}{F}

\def \A {\pmb A}

\def \D {\pmb D}

\def \N {\pmb N}
\def \M {\pmb M}

\def \Q {\pmb Q}

\def \D {\pmb D}
\def \C {\pmb C}

\def \u {\pmb u}
\def \v {\pmb v}

\def \x {\pmb x}

\newcommand\scalemath[2]{\scalebox{#1}{\mbox{\ensuremath{\displaystyle #2}}}}

\makeatletter
\@namedef{subjclassname@2020}{%
	\textup{2020} MSC2020 subject classifications: Primary 60B20; secondary 15B52}
\makeatother

\author[Y. Qi]{Yang Qi}
\address{Mathematical and Algorithmic Sciences Lab, Huawei France R\&D, Paris, France}
\email{yang.qi@inria.fr}

\author[A. Decurninge]{Alexis Decurninge}
\address{Mathematical and Algorithmic Sciences Lab, Huawei France R\&D, Paris, France}
\email{alexis.decurninge@huawei.com}

\subjclass[2020]{}
\keywords{random tensor theory, random matrix theory, low-rank approximation, multi-spiked tensor PCA, multi-index model}

\title[Low-rank random tensors]{Statistical Limits in Random Tensors with Multiple Correlated Spikes}

\date{}

\begin{document}

\begin{abstract}

We use tools from random matrix theory to study the multi-spiked tensor model, i.e., a rank-$r$ deformation of a symmetric random Gaussian tensor. In particular, thanks to the nature of local optimization methods used to find the maximum likelihood estimator of this model, we propose to study the phase transition phenomenon for finding critical points of the corresponding optimization problem, i.e., those points defined by the Karush-Kuhn-Tucker (KKT) conditions. %To this end, we investigate the threshold above which the detection of critical points is possible. 
Moreover, we characterize the limiting alignments between the estimated signals corresponding to a critical point of the likelihood and the ground truth signals. With the help of these results, we propose a new estimator of the rank-$r$ tensor weights by solving a system of polynomial equations, which is asymptotically unbiased contrary the maximum likelihood estimator.

\end{abstract}

\maketitle

\section{Introduction}

In modern statistics, machine learning and signal processing, data often appear in an extremely high dimensional way. To overcome the computational challenges due to high dimensionality, one is interested in developing an observation model to describe or approximate the data set such that the model is as simple as possible but can satisfy simultaneously the following four purposes, i.e., dimension reduction, extracting the most useful information, detecting latent variable structures, and convenience for denoising. In numerous scenarios, high-dimensional data are given in a higher-order multiway form and therefore are also called tensor data~\cite{Comon14, CMDZZCP15, SDFHPF17}. For such a real-valued (symmetric) order-$d$ data tensor $\mathcal{T}$ of size $N\times \dots\times N$, a natural statistical observation model is given for $1\leq i_1,\dots,i_d\leq N$ by
\[
\mathcal{T}_{i_1\dots i_d} = \sum_{j=1}^r \beta_j u^{(j)}_{i_1} \cdots u^{(j)}_{i_d} + \frac{1}{\sqrt{N}} \mathcal{X}_{i_1\dots i_d} \enspace,
\]
with $\mathcal{X}$ Gaussian noise, or in tensor notation, 
\begin{equation}\label{eq: symPCA}
\mathcal{T} = \sum_{j=1}^r \beta_j \pmb{u}_j^{\otimes d} + \frac{1}{\sqrt{N}} \sX \enspace,
\end{equation}
where all unit vectors $\pmb{u}_j = (u^{(j)}_1, \dots, u^{(j)}_N)$, together with the weights $\beta_j$, are to be inferred from the noisy multi-dimensional measurements. Equation~\eqref{eq: symPCA} is called the \emph{noisy symmetric rank-$r$ model} or \emph{multi-spiked tensor model}.
We will assume that the rank-$r$ signal in~\eqref{eq: symPCA} characterized by $\pmb{u}_j$ and $\beta_j$, namely $\sum_{j=1}^r \beta_j \pmb{u}_j^{\otimes d}$, is deterministic. Therefore, only the noise tensor $\sX$ is random and the maximum-likelihood estimator of the signal is given by a solution of the following optimization problem
\begin{equation}\label{lowrkopt}
\begin{aligned}
& \underset{\gamma_j, \pmb{v}_j}{\text{minimize}}
& & \left\lVert \mathcal{T} - \sum_{j=1}^r \gamma_j \pmb{v}_j^{\otimes d} \right\rVert_{\nF} \\
& \text{subject to}
& & \lVert \pmb{v}_1 \rVert_2 = \cdots = \lVert \pmb{v}_r \rVert_2 = 1.
\end{aligned}
\end{equation}
From the low-rank approximation perspective, solving~\eqref{lowrkopt} for the symmetric rank-$r$ model~\eqref{eq: symPCA} can be viewed as a higher-dimensional generalization of the classical Principal Component Analysis. Unlike the matrix case, solving Equation~\eqref{lowrkopt} exactly is at least NP-hard \cite{HL13}. As a special case, the rank-one model, i.e., when $r = 1$, is also called the spiked tensor model and was introduced by Montanari and Richard for tensor Principal Component Analysis (tensor PCA) \cite{richard2014statistical}. In general, a best rank-$r$ approximation of $\mathcal{T}$, i.e., an optimal solution of~\eqref{lowrkopt}, cannot be obtained by performing a tensor deflation, i.e. consecutively computing and subtracting best rank-one approximations like the matrix case \cite{SC10, VNVM14}. In fact, the asymptotic behavior of such tensor deflation has been analyzed in \cite{SGG23} using random matrix tools. However, these results cannot be applied directly to the rank-$r$ tensor decomposition and it is necessary to study Model~\eqref{eq: symPCA} as a whole, which motivates this paper.

In addition to the low-rank approximation problems arising from data science, another source of motivation comes from the study of the phase transition phenomenon of matrices. More precisely, for the matrix case, given an $N\times N$ symmetric rank-$r$ matrix $\pmb{B}_N$ and a noise matrix $\pmb{X}_N$, we are interested in how the eigenvalues and eigenvectors of
\begin{equation}\label{eq:lrkmatBBP}
\pmb{M}_N = \pmb{B}_N + \pmb{X}_N
\end{equation}
are related to the eigenvalues and eigenvectors of $\pmb{B}_N$ and $\pmb{X}_N$. An important feature of Model~\eqref{eq:lrkmatBBP} was pointed out in the seminal work~\cite{BBP05} by Baik, Ben Arous, and P{\'e}ch{\'e}, namely the BBP phase transition phenomenon. They showed in this work the existence of a phase transition threshold, expressed in function of a signal-to-noise ratio, above which the statistical strong detection is possible and below which it is impossible to detect the spikes \cite{onatski2013asymptotic, montanari2015limitation, perry2018optimality}. See also~\cite{johnstone2001distribution, Peche06, baik2006eigenvalues, feral2007largest, paul2007asymptotics, el2008spectrum, CDMF09, johnstone2009consistency, BGN11, bai2012sample, benaych2012singular, ledoit2012nonlinear, birnbaum2013minimax, cai2013sparse, ma2013sparse, vu2013minimax, cai2015optimal, donoho2018optimal, johnstone2018pca} for more work in this direction. A natural question is to investigate similar questions for tensors.

However, different from the matrix case, for tensors, we may have three types of transition thresholds, namely the information-theoretic threshold, statistical thresholds, and algorithmic thresholds. More precisely, when $r = 1$, i.e., in the tensor PCA model~\cite{richard2014statistical}, Model~\eqref{eq: symPCA} takes the form
\begin{equation}\label{eq: symPCArk1}
\sT = \beta\, \u^{\otimes d} + \frac{1}{\sqrt{N}} \sX.
\end{equation}
The information-theoretic threshold $\beta_{\operatorname{IT}}$ is the one such that when $\beta > \beta_{\operatorname{IT}}$ one can possibly detect the spiked signal $\beta \u^{\otimes d}$ and it is not possible to detect this signal when $\beta < \beta_{\operatorname{IT}}$. Once we pick a statistical procedure to estimate $\u$, i.e., once we determine the optimization problem, we have the corresponding statistical transition threshold. For instance, for MLE, the statistical threshold $\beta_{\operatorname{stat}, \operatorname{MLE}}$ is the threshold above which the MLE estimator is weakly correlated with $\u$ and below which the MLE estimator is not correlated with $\u$. In practice, given an optimization problem, instead of assuming unlimited computational resources, we need to design a polynomial-time algorithm and consider its algorithmic threshold $\beta_{\operatorname{Algo}}$, i.e., when $\beta > \beta_{\operatorname{Algo}}$ the algorithm detects the signal and when $\beta < \beta_{\operatorname{Algo}}$ the algorithm cannot find a solution correlated with $\u$. In addition, more subtle thresholds have been proposed, for instance, the statistical weak detection threshold, algorithmic weak detection threshold, and weak recovery threshold. We invite the interested readers to refer to~\cite{PWB20} for more details.

% namely the thresholds corresponding to the following problems.
% \begin{itemize}
% \item Statistical strong detection: distinguish between the spiked and unspiked signals with success $1 - o(1)$ as $N \to \infty$.
% \item Statistical weak detection: distinguish between the spiked and unspiked signals with success $\frac{1}{2} + \epsilon$ as $N \to \infty$, for some $\epsilon > 0$ independent of $N$.
% \item Algorithmic strong detection: there exists a polynomial-time algorithm that distinguishs between the spiked and unspiked signals with success $1 - o(1)$ as $N \to \infty$.
% \item Algorithmic weak detection: there exists a polynomial-time algorithm that distinguishs between the spiked and unspiked signals with success $\frac{1}{2} + \epsilon$ as $N \to \infty$, for some $\epsilon > 0$ independent of $N$.
% \item Strong recovery: given a sample, output a unit vector $\v$ such that $\E[\langle \u, \v \rangle^d] \to 1$ as $N \to \infty$.
% \item Weak recovery: given a sample, output a unit vector $\v$ with nontrivial correlation with the spike, i.e., $\E[\langle \u, \v \rangle^d] \ge \epsilon$ for some $\epsilon > 0$ independent of $N$.
% \end{itemize}

\subsection{Related work}

When $r = 1$, it is known that $\beta_{\operatorname{IT}} = \beta_{\operatorname{stat}, \operatorname{MLE}}$~\cite{chen2019phase, jagannath2020statistical}. Moreover, in~\cite{chen2019phase, jagannath2020statistical}, the sharp information-theoretic threshold is settled for a uniform spherical prior or a Rademacher prior $\u$. The detection and recovery thresholds for other priors, such as spherical and sparse Rademacher, have been studied in~\cite{PWB20}. In~\cite{richard2014statistical}, Montanari and Richard provide three polynomial-time algorithms. More precisely, they show that tensor unfolding provides a polynomial-time algorithm when $\beta \gtrsim N^{(\lceil d/2 \rceil - 1)/2}$, power iteration succeeds when $\beta \gtrsim N^{(d - 1)/2}$, and Approximate Message Passing (AMP) is effective. In~\cite{ben2023long} it is confirmed that the algorithmic threshold for tensor unfolding is $N^{(d-2)/4}$. Moreover, the algorithmic threshold for first-order optimization methods has been shown to be $N^{(d-2)/2}$ in~\cite{richard2014statistical, arous2020algorithmic}, and the algorithmic thresholds for sum-of-squares and spectral methods have been given in~\cite{hopkins2015tensor, perry2018optimality, PWB20}.

When $r > 1$, the detection and recovery problems for Model~\eqref{eq: symPCA} are much more difficult to study. In~\cite{lesieur2017statistical, chen2021phase}, the information-theoretic threshold is given under the condition that $\u_1, \dots, \u_r$ are sampled independently. In \cite{huang2022power}, the algorithmic threshold for power iteration is addressed under the condition that $\u_1, \dots, \u_r$ are assumed to be orthogonal to each other. In~\cite{arous2024high, arous2024stochastic}, the algorithmic thresholds for two local methods, i.e., online stochastic gradient descent and gradient flow, are provided, under the assumption that $\u_1, \dots, \u_r$ are orthogonal to each other.

\subsection{Contributions}

In general, local methods provide us critical points instead of global optima. Thus we propose to consider the phase transition phenomenon for finding critical points of the maximum likelihood optimization problem in \eqref{lowrkopt}, i.e., those points defined by the KKT conditions. In particular, we investigate the threshold above which detection of critical points is possible. Moreover, we study the asymptotic behavior of \textit{summary statistics} expressing the relationship between the ground truth signals $\u_i$ and the estimated signals corresponding to a critical point, as introduced in \cite{SGG23}. As a side-product, we provide estimators of the summary statistics related to the true signals, including estimators of $\beta_i$. Our approach is inspired by~\cite{GCC22, SGC24}, which study the rank-one version of Model~\eqref{eq: symPCA} in both symmetric and asymmetric case. In fact, studying the critical points mentioned above is an interesting question for its own sake. For instance, in~\cite{arous2019landscape}, the expected number of critical points of the rank-one case is computed. It is worth mentioning that our results do not assume orthogonality neither between the true signals in the model nor between the computed estimators.

\subsection{Organization}

This paper is organized as follows. We briefly recall basic tools from random matrix theory in Section~\ref{sec:pre}, and set up the problem we will be working on in Section~\ref{sec:setting}. 
In Section~\ref{sec:spectrum}, we define a random matrix $\flat(\sT)$ which arises from the KKT conditions and in Theorem~\ref{th:explicitmu_noise} we give its limiting empirical spectral measure.
Then, we study the limiting alignments between the critical points and ground truth signals in Section~\ref{sec:limitingalignments}. In Section~\ref{sec:inference}, we propose estimators of the summary statistics related to the ground truth signals. 

\section{Preliminaries}\label{sec:pre}

\subsection{Notations}

In this paper, we use calligraphic letters to denote tensors. For a $d$th order tensor $\mathcal{Y}$, we use the uppercase letter $Y_{i_1\dots i_d}$ for its $(i_1, \dots, i_d)$th coordinate. To distinguish from tensors, we use lowercase bold letters to denote vectors. Given a vector $\v_1$, we denote its $j$th entry by the lowercase letter $v^1_j$. For matrices, we employ uppercase bold letters for their notations. For a matrix $\Q^{11}$, by the uppercase letter $Q^{11}_{ij}$ we denote its $(i, j)$th entry. Given an arbitrary matrix $\pmb{M}$, we note $\pmb{M}^{\odot d}$ its Hadamard product iterated $d$ times. The identity matrix will be denoted by $\pId$.
In order to distinguish random quantities that depend on the random tensor $\sX$ from deterministic quantities, we use hat notation for the former, e.g., $\hat\v_1$ denotes a random vector depending on $\sX$.
The tensor contraction operation is denoted by $\langle , \rangle$. For instance, the $(i_{d-1}, i_d)$th entry of $\langle \mathcal{Y}, \v_1^{\otimes (d-2)}\rangle$ is given by
\[
\langle \mathcal{Y}, \v_1^{\otimes (d-2)}\rangle_{i_{d-1}i_d} = \sum_{i_1 = 1}^N \cdots \sum_{i_{d-2} = 1}^N Y_{i_1 \dots i_d} v^1_{i_1} \cdots v^1_{i_{d-2}} \enspace.
\]
For convenience, we use $[r]$ to denote the set $\{1, \dots, r\}$.

\subsection{Random matrix theory tools}

In this subsection we recall some useful tools from random matrix theory, which will play important roles in our analysis. 
In random matrix theory, given a random symmetric matrix, we are interested in its spectral measure. More precisely, given an $N\times N$ symmetric matrix $\pmb{A}$, the empirical spectral measure of $\pmb{A}$ is defined by
\[
\mu_{\pmb{A}} = \frac{1}{N} \sum_{i=1}^N \delta_{\lambda_i(\pmb{A})},
\]
where $\lambda_1(\pmb{A}) \ge \cdots \ge \lambda_N(\pmb{A})$ are the ordered eigenvalues of $\pmb{A}$, and $\delta_{\lambda_i(\pmb{A})}$ is a Dirac mass on $\lambda_i(\pmb{A})$. The limiting behavior of $\mu_{\pmb{A}}$ is a central subject in random matrix theory. The resolvent method provides a powerful tool to study these spectral characteristics. For any $z \in \mathbb{C}_+ = \{z\in \mathbb{C}\mid \mathfrak{Im}(z) > 0\}$, the \emph{resolvent} of $\pmb{A}$ is defined by
\[
\Q_{\pmb{A}}(z) = (\A - z \pId)^{-1}.
\]
On the other hand, for a finite Borel measure $\mu$ on $\mathbb{R}$, we define its \emph{Cauchy-Stieltjes transform} $S_{\mu}$ on $\mathbb{C} \setminus \mathbb{R}$ by
\[
S_{\mu} (z) = \int_{\R} \frac{1}{t - z} \,d\mu(t) \quad \text{for } z \in \mathbb{C}\setminus\R.
\]
Thus $S_{\mu_{\pmb{A}}} (z)$ contains all the information about $\mu_{\pmb{A}}$. The important relation between the Cauchy-Stieltjes transform of $\mu_{\pmb{A}}$ and the resolvent $\Q(z)$ is given by
\begin{equation}\label{relStieltjesSpec}
S_{\mu_{\pmb{A}}} (z) = \frac{1}{N} \operatorname{Tr} (\Q_{\pmb{A}}(z)) \enspace.
\end{equation}
Below we list several basic properties of the resolvent matrix.
\begin{enumerate}
\item Analytic: given $\pmb{A}$, each $(i, j)$th entry of $\Q_{\pmb{A}}(z)$ is an analytic function from $\mathbb{C}_+$ to $\mathbb{C}_+$.
\item Bounded: $\lVert \Q_{\pmb{A}}(z) \rVert \le \operatorname{dist}(z, \operatorname{Spec}(\pmb{A}))^{-1}$, where $\operatorname{Spec}(\pmb{A})$ is the spectrum of $\pmb{A}$.
\item Complement formula: for fixed $i$, the $(i, i)$th entry of $\Q_{\pmb{A}}(z)$, denoted by $Q_{ii}$, is given by
\[
Q_{ii} = - \bigl(z - A_{ii} + (\pmb{Q}_iA_i)^{\dag}A_i\bigr)^{-1},
\]
where $A_i = (A_{ij})_{j\ne i} \in \mathbb{C}^{N-1}$, $\pmb{Q}_i = (\pmb{A}^i - z\pId)^{-1}$, $\pmb{A}^i$ is the principal submatrix of $\pmb{A}$ obtained by removing the $i$th row and the $i$th column of $\pmb{A}$.
\end{enumerate}
By~\eqref{relStieltjesSpec}, we can study the limits of the moments of $\mu_{\pmb{A}}$ via studying the limit of the trace of the resolvent matrix and using the following inverse formula.
\begin{theorem}[Inverse formula of Stieltjes-Perron]
For two points $a < b$ where the cumulative distribution function of $\mu$ is continuous,
\[
\mu([a, b]) = \frac{1}{\pi} \lim_{\epsilon\to 0} \int_a^b \mathfrak{Im}\bigl(S_{\mu}(x + i\epsilon)\bigr) \,dx \enspace.
\]
If $\mu$ admits a density function at $x$, then the density at $x$ is given by
\[
\frac{1}{\pi} \lim_{\epsilon \to 0} \mathfrak{Im}\bigl(S_{\mu} (x + i\epsilon)\bigr) \enspace.
\]
\end{theorem}
The inversion formula does not only imply a one-to-one correspondence between Stieltjes functions $S_{\mu}(z)$ and probability measures $\mu$, it also implies the approximate statement holds. More precisely, we have the following equivalence (see for example~\cite{tao2012topics}).
\begin{theorem}
Let $\mu$ and $(\mu_n)_{n\ge 1}$ be a sequence of real probability measures. Then the following are equivalent
\begin{enumerate}
\item $\mu_n \to \mu$ weakly.
\item For any $z\in \mathbb{C}_+$, $S_{\mu_n}(z) \to S_{\mu}(z)$.
\item There is a subset $D \subseteq \mathbb{C}_+$ with an accumulation point in $\mathbb{C}_+$ such that for any $z\in D$, $S_{\mu_n}(z) \to S_{\mu}(z)$.
\end{enumerate}
\end{theorem}

\subsection{Gaussian calculations}

We will focus on the Gaussian noise in this paper, and will frequently apply the following Stein's lemma, a.k.a. Stein's identity or Gaussian integration by parts, to the computations.
\begin{lemma}[Stein's lemma]
Let $X \sim \mathcal{N}(0, \sigma^2)$ and $f: \mathbb{R} \to \mathbb{R}$ be continuously differentiable with at most polynomial growth. Then
\[
\E[Xf(X)] = \sigma^2\E[f'(X)].
\]
\end{lemma}
The following property of Gaussian random variables will help us control the variance of a function.
\begin{lemma}\label{lem:Poincare}
A standard Gaussian variable $X$ on $\mathbb{R}^N$ satisfies Poincar\'e inequality with constant $1$, namely, for any differential function $f$ with $\E f(X)^2 < \infty$,
\[
\operatorname{\mathbb{V}ar} \bigl(f(X)\bigr) \le \E\lVert \nabla f(X)\rVert^2_2.
\]
\end{lemma}

\subsection{Symmetric Gaussian distribution}
Let us explicitly describe the distribution of the symmetric Gaussian noise tensor $\sX$. It is characterized by the following equality:
\begin{equation}\label{eq:symmodel_noise}
\sX_{i_1\dots i_d} =  \frac{1}{d!} \sum_{\pi\in\mathfrak{S}_d} \sW_{i_{\pi(1)}\dots i_{\pi(d)}},
\end{equation}
where $\sW$ is a Gaussian noise tensor with i.i.d. entries $\sW_{i_1\dots i_d} \sim \mathcal{N}(0, 1)$ and $\mathfrak{S}_d$ is the symmetric group on the set $[d]$. Thus $\E(\sX_{i_1\dots i_d}) = 0$, and we denote the variance by $\sigma_{\bm{i}}^2 = \sigma^2_{i_1 \dots i_d}(\sX)$. More precisely, assume as a multiset
\[
I = \{i_1, \dots, i_d\} = \{\ell_1, \dots, \ell_1, \dots, \ell_k, \dots, \ell_k\},
\]
where $\ell_1, \dots, \ell_k$ are distinct integers, and the number of $\ell_i$ contained in $I$ is denoted by $m_i$. Then
\begin{equation}\label{eq:var}
\sigma^2_{\bm{i}} = \frac{1}{\binom{d}{m_1, \dots, m_k}} \enspace.
\end{equation}
\section{General setting}\label{sec:setting}

Let us recall that we study in this paper the following noisy symmetric rank-$r$ model
\begin{equation}\label{eq:tensormodel}
\mathcal{T} = \sum_{i=1}^r \beta_i \u_i^{\otimes d} + \frac{1}{\sqrt{N}} \sX,
\end{equation}
where $\sT$ is our observed tensor, $\beta_1, \dots, \beta_r$ are the signal-to-noise ratios, $\u_1, \dots, \u_r$ are unit-length vectors to be inferred, and $\sX$ is a symmetric Gaussian noise whose elements are drawn as in \eqref{eq:symmodel_noise}. Without loss of generality, we further assume
\begin{equation}\label{eq:betarelation}
\lvert \beta_1 \rvert \ge \cdots \ge \lvert \beta_r \rvert.
\end{equation}
Given a sequence of tensors $(\sT_N)_{N\ge 1}$ such that each member $\sT_N$ satisfies Model~\eqref{eq:tensormodel}, we are interested in its limiting behaviour when $N$ goes to infinity. For convenience of notation, we drop the $N$ assuming implicitly its dependence on $N$.

Recall that, by definition, $\sum_{j=1}^r \hat\gamma_j \hat\v_j^{\otimes d}$ is a best symmetric rank-$r$ approximation of $\sT$ if and only if $\sum_{j=1}^r \hat\gamma_j \hat\v_j^{\otimes d}$ is a solution of the  minimization problem \eqref{lowrkopt}.
Generally speaking, given a tensor $\sT$, it is possible that the infimum in~\eqref{lowrkopt} cannot be achieved \cite{deSilva2008tensor}. However, in our case, due to the Gaussian noise $\mathcal{\sW}$, the infimum can be always obtained. For this optimization problem~\eqref{lowrkopt}, we start from its Karush-Kuhn-Tucker (KKT) conditions. To this end, we will consider the Lagrangian of~\eqref{lowrkopt}, which is given by
\begin{equation}\label{eq:Lagop}
\sL(\gamma_1, \dots, \gamma_r, \v_1, \dots, \v_r, \mu_1, \dots, \mu_r) = \Vert \sT - \sum_{j=1}^r \gamma_j \v_j^{\otimes d}\Vert^2 + \sum_{j=1}^r \mu_j(\Vert \v_j \Vert^2 - 1).
\end{equation}
Then the KKT conditions are given by
\begin{proposition}\label{prop:lowrkappopt}
Given a symmetric tensor $\sT$, assume $\sum_{j=1}^r \hat\gamma_j \hat\v_j^{\otimes d}$ is a best symmetric rank-$r$ approximation of $\sT$. Then the tuple $(\hat\gamma_1, \dots, \hat\gamma_r, \hat\v_1, \dots, \hat\v_r)$ satisfies
\begin{equation}\label{eq:rkrapprox}
\begin{dcases}
\langle \sT - \sum_{j=1}^r \hat\gamma_j \hat\v_j^{\otimes d}, \hat\v_i^{\otimes (d-1)} \rangle = 0 \\
\langle \hat\v_i, \hat\v_i \rangle = 1
\end{dcases}
\end{equation}
for any $i\in [r]$.
\end{proposition}
\begin{proof}
See Appendix~\ref{subsec:proplowrkopt}.
\end{proof}

Similarly to ~\eqref{eq:betarelation}, we will also assume in the sequel
\begin{equation}
\vert\hat\gamma_1\vert \geq \dots \geq \vert\hat\gamma_r\vert .
\end{equation} 
We are interested in evaluating how  \textit{estimated signals} $(\hat\gamma_1\hat\v_1,\dots,\hat\gamma_r\hat\v_r)$ satisfying \eqref{eq:rkrapprox} relate to the \textit{true signals} $(\beta_1\u_1,\dots,\beta_r\u_r)$, when $N$ goes large.
For this purpose, we first reformulate \eqref{eq:rkrapprox} so that the signals $(\hat\gamma_1\hat\v_1,\dots,\hat\gamma_r\hat\v_r)$ form an eigenvector of some matrix correlated with the noise tensor, and then apply tools from random matrix theory to study its asymptotic behavior. More precisely, let us define the Gram matrix
\begin{equation}
\pmb{\hat R}_{vv,r} = \begin{bmatrix}
    \langle \hat\v_1,\hat\v_1\rangle  & \dots & \langle \hat\v_1,\hat\v_r\rangle \\
    \vdots & \ddots  & \vdots  \\
    \langle \hat\v_r,\hat\v_1\rangle & \dots & \langle \hat\v_r,\hat\v_r\rangle \\
\end{bmatrix},
\end{equation}
whose $(d-1)$th Hadamard power is given by
\[
\pmb{\hat{R}}^{\odot (d-1)}_{vv,r} = \begin{bmatrix}
    1 & \langle \hat\v_1, \hat\v_2 \rangle^{d-1} & \langle \hat\v_1, \hat\v_3 \rangle^{d-1} & \dots  & \langle \hat\v_1, \hat\v_r \rangle^{d-1} \\
    \langle \hat\v_1, \hat\v_2 \rangle^{d-1} &  & \langle \hat\v_2, \hat\v_3 \rangle^{d-1} & \dots  & \langle \hat\v_2, \hat\v_r \rangle^{d-1} \\
    \vdots & \vdots & \vdots & \ddots & \vdots \\
    \langle \hat\v_1, \hat\v_r \rangle^{d-1} & \langle \hat\v_2, \hat\v_r \rangle^{d-1} & \langle \hat\v_3, \hat\v_r \rangle^{d-1} & \dots  & 1
\end{bmatrix}.
\]
In addition, let
\begin{equation}\label{eq:Wdef}
\pmb{\hat W}_r = 
\Big(\pmb{\hat{R}}^{\odot (d-1)}_{vv,r}\Big)^{-1}
\begin{bmatrix}
    \frac{\hat\gamma_r}{\hat\gamma_1}  & 0 & \dots  & 0 \\
    0 & \frac{\hat\gamma_r}{\hat\gamma_2}  & \dots  & 0 \\
    \vdots & \vdots & \ddots & \vdots \\
    0 & 0 & \dots  & 1
\end{bmatrix}, 
\end{equation}
and
\begin{equation}\label{eq:defbsTgeneral}
\flat(\sT) = (\pmb{\hat W}_r\otimes\pId) \begin{bmatrix}
    \langle \sT, \hat\v^{\otimes (d-2)}_1 \rangle & 0 & \dots  & 0 \\
    0 & \langle \sT, \hat\v^{\otimes (d-2)}_2 \rangle & \dots  & 0 \\
    \vdots & \vdots & \ddots & \vdots \\
    0 & 0 & \dots  & \langle \sT, \hat\v^{\otimes (d-2)}_r \rangle
\end{bmatrix},
\end{equation}
where $\otimes$ denotes the Kronecker product.
Note that the definition of $\flat(\sT)$ depends on the choice of $\hat\v_1, \dots, \hat\v_r, \hat\gamma_1, \dots, \hat\gamma_r$ as well. Yet, to simplify the notation, we will only keep the dependency on $\sT$.
We have then the following result expressing the solutions of \eqref{eq:rkrapprox} as eigenvectors of $\flat(\sT)$.
\begin{proposition}\label{prop:lowrkappopt2}
Consider a tuple $(\hat\gamma_1, \dots, \hat\gamma_r, \hat\v_1, \dots, \hat\v_r)\in\mathbb{R} \times \cdots \times \mathbb{R} \times\mathbb{S}^{N-1}\times \cdots \times \mathbb{S}^{N-1}$ such that ${\hat \v_1}^{\otimes (d-1)}, \dots, {\hat \v_r}^{\otimes (d-1)}$ are linearly independent.
Then $(\hat\gamma_1, \dots, \hat\gamma_r, \hat\v_1, \dots, \hat\v_r)$ is a solution of Equation~\eqref{eq:rkrapprox} if and only if the vector $\begin{bmatrix}\hat\gamma_1\hat\v_1 \cdots \hat\gamma_r\hat\v_r\end{bmatrix}^{\top}$ is an eigenvector of $\flat(\sT)$ associated with eigenvalue $\hat\gamma_r$, i.e.,
\begin{equation}\label{eq:kktrkrmatrix}
\flat(\sT)\begin{bmatrix}\hat\gamma_1\hat\v_1\\ \vdots \\ \hat\gamma_r\hat\v_r\end{bmatrix} = \hat\gamma_r \begin{bmatrix}\hat\gamma_1\hat\v_1\\ \vdots \\ \hat\gamma_r\hat\v_r\end{bmatrix} \enspace.
\end{equation}
\end{proposition}
\begin{proof}
See Appendix~\ref{subsec:proplowrkopt2}.
\end{proof}

% \begin{equation}
% \begin{bmatrix}
%     1 & \langle \hat\v_1, \hat\v_2 \rangle^{d-1} & \langle \hat\v_1, \hat\v_3 \rangle^{d-1} & \dots  & \langle \hat\v_1, \hat\v_r \rangle^{d-1} \\
%     \langle \hat\v_1, \hat\v_2 \rangle^{d-1} & 1  & \langle \hat\v_2, \hat\v_3 \rangle^{d-1} & \dots  & \langle \hat\v_2, \hat\v_r \rangle^{d-1} \\
%     \vdots & \vdots & \vdots & \ddots & \vdots \\
%     \langle \hat\v_1, \hat\v_r \rangle^{d-1} & \langle \hat\v_2, \hat\v_r \rangle^{d-1} & \langle \hat\v_3, \hat\v_r \rangle^{d-1} & \dots  & 1
% \end{bmatrix}
% \end{equation}
% is invertible. 

We will call the tuples $(\hat\gamma_1\hat\v_1, \dots,\hat\gamma_r\hat\v_r)$ satisfying \eqref{eq:kktrkrmatrix} (or equivalently \eqref{eq:rkrapprox}) the \textit{critical points} of the likelihood. The maximum likelihood estimator is obviously one of the critical points but is not the only one. We want to study the asymptotic behavior of these sequences of critical points. In particular, we consider the asymptotic behavior of summary statistics among 
\[
\langle\hat\v_i,\hat\v_j\rangle,\quad\langle\u_i,\u_j\rangle,\quad\langle\u_i,\hat\v_j\rangle \text{\quad  and \quad} \gamma_j\text{\quad for $1\leq i,j\leq r$}.
\]
We will focus on sequences of critical points such that their summary statistics converge when $N\rightarrow\infty$. Showing such convergence is challenging and it is likely that an arbitrary sequence of critical points is not convergent. However, we conjecture that specific sequences of critical points such as the sequence of maximum likelihood estimators can be proven to have convergent summary statistics as it is the case in the rank-1 case \cite{jagannath2020statistical}.

%Equation~\eqref{eq:defbsTgeneral} reads as
%\begin{eqnarray}
%\flat(\sT)&=& \frac{1}{1-\hat\lambda^2} \begin{bmatrix}
%\hat\nu\pId  & -\hat\lambda\pId  \\
%-\hat\lambda\hat\nu\pId   & \pId
% \end{bmatrix}\begin{bmatrix}
% \langle \sT, \hat\v_1^{\otimes (d-2)} \rangle & 0 \\
%0 &  \langle \sT, \hat\v_2^{\otimes (d-2)} \rangle
% \end{bmatrix}
%\\ &=&
%\frac{1}{1-\hat\lambda^2} \begin{bmatrix}
%\hat\nu  \langle \sT, \hat\v_1^{\otimes (d-2)} \rangle & -\hat\lambda \langle \sT, \hat\v_2^{\otimes (d-2)} \rangle \\
%-\hat\lambda \hat\nu \langle \sT, \hat\v_1^{\otimes (d-2)} \rangle & \langle \sT, \hat\v_2^{\otimes (d-2)} \rangle
% \end{bmatrix}.
%\end{eqnarray}
In order to understand the asymptotic behavior of sequence of critical points, we study the spectra of $\flat(\sT)$. 
To this end, we define its resolvent as
\[
\Q(z) = (\flat(\sT) - z\pId)^{-1} = \begin{bmatrix}
\Q^{11}(z) & \Q^{12}(z) & \dots & \Q^{1r}(z) \\
\Q^{21}(z) & \Q^{22}(z) & \dots & \Q^{2r}(z) \\
\vdots &\vdots  & \ddots & \vdots \\
\Q^{r1}(z) & \Q^{r2}(z) & \dots & \Q^{rr}(z) \\
\end{bmatrix}.
\]
Note that $\flat(\sT)$ may not be symmetric. Nevertheless, the following lemma holds.
\begin{lemma}\label{lem:eigenreal}
When ${\hat \v_1}^{\otimes (d-1)}, \dots, {\hat \v_r}^{\otimes (d-1)}$ are linearly independent, the eigenvalues of $\flat(\sT)$ are real, and the Stieltjes transform of $\mu_{\flat(\sT)}$ satisfies
\begin{equation}
S_{\mu_{\flat(\sT)}} (z) = \frac{1}{rN} \operatorname{Tr} (\Q(z)) \enspace.
\end{equation}
\end{lemma}

\begin{proof}
See Appendix~\ref{subsec:lemeigenreal}.
\end{proof}

\begin{remark}
    From the proof of Proposition~\ref{prop:lowrkappopt2} we can see that $\pmb{\hat{R}}^{\odot (d-1)}_{vv,r}$ is positive-definite if and only if 
    ${\hat \v_1}^{\otimes (d-1)}, \dots, {\hat \v_r}^{\otimes (d-1)}$ are linearly independent.
    Here the condition that ${\hat \v_1}^{\otimes (d-1)}, \dots, {\hat \v_r}^{\otimes (d-1)}$ are linearly independent is weaker than requiring $\hat \v_1, \dots, \hat \v_r$ to be linearly independent. For example, for linearly dependent vectors $\hat \v_1 = (1, 0)$, $\hat \v_2 = (0, 1)$, and $\hat \v_3 = (1, 1)$, the powers ${\hat \v_1}^{\otimes 3}, {\hat \v_2}^{\otimes 3}, {\hat \v_3}^{\otimes 3}$ are linearly independent.
\end{remark}

%Thanks to Lemma~\ref{lem:eigenreal}, though $\M$ is not symmetric, we can still define its resolvent
%\[
%\Q(z) = (\M - z\pId)^{-1} = \begin{bmatrix}
%\Q^{11}(z) & \Q^{12}(z) \\
%\Q^{21}(z) & \Q^{22}(z)
%\end{bmatrix}
%\]
%and study the empirical spectral measure of $\M$ via its Stieljes transform, which is the main purpose of the next section.

\section{Asymptotic spectrum of $\flat(\sT)$}\label{sec:spectrum}
\subsection{Sequence of critical points with fully convergent summary statistics}\label{sec:full}
%  First note
% \begin{equation}
% \pmb{\hat R}_{vv,r} = \begin{bmatrix}
%     \langle \hat\v_1,\hat\v_1\rangle  & \dots & \langle \hat\v_1,\hat\v_r\rangle \\
%     \vdots & \ddots  & \vdots  \\
%     \langle \hat\v_r,\hat\v_1\rangle & \dots & \langle \hat\v_r,\hat\v_r\rangle \\
% \end{bmatrix}
% \end{equation}
% and 
Consider the following assumption where we assume that the summary statistics of the $r$ estimated signals are converging.
\begin{assumption}\label{assump:general}
The sequence $(\hat\gamma_1, \dots,\hat\gamma_r, \hat\v_1, \dots,\hat\v_r)$ satisfies 
\[
 \hat\gamma_i \xrightarrow{\text{a.s.}} \gamma_i \text{\ for $1\leq i\leq r$ }, \pmb{\hat R}_{vv,r} \xrightarrow{\text{a.s.}} \pmb{R}_{vv,r}\in\mathbb{R}^{r\times r} \text{such that $\pmb{ R}^{\odot (d-1)}_{vv,r}$ is invertible.} 
\]
\end{assumption}
From the proof of Proposition~\ref{prop:lowrkappopt2}, we can see Assumption~\ref{assump:general} indicates that $\pmb{ R}^{\odot (d-1)}_{vv,r}$ is positive-definite. It is immediate to see that Assumption~\ref{assump:general} yields the convergence of $\pmb{\hat W}_r$,
\begin{equation}\label{eq:Wdef2}
\pmb{\hat W}_{r}\xrightarrow{a.s.} \pmb{W}_{r} = ( \pmb{ R}^{\odot (d-1)}_{vv,r})^{-1}
\begin{bmatrix}
    \frac{\gamma_{r}}{\gamma_1}  & 0 & \dots  & 0 \\
    0 & \frac{\gamma_{r}}{\gamma_2}  & \dots  & 0 \\
    \vdots & \vdots & \ddots & \vdots \\
    0 & 0 & \dots  & 1
\end{bmatrix}.
\end{equation}
Since $\pmb{W}_{r}$ is the product of a symmetric positive-definite matrix and a symmetric matrix, $\pmb{W}_{r}$ is diagonalizable. By applying techniques from random matrix theory, we have the following property of the limits of the Stieltjes transforms $\frac{1}{N} \Tr \Q^{ij}(z)$.

\begin{theorem}\label{thm:generalized_measure}
Define $\pmb{P}_r$ and $\lvert\kappa_1^{(r)}\rvert\geq \dots\geq \lvert\kappa_r^{(r)}\rvert$ through the following eigenvalue decomposition
\begin{equation}\label{eq:Pdef2}
\frac{2}{\sqrt{d(d-1)}}\pmb{W}_{r} = \pmb{P}_r\begin{bmatrix}
    \kappa_1^{(r)}  & 0  & 0 \\
     0 & \ddots & 0 \\
    0 & 0   & \kappa_r^{(r)}\end{bmatrix}\pmb{P}_r^{-1}.
\end{equation}
Under Assumption~\ref{assump:general}, when
\begin{equation}
|\gamma_r| \geq (d-1) \lvert\kappa_1^{(r)}\rvert,
\end{equation}
the quantity $\frac{1}{N} \Tr \Q^{ij}(z)$ converges almost surely to a deterministic function $g_{ij}(z)$ which is complex analytic on $\mathbb{C} \setminus \mathbb{R}$ as $N \to \infty$ for all $i, j\in [r]$. Let $\pmb{G}(z)$ be the matrix whose $(i, j)$th entry is $g_{ij}(z)$. Then
\begin{equation} \label{eq:measurefpeqmatf_noise}
\pmb{G}(z) =
\pmb{P}_r\begin{pmatrix} \zeta(\kappa_1^{(r)}) & 0 & 0 \\ 0 & \ddots & 0 \\ 0 & 0 &\zeta(\kappa_r^{(r)})\end{pmatrix}\pmb{P}_r^{-1},
\quad \text{\ where \ }
\zeta(\kappa) = \left\{\begin{array}{ll}
\frac{2}{\kappa^2}\left(-z+\sqrt{z^2-\kappa^2}\right) & \text{if $\kappa \neq 0$}\\
\frac{-1}{z} & \text{if $\kappa=0$}
\end{array}\right..
\end{equation}
\end{theorem}
\begin{proof}
See Appendix~\ref{sec:prooftotal}.
\end{proof}

In particular, with this theorem, we can explicitly give the limiting empirical spectral distribution of $\flat(\sT)$.
\begin{theorem}\label{th:explicitmu_noise}
The empirical spectral measure of $\flat(\sT)$ converges weakly almost surely to a deterministic measure $\mu$ whose Stieltjes transform is $g(z) = \frac1r \sum_{i=1}^r g_{ii}(z)$, i.e.,
\begin{equation}\label{eq:measure_limit}
\mu(dx) = \frac{1}{r}\sum_{i=1}^r \mu_{\kappa_i^{(r)}}(dx)\text{\quad with\quad}
\mu_{\kappa}(dx) = \left\{\begin{array}{ll}
\frac{2}{\pi\kappa^2}\sqrt{\bigg(\kappa^2 - x^2\bigg)_+}dx  & \text{if $\kappa\neq 0$}\\
\delta_{0}( dx) & \text{if $\kappa = 0$}
\end{array}\right.,
\end{equation}
which is a measure supported on $[-\lvert\kappa_1^{(r)}\rvert, \lvert\kappa_1^{(r)}\rvert]$.
\end{theorem}

\begin{proof}
See Appendix~\ref{subsec:proofexplicitmu}.
\end{proof}

Because the statement of Theorem~\ref{th:explicitmu_noise} does not depend on the choice of signal-to-noise ratios, as a direct corollary of Theorem~\ref{th:explicitmu_noise}, assuming $\beta_1 = \cdots = \beta_r = 0$ in Model~\ref{eq:tensormodel} gives us the pure noise case, namely,
\begin{corollary}
    For the pure noise $\frac{1}{\sqrt{N}} \sX$, the empirical spectral measure of $\flat(\frac{1}{\sqrt{N}} \sX)$ is given by Equation~\eqref{eq:measure_limit}.
\end{corollary}

In the rank-two case, i.e., when $r=2$, we can provide explicit formulas for the eigenvalue decomposition defined in Equation~\ref{eq:Pdef2}. 
Indeed, if we note $\lambda=\lim_{N \rightarrow\infty}\langle\hat\v_1,\hat\v_2\rangle^{d-1}$ and $\nu = \frac{\gamma_2}{\gamma_1}$, we have the following eigendecomposition
\begin{equation}
\frac{2}{(1-\lambda^2)\sqrt{d(d-1)}}\begin{bmatrix} \nu & - \lambda \\ - \lambda\nu & 1 \end{bmatrix} = \pmb{P}_2\begin{pmatrix} \kappa_1 & 0\\ 0 & \kappa_2\end{pmatrix}\pmb{P}_2^{-1}
\end{equation}
for some $\pmb{P}_2$, where
\begin{equation}\label{eq:kappas}
\begin{dcases}
\kappa_1 = \frac{\nu + 1 + \sqrt{(\nu+1)^2-4\nu(1-\lambda^2)}}{\sqrt{d(d-1)}(1-\lambda^2)} \\
\kappa_2 = \frac{\nu+1 - \sqrt{(\nu+1)^2-4\nu(1-\lambda^2)}}{\sqrt{d(d-1)}(1-\lambda^2)}\end{dcases}.
\end{equation}
Then Theorem~\ref{thm:generalized_measure} can then be rewritten as follows.
\begin{corollary}\label{thm:pure_noise_Stieltjes}
Under Assumption~\ref{assump:general},when $|\gamma_2|\geq (d-1)|\kappa_1|$, $\frac{1}{N} \Tr \Q^{ij}(z)$ converges almost surely to a deterministic function $g_{ij}(z)$ which is complex analytic on $\mathbb{C} \setminus \mathbb{R}$ as $N \to \infty$ for all $i, j\in [2]$. Moreover, 
\begin{equation} 
\pmb{G}(z) =
\pmb{P}_2\begin{pmatrix} \zeta(\kappa_1) & 0  \\  0 &\zeta(\kappa_2)\end{pmatrix}\pmb{P}_2^{-1},
\end{equation}
where $\kappa_1$ and $\kappa_2$ are given in \eqref{eq:kappas}, and $\zeta$ is defined in \eqref{eq:measurefpeqmatf_noise}.
\end{corollary}

Figure~\ref{fig:specmeasure} gives an illustration of the almost sure convergence of the spectrum of $\flat(\sT)$. Indeed, from the proof of Theorem~\ref{thm:generalized_measure}, in particular Subsection~\ref{sec:estimateA1}, it is straightforward to verify that the result still hold for sequence of vectors $(\hat\gamma_1\hat\v_1,\dots,\gamma_r\hat\v_r)$ drawn independently from $\sT$. Therefore, in order to illustrate this result for $r=2$, we generate $100$ random symmetric pure noise tensors $\sT$ and independent vectors $\hat\gamma_1\hat\v_1,\hat\gamma_2\hat\v_2$ satisfying
\[
\pmb{\hat R}_{vv, r} = \begin{bmatrix}1 & 0.5\\ 0.5 & 1\end{bmatrix} \text{\quad and \quad} \frac{\hat\gamma_2}{\hat\gamma_1} = 0.4.
\]
In particular, 
Figure~\ref{fig:sub1} depicts the spectrum of $\flat(\sT)$ of order-$3$ of dimension $200$. In~Figure~\ref{fig:sub2}, we illustrate the spectrum of order-$4$ tensors $\sT$ of dimension $100$.

\begin{figure}[h]
\centering
\begin{subfigure}{.5\textwidth}
  \centering
  \caption{order $ = 3$, dimension $ = 200$}
  \includegraphics[scale=0.15]{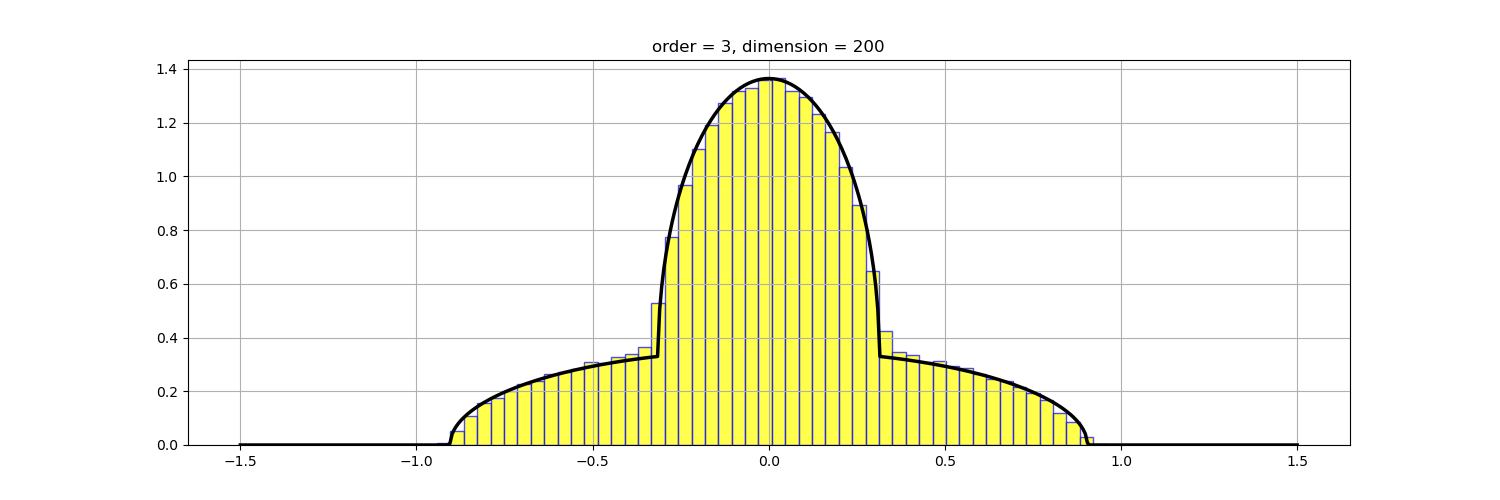}
  \label{fig:sub1}
\end{subfigure}%
\begin{subfigure}{.5\textwidth}
  \centering
 \caption{order $ = 4$, dimension $ = 100$}
 \includegraphics[scale=0.15]{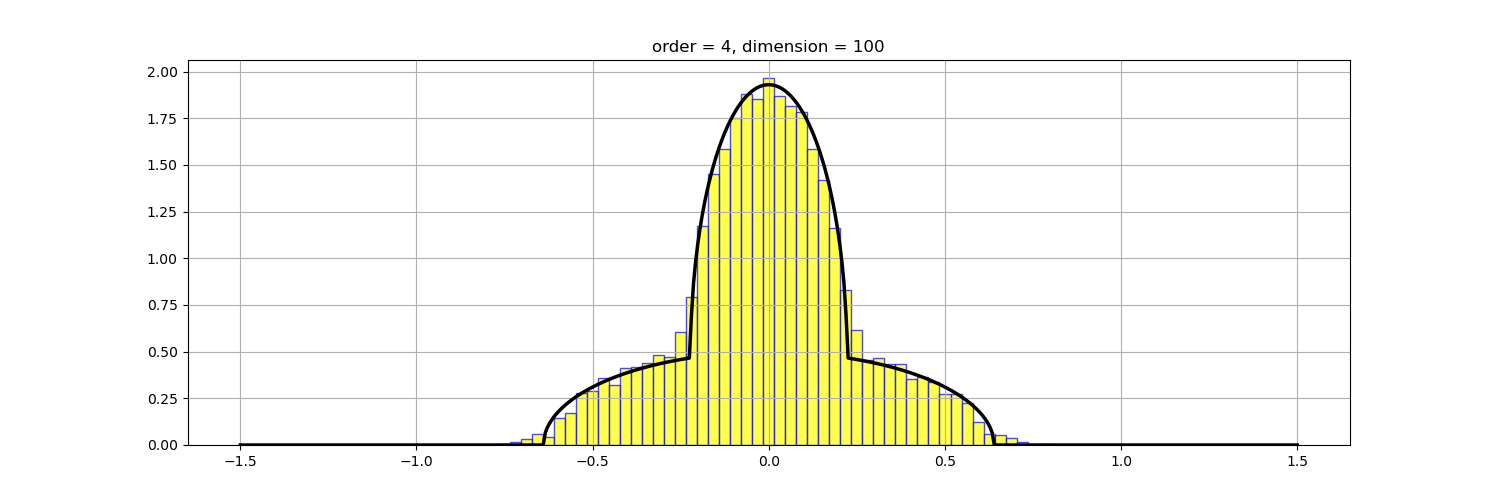}
  \label{fig:sub2}
\end{subfigure}
\caption{Density functions of the spectral measure of Theorem~\ref{thm:generalized_measure} (in black) for $d = 3$ and $d = 4$ and histograms of eigenvalues for one realization of $\flat (\sT)$.}
\label{fig:specmeasure}
\end{figure}

\subsection{Generalization to sequence with partially convergent summary statistics}
We now consider a sequence of critical points such that only the convergence of the summary statistics of $s\leq r$ spikes is ensured while the remaining spikes are assumed to be orthogonal to the first $s$ spikes but without extra convergence assumption. 
\begin{assumption}\label{assump:general0}
The sequence $(\hat\gamma_1, \dots,\hat\gamma_r, \hat\v_1, \dots,\hat\v_r)$ satisfies 
\[
 \hat\gamma_i \xrightarrow{\text{a.s.}} \gamma_i \text{\ for $1\leq i\leq s$}, \quad \pmb{\hat R}_{vv,s} \xrightarrow{\text{a.s.}} \pmb{R}_{vv,s}\in\mathbb{R}^{s\times s}, \quad \langle \hat\v_i,\hat\v_j\rangle =0 \text{ for $1\leq i \leq s < j\leq r$ if $s < r$}.
\]
\end{assumption}
Assumption~\ref{assump:general0} captures several classes of sequence of critical points indexed by some $s\leq r$. In particular, the case $s=r$ corresponds to Assumption~\ref{assump:general}.
Each class of critical points sequences aims at capturing a different possible behavior in terms of convergence which are observed in practice, namely only a fraction of the the estimated signals summary statistics are converging.
Thanks to the orthogonality property of $\hat\v_i$ and  $\hat\v_j$ for $1\leq i \leq s < j\leq r$, we are able to state a convergence result on the partial matrix
\begin{equation}\label{eq:defbsTgeneral2}
\flat_s(\sT) = (\pmb{\hat W}_s\otimes\pId) \begin{bmatrix}
    \langle \sT, \hat\v^{\otimes (d-2)}_1 \rangle & 0  & 0 \\
    0  & \ddots & 0 \\
    0 & 0  & \langle \sT, \hat\v^{\otimes (d-2)}_s \rangle
\end{bmatrix}
\end{equation}
and its resolvent
\begin{equation}
\Q_s(z) = (\flat_s(\sT) -z\pId)^{-1} = \begin{bmatrix}
\Q_s^{11}(z) & \Q_s^{12}(z) & \dots & \Q^{1s}_s(z) \\
\Q_s^{21}(z) & \Q_s^{22}(z) & \dots & \Q^{2s}_s(z) \\
\vdots &\vdots  & \ddots & \vdots \\
\Q^{s1}_s(z) & \Q^{s2}_s(z) & \dots & \Q^{ss}_s(z) \\
\end{bmatrix}.
\end{equation}

We can now state the result generalizing Theorem~\ref{thm:generalized_measure}.
\begin{theorem}\label{thm:mostgeneral}
Define $\pmb{P}_s$ and $\lvert\kappa_1^{(s)}\rvert \geq \dots\geq \lvert\kappa_s^{(s)}\rvert$ defined through the following eigenvalue decomposition
\begin{equation}\label{eq:Pdef2s}
\frac{2}{\sqrt{d(d-1)}}\pmb{W}_{s} = \pmb{P}_s\begin{bmatrix}
    \kappa_1^{(s)}  & 0  & 0 \\
     0 & \ddots & 0 \\
    0 & 0   & \kappa_s^{(s)}\end{bmatrix}\pmb{P}_s^{-1}.
\end{equation}
Under Assumption~\ref{assump:general0}, when
\begin{equation}
|\gamma_s| \geq (d-1)\,\lvert\kappa_1^{(s)}\rvert,
\end{equation}
the quantity $\frac{1}{N} \Tr \Q_s^{ij}(z)$ converges almost surely to a deterministic function $g_{ij,s}(z)$ which is complex analytic on $\mathbb{C} \setminus \mathbb{R}$ as $N \to \infty$ for all $i, j\in [s]$.
Let $\pmb{G}_s(z)$ be the matrix whose $(i, j)$th entry is $g_{ij}(z)$. Then, for $\zeta$  defined in \eqref{eq:measurefpeqmatf_noise},
\begin{equation} \label{eq:measurefpeqmatf_noise_s}
\pmb{G}_s(z) =
\pmb{P}_s\begin{pmatrix} \zeta(\kappa_1^{(s)}) & 0 & 0 \\ 0 & \ddots & 0 \\ 0 & 0 &\zeta(\kappa_s^{(s)})\end{pmatrix}\pmb{P}_s^{-1},
\end{equation}
\end{theorem}

\begin{proof}
See Appendix~\ref{subsec:proofmeasure_orth}.
\end{proof}

In particular, taking $s = r$ in Theorem~\ref{thm:mostgeneral} gives us the rank-$r$ case, i.e., Theorem~\ref{thm:generalized_measure}. When $s=1$, we recover the limiting spectrum for the rank-one case, which is given by~\cite[Theorem~2]{GCC22}.
\begin{corollary}\label{th:measure_orth}
Assume Assumption~\ref{assump:general0} holds with $|\gamma_1|\geq \frac{2(d-1)}{\sqrt{d(d-1)}}$, the empirical spectral measure of $\langle \sT, \hat\v_1^{\otimes (d-2)}\rangle$ converges weakly almost surely to a deterministic measure whose Stieltjes transform is defined by
\begin{equation}\label{eq:stieltjes1}
g_1(z) = \frac{d(d-1)}{2}\left(-z+\sqrt{z^2-\frac{4}{d(d-1)}}\right).
\end{equation}
%\begin{equation}
%\mu_1(dx) = \frac{d(d-1)}{2\pi} \sqrt{\bigg(\frac{4}{d(d-1)} - x^2\bigg)_+} \, dx,
%\end{equation}
\end{corollary}
\begin{proof}
It is a direct application of Theorem~\ref{thm:mostgeneral}.
\end{proof}

Even though the orthogonality property of $\hat\v_i$ and  $\hat\v_j$ for $1\leq i\leq s < j\leq r$ in Assumption~\ref{assump:general0} may not be satisfied in practice, we conjecture that Theorem~\ref{thm:mostgeneral} remains true when only assuming that $\langle \hat\v_i,\hat\v_j\rangle$ converges to $0$ when $N$ goes to infinity.

\begin{remark}\label{rem:spectrum}
Let us define
\[
\pmb{U}_i = \begin{bmatrix}
\u_i & 0 & 0 \\
0 & \ddots & 0\\
0 & 0 & \u_i
\end{bmatrix} \text{\quad and \quad} \A_i = \beta_i \pmb{\hat W}_r \begin{bmatrix}
\langle \u_i, \hat\v_1 \rangle^{d-2} & 0 & 0 \\
0 & \ddots & 0\\
0 & 0 & \langle \u_i, \hat\v_r\rangle^{d-2}
\end{bmatrix}.
\]
Then
\begin{equation}\label{eq:rk2deformation}
\flat(\sT) = \flat(\sum_{j=1}^r \beta_j \u_j^{\otimes d}) + \frac{1}{\sqrt{N}}\flat(\sX) = \sum_{i=1}^r \pmb{U}_i \A_i \pmb{U}_i^{\top} + \frac{1}{\sqrt{N}}\flat(\sX).
\end{equation}
Although we have seen that the empirical spectral measures of $\flat(\sT)$ and $\frac{1}{\sqrt{N}}\flat(\sX)$ converge to the same limit, thanks to the rank-$r^2$ matrix $\pmb{U}_i \A_i \pmb{U}_i^{\top}$, the largest eigenvalues of $\flat(\sT)$ and $\frac{1}{\sqrt{N}}\flat(\sX)$ may be different, which can help us distinguish the two matrices, hence providing another criterion to detect the low-rank signal.
\end{remark}
% Figure~\ref{fig:specmeasure} gives an illustration of the almost sure convergence of the spectrum of $\flat(\sT)$. In particular, in Figure~\ref{fig:sub1}, we randomly generate $100$ random order-$3$ tensors $\sT$ and $100$ unit vectors $\v_1$ and $\v_2$ of dimension $200$, such that $\langle \v_1, \v_2, \rangle = 0.5$ and $\nu = 0.4$. Figure~\ref{fig:sub1} depicts the spectrum of $\flat(\sT)$. In~Figure~\ref{fig:sub2}, we randomly generate $100$ random order-$4$ tensors $\sT$ and $100$ unit vectors $\v_1$ and $\v_2$ of dimension $100$, such that $\langle \v_1, \v_2, \rangle = 0.5$ and $\nu = 0.4$. Figure~\ref{fig:sub2} depicts the spectrum of $\flat(\sT)$.

% \begin{figure}[h]
% \centering
% \begin{subfigure}{.5\textwidth}
%   \centering
%   \caption{order $ = 3$, dimension $ = 200$}
%   \includegraphics[scale=0.15]{spec_measure_order3}
%   \label{fig:sub1}
% \end{subfigure}%
% \begin{subfigure}{.5\textwidth}
%   \centering
%  \caption{order $ = 4$, dimension $ = 100$}
%  \includegraphics[scale=0.15]{spec_measure_order4}
%   \label{fig:sub2}
% \end{subfigure}
% \caption{Density functions of the spectral measure of Theorem~\ref{thm:lowrank_Stieltjes} (in black) for $d = 3$ and $d = 4$ and histograms of eigenvalues for one realization of $\flat (\sT)$.}
% \label{fig:specmeasure}
% \end{figure}

\section{Limiting alignments}\label{sec:limitingalignments}
\subsection{Asymptotic results}

In this section, we study the limiting alignments for Model~\eqref{eq:tensormodel} by using the results on the spectrum of $\flat(\sT)$ obtained in Section~\ref{sec:spectrum}. 
In order to state our results, we will assume the convergence of the following correlation matrices related to the true signals
\begin{equation}
\pmb{\hat R}_{uv,s} = \begin{bmatrix}\langle \u_1,\hat\v_1\rangle & \dots & \langle \u_1,\hat\v_s\rangle \\ \vdots & \ddots & \vdots\\ \langle \u_r,\hat\v_1\rangle & \dots & \langle \u_r,\hat\v_s\rangle\end{bmatrix}
\text{\quad and \quad}
\pmb{\hat R}_{uu,r} = \begin{bmatrix}1 & \dots & \langle \u_1,\u_r\rangle \\ \vdots & \ddots & \vdots\\ \langle \u_r,\u_1\rangle & \dots & 1\end{bmatrix}.
\end{equation}
\begin{assumption}\label{assump:general2}
The sequence $(\hat\gamma_1, \dots,\hat\gamma_s, \hat\v_1, \dots,\hat\v_s)$ satisfies 
\[
 \pmb{\hat R}_{uv,s}\xrightarrow{\text{a.s.}} \pmb{R}_{uv,s}\in\mathbb{R}^{r\times s}, \quad \pmb{\hat R}_{uu,r} \xrightarrow{\text{a.s.}} \pmb{R}_{uu,r}\in\mathbb{R}^{r\times r}.
\]
\end{assumption}

Denote furthermore
\[
\pmb{D}_{\beta} = \begin{bmatrix}
    \beta_1  & 0 & 0 \\
    0 & \ddots  & 0  \\
   0 & 0 & \beta_r \\
\end{bmatrix}
\text{\ and \ }
\pmb{D}_{\gamma} = \begin{bmatrix}
    \gamma_1  & 0 & 0 \\
    0 & \ddots  & 0  \\
   0 & 0 & \gamma_s \\
\end{bmatrix}.
\]
Then we have the following result governing the asymptotic relationships between the critical points and the true signals, more precisely, between the $s$ estimated signals whose summary statistics are assumed to \textit{converge} and the $r$ true signals. 

\begin{theorem}\label{thm:generalized_align}
Under Assumption~\ref{assump:general0}~and~\ref{assump:general2}, when
\begin{equation}\label{eq:bound_gamma}
    |\gamma_s| > (d-1) \, \lvert\kappa_1^{(s)}\rvert,
\end{equation}
  the limiting alignments satisfy the following equations
\begin{equation}\label{eq:bigsystems}
\begin{dcases}
\pmb{R}_{uu,r}\pmb{D}_{\beta} \pmb{R}_{uv,s}^{\odot (d-1)} &= \pmb{R}_{uv,s} \pmb{M} \\
%\pmb{R}_{uv}^{\top}\pmb{D}_{\beta} \pmb{R}_{uv}^{\odot (d-1)}  &= \pmb{R}_{vv}\pmb{M} + \frac{\gamma_s}{d(d-1)}\pmb{G}\left(\frac{\gamma_s}{d-1}\right)^{\top}\pmb{D}_ {\gamma}^{-1}
\pmb{R}_{uv,s}^{\top}\pmb{D}_{\beta} \pmb{R}_{uv,s}^{\odot (d-1)}  &= \pmb{R}_{vv,s}\pmb{M} + \frac{1}{d(d-1)}\left(\pmb{R}_{vv,s}^{\odot (d-1)}\pmb{G}_s\left(\frac{\gamma_s}{d-1}\right)\pmb{W}_s\right)^{\top}
\end{dcases},
\end{equation}
where
\begin{equation}
    \pmb{M} =  \pmb{D}_{\gamma}\pmb{R}_{vv,s}^{\odot (d-1)}+  \frac{1}{d}\pmb{R}_{vv,s}^{\odot (d-2)}\odot\left(\pmb{G}_s\left(\frac{\gamma_s}{d-1}\right)\pmb{W}_s\right).
\end{equation}
\end{theorem}
\begin{proof}
See Appendix~\ref{sec:generalized_align}.
\end{proof}

In particular, taking $s = r$ in Theorem~\ref{thm:generalized_align} gives us the limiting alignments for the rank-$r$ case. As another direct corollary, we apply Theorem~\ref{thm:generalized_align} to the case $s=1,r=2$.
\begin{corollary}\label{thm:orthogonal}
Under Assumption~\ref{assump:general0} and Assumption~\ref{assump:general2}, when
\begin{equation}\label{eq:bound_gamma1}
    \gamma_1 > \frac{2(d-1)}{\sqrt{d(d-1)}},
\end{equation}
denoting by 
\[
\pmb{R}_{uv,1}=\begin{bmatrix}\alpha_{11}\\ \alpha_{21}\end{bmatrix} \text{\quad  and \quad} \pmb{R}_{uu,2} = \begin{bmatrix}1 & \rho\\ \rho & 1\end{bmatrix},
\]
then $\gamma_1, \alpha_{11}, \alpha_{21}$ satisfy the following equations
%\begingroup\makeatletter\def\f@size{9}\check@mathfonts
%\def\maketag@@@#1{\hbox{\m@th\large\normalfont#1}}%
\begin{equation}\label{eq:bigsystem_orth}
\begin{dcases}
    \beta_1 \alpha_{11}^{d-1} + \beta_2 \alpha_{21}^{d-1}\rho &= \bigg(\gamma_1 + \frac{1}{d} g_1(\frac{\gamma_1}{d-1}) \bigg) \alpha_{11}\\
    \beta_1 \rho\alpha_{11}^{d-1} + \beta_2 \alpha_{21}^{d-1} &= \bigg(\gamma_1 + \frac{1}{d}g_1(\frac{\gamma_1}{d-1}) \bigg) \alpha_{21}\\
    \beta_1\alpha_{11}^d + \beta_2\alpha_{21}^d &= \gamma_1 + \frac{1}{d-1} g_1(\frac{\gamma_1}{d-1}) \\
\end{dcases},
\end{equation}
%\endgroup
where $g_1$ given in \eqref{eq:stieltjes1}.
\end{corollary}
\begin{proof}
This is a direct application of Theorem~\ref{thm:generalized_align}.
\end{proof}

Note that the lower bound given by~\eqref{eq:bound_gamma1} coincides with the bound for $\gamma_1$ in the rank-one case. See~\cite[Theorem~3]{GCC22}. Moreover, we remark that the limiting alignments in this case correspond to the limiting alignments obtained by the first step of the tensor deflation method presented in \cite[Corollary~4.3]{SGG23}.

%%%%%%%%

As another explicit example, let us take $s = r = 2$, i.e., the rank-$2$ case. Denoting by 
\[
\nu = \frac{\gamma_1}{\gamma_1}, \quad\pmb{R}_{vv,2} = \begin{bmatrix}1 & \tau\\ \tau & 1\end{bmatrix}, \quad \lambda = \tau^{d-1}, \quad
\pmb{R}_{uv,2}=\begin{bmatrix}\alpha_{11} & \alpha_{12}\\ \alpha_{21} & \alpha_{22}\end{bmatrix}, \quad \text{and} \quad \pmb{R}_{uu,2} = \begin{bmatrix}1 & \rho\\ \rho & 1\end{bmatrix},
\]
Equation~\ref{eq:bigsystems} has the form
\begin{equation}\label{eq:bigsystem2}
\begin{dcases}
\begin{bmatrix} 1 & \rho \\ \rho & 1\end{bmatrix}\begin{bmatrix} \beta_1 & 0 \\ 0 & \beta_2\end{bmatrix} \begin{bmatrix} \alpha_{11}^{d-1} & \alpha_{12}^{d-1} \\ \alpha_{21}^{d-1} & \alpha_{22}^{d-1} \end{bmatrix} &= \begin{bmatrix} \alpha_{11} & \alpha_{12} \\ \alpha_{21} & \alpha_{22} \end{bmatrix}\pmb{M} \\
\begin{bmatrix} \alpha_{11} & \alpha_{12} \\ \alpha_{21} & \alpha_{22} \end{bmatrix}^{\top}\begin{bmatrix} \beta_1 & 0 \\ 0 & \beta_2\end{bmatrix}\begin{bmatrix} \alpha_{11}^{d-1} & \alpha_{12}^{d-1} \\ \alpha_{21}^{d-1} & \alpha_{22}^{d-1} \end{bmatrix} &= \begin{bmatrix} 1 & \tau \\ \tau & 1 \end{bmatrix}\pmb{M} + \frac{1}{d(d-1)}\pmb{G}\left(\frac{\gamma_2}{d-1}\right)\begin{bmatrix} \nu & 0 \\ 0 & 1 \end{bmatrix} 
\end{dcases},
\end{equation}
with $\pmb{G}(z) = \begin{bmatrix} g_{11}(z) & g_{12}(z) \\ g_{21}(z) & g_{22}(z) \end{bmatrix}$ with $g_{ij}(z)$ defined in Corollary~\ref{thm:pure_noise_Stieltjes}, and
\begin{multline}
    \pmb{M} =  \begin{bmatrix} \gamma_1 & \gamma_1\lambda \\ \gamma_2\lambda & \gamma_2 \end{bmatrix} \\
    + \frac{1}{d(1-\lambda^2)}\begin{bmatrix} \nu g_{11}\left(\frac{\gamma_2}{d-1}\right)-\lambda g_{12}\left(\frac{\gamma_2}{d-1}\right) & \tau^{d-2}\bigg(\nu g_{21}\left(\frac{\gamma_2}{d-1}\right)-\lambda g_{22}\left(\frac{\gamma_2}{d-1}\right)\bigg) \\ \tau^{d-2}\bigg(g_{12}\left(\frac{\gamma_2}{d-1}\right)-\lambda \nu g_{11}\left(\frac{\gamma_2}{d-1}\right)\bigg) & g_{22}\left(\frac{\gamma_2}{d-1}\right) - \lambda\nu g_{21}\left(\frac{\gamma_2}{d-1}\right)\end{bmatrix}.
\end{multline}

\subsection{Numerical experiments}

In order to compute critical points, we use the following iterative algorithm.
Given an observed $\sT$, we initialize randomly the point $(\hat\gamma_1^{(0)}\hat\v_1^{(0)},\dots,\hat\gamma_r^{(0)}\hat\v_r^{(0)})$ and iteratively update the vector using \eqref{eq:kktrkrmatrix} by first updating $\flat(\sT)$ and then performing
\begin{equation}
\begin{bmatrix} \hat\gamma_1^{(t+1)}\hat\v_1^{(t+1)}\\ \vdots \\ \hat\gamma_r^{(t+1)}\hat\v_r^{(t+1)}\end{bmatrix}  = \frac{1}{\hat\gamma_r^{(t)}}b(\sT)\begin{bmatrix} \hat\gamma_1^{(t)}\hat\v_1^{(t)}\\ \vdots \\ \hat\gamma_r^{(t)}\hat\v_r^{(t)}\end{bmatrix}
\end{equation}
until convergence, which will provide us empirical values $\pmb{\hat R}_{uv,r}, \pmb{\hat R}_{vv,r}, \hat{\gamma}_1, \dots, \hat{\gamma}_r$.

Secondly, in order to compare with our theoretical asymptotic values $\pmb{R}_{uv, r}, \pmb{R}_{vv, r}, \gamma_1, \dots, \gamma_r$, we numerically solve \eqref{eq:bigsystems} to find the asymptotic parameters in particular for the following two classes of sequences of critical points.
\begin{itemize}
\item First, we consider Assumption~\ref{assump:general0} with $r=s=2$ and solve \eqref{eq:bigsystems} with these parameters. We call this solution the \textbf{2-signal solution}.
\item Second, we consider Assumption~\ref{assump:general0} with $s=1,r=2$ and solve \eqref{eq:bigsystems} with these parameters. We call this solution the \textbf{1-signal solution}. Note that in this case, we only obtain the limits $\gamma_1$ and $\pmb{R}_{uv,1}$.
\end{itemize}
We simulate a rank-two model with two true signals satisfying $\beta_1=\beta_2$ and a correlation
\[
\pmb{R}_{uu,2} = \begin{bmatrix}1 & \rho\\ \rho & 1\end {bmatrix}.
\]

In Figure~\ref{fig:rho0nu1} we simulate the orthogonal case $\rho=0$. In this special case,  we observe that there exist three different 1-signal solutions corresponding to one of the three following asymptotic cases:
\begin{itemize}
    \item the estimated signal $\hat\v_1$ is strongly correlated with $\u_1$ and independent from $\u_2$,
    \item the estimated signal $\hat\v_1$ is a mixture of the two signals $\u_1$ and  $\u_2$,
    \item the estimated signal $\hat\v_1$ is strongly correlated with $\u_2$ and independent from $\u_1$.
\end{itemize}
We observe that the alignments of the 2-signal solution and the first 1-signal solution coincide when $\beta_2$ is sufficiently large, even though the $1$-signal solution exists for smaller $\beta_2$. This makes sense because Assumption~\ref{assump:general0} is weaker when $s=1$. 
We also observe in Figure~\ref{fig:rho0nu1} the coincidence between the empirical quantities and their asymptotic theoretical computations.

In Figure~\ref{fig:rho0nu07}, we consider the correlated case $\rho=0.7$. In this case, the two classes of alignments show different behaviors. We see in this case that some empirical critical points satisfy the first class while the others belong to the second class of alignments. 

\begin{figure}[h!]
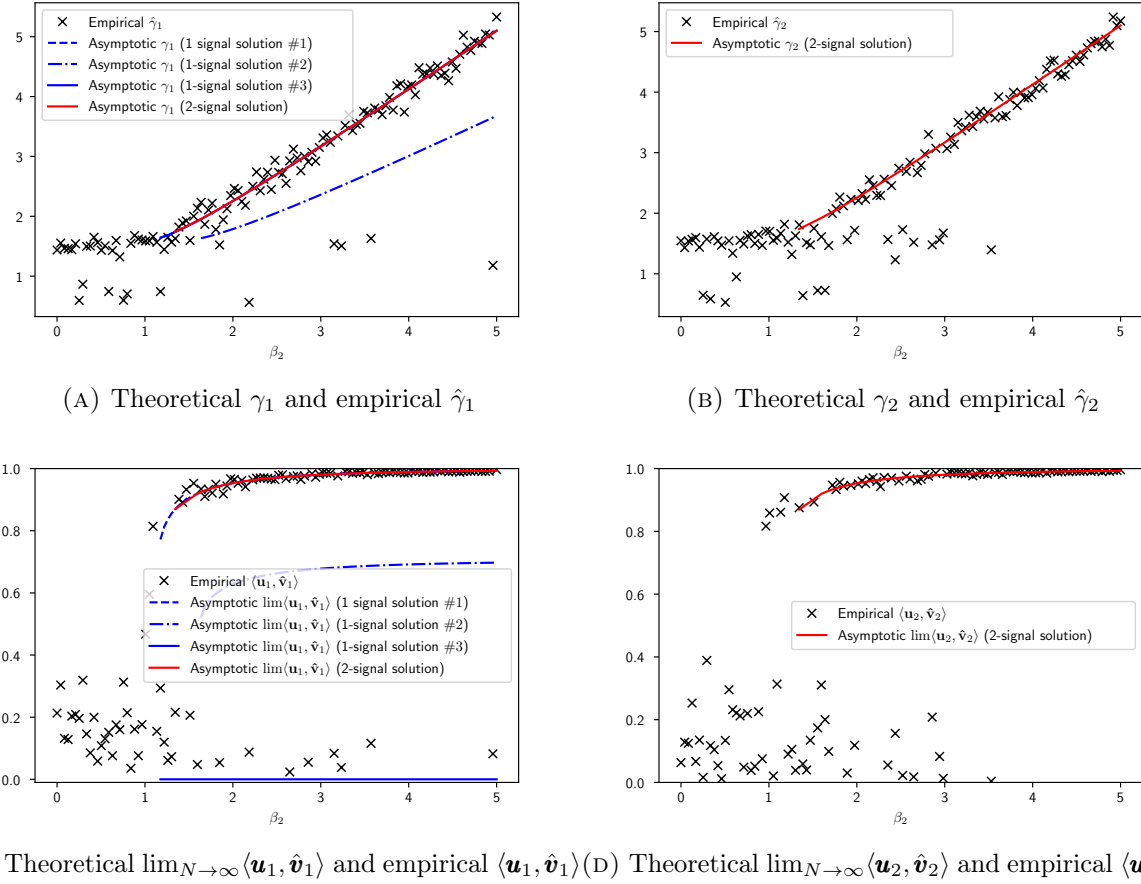

\centering
\begin{subfigure}{.5\textwidth}
  \centering
\resizebox{\columnwidth}{!}{\input{images/gamma1_rho0_ratio1.pgf}}
\caption{Theoretical $\gamma_{1}$ and empirical $\hat\gamma_{1}$}
\end{subfigure}%
\begin{subfigure}{.5\textwidth}
  \centering
\resizebox{\columnwidth}{!}{\input{images/gamma2_rho0_ratio1.pgf}}
\caption{Theoretical $\gamma_{2}$ and empirical $\hat\gamma_{2}$}
\end{subfigure}
\begin{subfigure}{.5\textwidth}
  \centering
\resizebox{\columnwidth}{!}{\input{images/alpha11_rho0_ratio1.pgf}}
\caption{Theoretical $\lim_{N\rightarrow\infty}\langle \u_1,\hat\v_1\rangle$ and empirical $\langle \u_1,\hat\v_1\rangle$}
\end{subfigure}%
\begin{subfigure}{.5\textwidth}
  \centering
\resizebox{\columnwidth}{!}{\input{images/alpha22_rho0_ratio1.pgf}}
\caption{Theoretical $\lim_{N\rightarrow\infty}\langle \u_2,\hat\v_2\rangle$ and empirical $\langle \u_2,\hat\v_2\rangle$}
\end{subfigure}
\caption{Solutions with $\rho=0$}\label{fig:rho0nu1}
\end{figure}

\begin{figure}[h!]
\centering
\begin{subfigure}{.5\textwidth}
  \centering
\resizebox{\columnwidth}{!}{\input{images/gamma1_rho0.7_ratio1.pgf}}
\caption{Theoretical $\gamma_{1}$ and empirical $\hat\gamma_{1}$}
\end{subfigure}%
\begin{subfigure}{.5\textwidth}
  \centering
\resizebox{\columnwidth}{!}{\input{images/gamma2_rho0.7_ratio1.pgf}}
\caption{Theoretical $\gamma_{2}$ and empirical $\hat\gamma_{2}$}
\end{subfigure}
\begin{subfigure}{.5\textwidth}
  \centering
\resizebox{\columnwidth}{!}{\input{images/alpha11_rho0.7_ratio1.pgf}}
\caption{Theoretical and empirical $\langle \u_1,\hat\v_1\rangle$}
\end{subfigure}%
\begin{subfigure}{.5\textwidth}
  \centering
\resizebox{\columnwidth}{!}{\input{images/alpha22_rho0.7_ratio1.pgf}}
\caption{Theoretical and empirical $\langle \u_2,\hat\v_2\rangle$}
\end{subfigure}
\caption{Solutions with $\rho=0.7$}\label{fig:rho0nu07}
\end{figure}

Remark that, in general, it is possible that there may be multiple algebraic solutions to equations~\eqref{eq:bigsystems}. All of them  may correspond to limits of critical points of Equation~\eqref{eq:rkrapprox}. 
We can pick the solution corresponding to the minimizer of~\eqref{lowrkopt} among the class of critical points corresponding to $s=r$ obtained from~\eqref{eq:bigsystems}. For this aim, in addition to the semialgebraic conditions imposed by the Stieltjes transform, we also need the solution $\pmb{R}_{uv, r}, \pmb{R}_{vv, r}, \gamma_1, \dots, \gamma_r$ minimizing the asymptotic squared error since we can express this limit squared error in terms of the signal summary statistics when $s=r$.

Indeed, for a sequence of critical points $(\hat\gamma_1, \dots,\hat\gamma_r, \hat\v_1, \dots,\hat\v_r)$ (i.e., satisfying~\eqref{eq:rkrapprox}), we define the normalized squared error by
\[
\begin{split}
\mathcal{H}(\hat\gamma_1, \dots,\hat\gamma_r, \hat\v_1, \dots,\hat\v_r) &\coloneqq \bigg\lVert \sT - \sum_{i=1}^r \hat\gamma_i \hat\v_i^{\otimes d} \bigg\rVert - \frac{1}{N} \lVert \sX \rVert^2.
\end{split}
\]
Note that we need to define a normalized version of the squared error since the term $\frac{1}{N} \lVert \sX \rVert^2$ will dominate the error for large $N$. Then, the following result holds.
\begin{lemma}\label{lemma:likelihood}
Under Assumption~\ref{assump:general}~and~\ref{assump:general2}, we have
\[
\lim_{N\to\infty} \mathcal{H}(\hat\gamma_1, \dots,\hat\gamma_r, \hat\v_1, \dots,\hat\v_r) =  \begin{bmatrix}\beta_1\\ \vdots\\ \beta_r\end{bmatrix}^{\top}\pmb{R}_{uu,r}^{\odot d}\begin{bmatrix}\beta_1\\ \vdots\\ \beta_r\end{bmatrix} -  \begin{bmatrix}\gamma_1\\ \vdots\\ \gamma_r\end{bmatrix}^{\top}\pmb{R}_{vv,r}^{\odot d}\begin{bmatrix}\gamma_1\\ \vdots\\ \gamma_r\end{bmatrix}.
\]
\end{lemma}
\begin{proof}
See Appendix~\ref{sec:likelihood}.
\end{proof}

\subsection{Direct threshold}

As a corollary of Theorem~\ref{thm:generalized_align}, we would like to establish a threshold above which the detection of critical points of the optimization problem~\eqref{lowrkopt} is possible. To this end, we first use the arguments of Remark~\ref{rem:spectrum} to prove the following result
\begin{corollary}\label{thm:alignmentbeta2}
There exists some $\beta_{\operatorname{cri}}(\pmb{R}_{uu,r}) > 0$ which depends on $\pmb{R}_{uu,r}$ such that when\[
\lvert\beta_r \rvert > \beta_{\operatorname{cri}}(\pmb{R}_{uu,r}),
\]
any sequence of critical points satisfying Assumption~\ref{assump:general0} and Assumption~\ref{assump:general2}, where $\hat{\gamma}_r$ is the largest eigenvalue of $\flat(\sT)$, also satisfies the limiting alignments in Equation~\eqref{eq:bigsystems}.
\end{corollary}
\begin{proof}
See Appendix~\ref{subsec:alignmentbeta}.
\end{proof}

As a complement to Corollary~\ref{thm:alignmentbeta2}, in numerical experiments we also observe the existence of some bound $\beta_{*}(\pmb{R}_{uu,s}) > 0$ which depends on $\pmb{R}_{uu,s}$ such that whenever
\[
\lvert\beta_1 \rvert \leq \beta_{*}(\pmb{R}_{uu,s}),
\]
there is no solutions to~\eqref{eq:bigsystems}, namely there does not exist any sequence of critical points satisfying Assumption~\ref{assump:general0}, Assumption~\ref{assump:general2} and \eqref{eq:bound_gamma} simultaneously. Combined with Corollary~\ref{thm:alignmentbeta2}, it indicates a phase transition phenomenon happening in the detection of the true signals through tensor decomposition \eqref{lowrkopt}, i.e., the existence of certain bound above which the detection of critical points of the optimization problem~\eqref{lowrkopt} is possible and below which such detections are not possible, which is visible in Figure~\ref{fig:rho0nu1} and Figure~\ref{fig:rho0nu07}.

% On the other hand, we can also state the following result.
% \begin{corollary}\label{cor:alignmentbeta}
% There exists some $\beta_{*}(\pmb{R}_{uu,s})$ which depends on $\pmb{R}_{uu,s}$ such that when\[
% \lvert\beta_1 \rvert \leq \beta_{*}(\pmb{R}_{uu,s}),
% \]
% there does not exist any sequence of critical points satisfying Assumption~\ref{assump:general0}, Assumption~\ref{assump:general2} and \eqref{eq:bound_gamma}.
% \end{corollary}
% \begin{proof}
% Let us assume that for any $\beta_1,\dots,\beta_r$, there exist a sequence of critical points satisfying Assumption~\ref{assump:general0}, Assumption~\ref{assump:general2} and \eqref{eq:bound_gamma} and consider such a sequence when $\beta_1=\dots=\beta_r=0$. 
% Then, Theorem~\ref{thm:generalized_align} guarantees that Equation~\eqref{eq:bigsystems} is satisfied.
% However, we check that \eqref{eq:bigsystems} admits no solution in this case which contradicts our assumption.
% \end{proof}

\section{Summary statistics inference from estimated signals}\label{sec:inference}

For a given received tensor $\mathcal{T}$ we can compute a critical point of the rank-$r$ approximation loss, i.e., the squared error in \eqref{lowrkopt}. Then, if $(\hat\gamma_i \hat\v_i)_i$ corresponds to an optimal point of the likelihood, it is the maximum likelihood estimator of $(\beta_i \u_i)_i$. In previous sections we have seen the performance of the convergent critical points summary statistics. In particular, we saw that $\hat\gamma_i$ is an asymptotically biased estimator of $\beta_i$ since  $\beta_i\neq\gamma_i$ in general.
Given any $N$, for an observation $\sT$, instead of considering $\hat\gamma_1, \dots,\hat\gamma_r, \hat\v_1, \dots, \hat\v_r$ as estimators for $\beta_1, \dots,\beta_r, \u_1, \dots,\u_r$, Theorem~\ref{thm:generalized_align} gives us a way to obtain asymptotically unbiased estimators by viewing the equations of System~\eqref{eq:bigsystems} as equations in $\beta_1, \dots,\beta_r, \pmb{R}_{uu,r},\pmb{R}_{uv,r}$.
Indeed, we can design a plug-in estimator of these quantities, namely, by solving this polynomial system of equations~\eqref{eq:bigsystems} in $\pmb{\hat\beta}, \pmb{R}_{uu,r},\pmb{R}_{uv,r}$ with $\pmb{\hat\beta}=(\beta_1,\dots,\beta_r)$.

To this end, let us define
\begin{equation}\label{eq:changeofvarCD}
\begin{array}{l}
%    \pmb{\hat M} =  \begin{bmatrix} \hat\gamma_1 & \hat\gamma_1\hat\lambda \\ \hat\gamma_2\hat\lambda & \hat\gamma_2 \end{bmatrix} +  \frac{1}{d(1-\hat\lambda^2)}\begin{bmatrix} 1 & \hat\tau^{d-2} \\ \hat\tau^{d-2} & 1 \end{bmatrix}\odot\left(\pmb{\hat G}\begin{bmatrix} \hat\nu &-\hat\lambda\\ -\hat\lambda\hat\nu& 1\end{bmatrix}\right)\\
%\pmb{\hat C} = \begin{bmatrix} 1 & \hat\tau \\ \hat\tau & 1 \end{bmatrix}\pmb{\hat M} + \frac{1}{d(d-1)} \pmb{\hat G}^{\top}\begin{bmatrix} \hat\nu & 0 \\ 0 & 1 \end{bmatrix}.
    \pmb{\hat M} =  \pmb{D}_{\hat\gamma}\pmb{\hat R}_{vv,s}^{\odot (d-1)}+  \frac{1}{d}\pmb{\hat R}_{vv,s}^{\odot (d-2)}\odot\left(\pmb{\hat G}\pmb{\hat W}_s\right)\\
    \pmb{\hat C} =  \pmb{\hat R}_{vv,s}\pmb{\hat M} + \frac{1}{d(d-1)}\left(\pmb{\hat R}_{vv,s}^{\odot (d-1)}\pmb{\hat G}\pmb{\hat W}_r\right)^{\top}
\end{array}
\end{equation}
where 
\[
\hat{\pmb{G}} = \hat{\pmb{G}}_r\left(\frac{\hat\gamma_r}{d-1}\right)
\]
is the plugin estimator whose expression is given by~\eqref{eq:measurefpeqmatf_noise}.
Note that $\pmb{\hat M}$ and $\pmb{\hat C}$ depend on the estimated signals $\hat\gamma_i\hat\v_i$, i.e., they are completely determined by the computed critical point. 

Then, let us denote by $\mathcal{S}$ the set of potential plug-in estimates $(\pmb{\hat\beta},\pmb{\hat R}_{uv,r},\pmb{\hat R}_{uu,r})$ satisfying
\begin{equation}\label{eq:system_poly}
\begin{dcases}
\pmb{\hat R}_{uu,r}\pmb{D}_{\pmb{\hat\beta}} \pmb{\hat  R}_{uv,s}^{\odot (d-1)} &= \pmb{\hat  R}_{uv,s} \pmb{\hat  M} \\
\pmb{\hat  R}_{uv,s}^{\top}\pmb{D}_{\pmb{\hat\beta}} \pmb{\hat R}_{uv,s}^{\odot (d-1)}  &= \pmb{\hat C}
\end{dcases}.
\end{equation}
Explicit solutions of \eqref{eq:system_poly} cannot be provided in general. However, in the rank-two case, we provide a characterization of the solutions in $\mathcal{S}$ in Appendix~\ref{sec:proofmlecorrection}.
Since $\mathcal{S}$ may contain multiple solutions, we pick the solutions maximizing the asymptotic likelihood. More explicitly, Lemma~\ref{lemma:likelihood} finally inspires us to define an estimator for $\pmb{\beta}$ as
\[
(\pmb{\hat\beta},\pmb{\hat R}_{uv,r},\pmb{\hat R}_{uu,r}) \in \arg\max_{(\pmb{\beta}, \pmb{R}_{uu,r}, \pmb{R}_{uv,r})\in\mathcal{S}} \begin{bmatrix}\beta_1\\ \vdots\\ \beta_r\end{bmatrix}^{\top}\pmb{R}_{uu,r}^{\odot d}\begin{bmatrix}\beta_1\\ \vdots\\ \beta_r\end{bmatrix}.
\]

Note that it is possible that $\mathcal{S}$ contains no solution. In this case, no refined estimator can be provided. This intuitively means that the found critical point is not informative enough.
In Figure~\ref{fig:beta_estimated} we depict the results of this process in the rank-two case where we picked different values for $\beta_1$ and $\beta_2$ and we fixed
\[
\pmb{R}_{uu,2} = \begin{bmatrix}1 & \rho\\ \rho & 1\end{bmatrix} = \begin{bmatrix}1 & 0.7\\ 0.7 & 1\end{bmatrix}.
\]

\begin{figure}[h!]
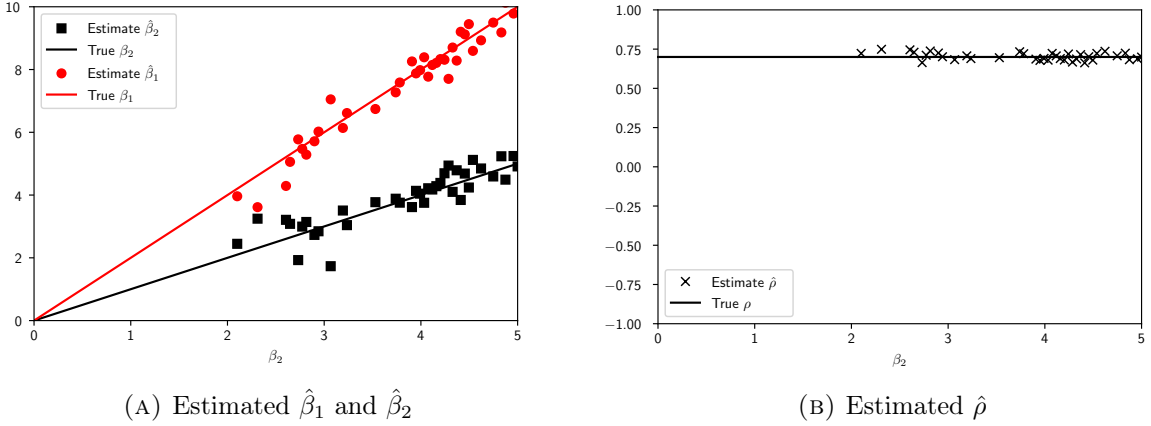

\centering
%\begin{subfigure}{.5\textwidth}
%  \centering
%\resizebox{\columnwidth}{!}{\input{images/gamma1_rho0.7_ratio0.5.pgf}}
%\caption{$\gamma_1$}
%\end{subfigure}%
%\begin{subfigure}{.5\textwidth}
%  \centering
%\resizebox{\columnwidth}{!}{\input{images/gamma2_rho0.7_ratio0.5.pgf}}
%\caption{$\gamma_2$}
%\end{subfigure}
\begin{subfigure}{.5\textwidth}
  \centering
\resizebox{\columnwidth}{!}{\input{images/betas0.7_ratio0.5.pgf}}
\caption{Estimated $\hat\beta_1$ and $\hat\beta_2$}
\end{subfigure}%
\begin{subfigure}{.5\textwidth}
  \centering
\resizebox{\columnwidth}{!}{\input{images/rho0.7_ratio0.5.pgf}}
\caption{Estimated $\hat\rho$}
\end{subfigure}
\caption{Solutions with $\beta_1=2\beta_2$ and $\rho=0.7$}\label{fig:beta_estimated}
\end{figure}

%\begin{remark}
%For simplicity, let us denote Equations~\eqref{eq:bigsystem2} by $\Psi(\underline{Y}, \underline{\alpha}, \underline{\beta}) = 0$, where $\underline{Y}$ denotes $Y_1, Y_2, Z_1, Z_2$, $\underline{\alpha}$ denotes $\alpha_{11}, \alpha_{12}, \alpha_{21}, \alpha_{22}$, and $\underline{\beta}$ denots $\beta_1, \beta_2, \rho$.
%However, from Theorem~\ref{prop:eqr2d3}, 
%we can achive a more precise estimator $\widehat{\underline{\alpha}}_N, \widehat{\underline{\beta}}_N$ by solving $\Psi(\underline{Y}, \underline{\alpha}_N, \underline{\beta}_N) = 0$, namely, knowing $\gamma_1, \gamma_2, \v_1, \v_2$ we have the information $\underline{Y}_N$. From the solution set
%\[
%\{ (\underline{\alpha}, \underline{\beta}) \mid \Psi(\underline{Y}_N, \underline{\alpha}, \underline{\beta}) = 0\},
%\]
%\end{remark}

\appendix

\section{Proof of Proposition~\ref{prop:lowrkappopt}}\label{subsec:proplowrkopt}

For the Lagrangian~\eqref{eq:Lagop}, the best rank-$r$ approximation $(\hat\gamma_1, \dots, \hat\gamma_r, \hat\v_1, \dots, \hat\v_r)$ needs to satisfy  	
\begin{equation}\label{eq:Lagde}
\begin{dcases}
\frac{\partial \sL}{\partial \gamma_i} = \langle \sT - \sum_{j=1}^r \hat\gamma_j \hat\v_j^{\otimes d}, \hat\v_i^{\otimes d} \rangle = 0 \\
\frac{\partial \sL}{\partial \v_i} = \langle \sT - \sum_{j=1}^r \hat\gamma_j \hat\v_j^{\otimes d}, \hat\v_i^{\otimes (d-1)} \pmb{w}_i \rangle = 0 \\
\frac{\partial \sL}{\partial \mu_i} = \langle \hat\v_i, \hat\v_i \rangle - 1 = 0
\end{dcases} \enspace,
\end{equation}
where $\pmb{w}_i$ is any tangent vector of the unit sphere at $\hat\v_i$. Since any vector $\u$ in $\R^N$ can be written as $\u = \alpha_i \hat\v_i + \pmb{w}_i$ for some $\alpha_i \in\mathbb{R}$ and tangent vector $\pmb{w}_i$ of the unit sphere at $\hat\v_i$, by~\eqref{eq:Lagde} we have
\[
 \langle \sT - \sum_{j=1}^r \hat\gamma_j \hat\v_j^{\otimes d}, \hat\v_i^{\otimes (d-1)}\u \rangle = \langle \sT - \sum_{j=1}^r \hat\gamma_j \hat\v_j^{\otimes d}, \hat\v_i^{\otimes (d-1)} (\alpha_i \hat\v_i + \pmb{w}_i) \rangle = 0,
\]
which implies that
\begin{equation*}
\langle \sT - \sum_{j=1}^r \hat\gamma_j \hat\v_j^{\otimes d}, \hat\v_i^{\otimes (d-1)} \rangle = 0.
\end{equation*}

%===========================================================================
% Proof for fixed point properties
%===========================================================================

\section{Proof of Proposition~\ref{prop:lowrkappopt2}}\label{subsec:proplowrkopt2}

We reformulate the first equation of~\eqref{eq:rkrapprox} as follows
\begin{equation}\label{eq:kkteq}
\begin{dcases}
\langle \sT, \hat\v^{\otimes (d-1)}_1 \rangle = \hat\gamma_1 \hat\v_1 + \hat\gamma_2 \langle \hat\v_1, \hat\v_2 \rangle^{d-1} \hat\v_2 + \cdots +  \hat\gamma_r \langle \hat\v_1, \hat\v_r \rangle^{d-1} \hat\v_r \\
\qquad \dots \\
\langle \sT, \hat\v^{\otimes (d-1)}_r \rangle = \hat\gamma_1 \langle \hat\v_1, \hat\v_r \rangle^{d-1} \hat\v_1 + \cdots + \hat\gamma_{r-1} \langle \hat\v_{r-1}, \hat\v_r \rangle^{d-1} \hat\v_{r-1} +  \hat\gamma_r \hat\v_r
\end{dcases},
\end{equation}
which can be written in the following matrix form
\begin{equation}\label{eq:kktma}
\begin{bmatrix}
    \langle \sT, \hat\v^{\otimes (d-2)}_1 \rangle & 0 & \dots  & 0 \\
    0 & \langle \sT, \hat\v^{\otimes (d-2)}_2 \rangle & \dots  & 0 \\
    \vdots & \vdots & \ddots & \vdots \\
    0 & 0 & \dots  & \langle \sT, \hat\v^{\otimes (d-2)}_r \rangle
\end{bmatrix} 
\begin{bmatrix}
\hat\v_1 \\
\hat\v_2 \\
\vdots \\
\hat\v_r
\end{bmatrix} =
  \begin{bmatrix}
    1  & \dots  & \langle \hat\v_1, \hat\v_r \rangle^{d-1} \\
    \vdots & \ddots  & \vdots\\
    \langle \hat\v_1, \hat\v_r \rangle^{d-1} & \dots  & 1
\end{bmatrix}
  \begin{bmatrix}
\hat\gamma_1\hat\v_1 \\
\hat\gamma_2\hat\v_2 \\
\vdots \\
\hat\gamma_r\hat\v_r
\end{bmatrix}.
\end{equation}
When $\pmb{\hat{R}}^{\odot (d-1)}_{vv,r}$ is invertible, by definition of $\flat(\sT)$, Equation~\eqref{eq:kktma} is of the form
\begin{equation}
\flat(\sT) \begin{bmatrix}
\hat\gamma_1\hat\v_1 \\
\hat\gamma_2\hat\v_2 \\
\vdots \\
\hat\gamma_r\hat\v_r
\end{bmatrix} = \hat\gamma_r \begin{bmatrix}
\hat\gamma_1\hat\v_1 \\
\hat\gamma_2\hat\v_2 \\
\vdots \\
\hat\gamma_r\hat\v_r
\end{bmatrix}.
\end{equation}
It remains to show that $\pmb{\hat{R}}^{\odot (d-1)}_{vv,r}$ is invertible when ${\hat \v_1}^{\otimes (d-1)}, \dots, {\hat \v_r}^{\otimes (d-1)}$ are linearly independent. In fact, we will show that $\pmb{\hat{R}}^{\odot (d-1)}_{vv,r}$ is positive-definite. Given any nonzero vector $\x = (x_1, \dots, x_r) \in \R^r$,
\begin{multline*}
\x^{\top} \pmb{\hat{R}}^{\odot (d-1)}_{vv,r} \x = \sum_{i, j} x_i \langle \hat\v_i, \hat\v_j \rangle^{d-1} x_j = \sum_{i, j} x_i \langle {\hat\v_i}^{\otimes (d-1)}, {\hat\v_j}^{\otimes (d-1)} \rangle x_j = \sum_{i, j}  \langle x_i {\hat\v_i}^{\otimes (d-1)}, x_j {\hat\v_j}^{\otimes (d-1)} \rangle \\
= \langle \sum_i x_i {\hat\v_i}^{\otimes (d-1)}, \sum_j x_j {\hat\v_j}^{\otimes (d-1)} \rangle = \lVert \sum_i x_i {\hat\v_i}^{\otimes (d-1)} \rVert^2 > 0.
\end{multline*}
Conversely, if $\pmb{\hat{R}}^{\odot (d-1)}_{vv,r}$ is positive-definite, ${\hat \v_1}^{\otimes (d-1)}, \dots, {\hat \v_r}^{\otimes (d-1)}$ have to be linearly independent.

\section{Proof of Lemma~\ref{lem:eigenreal}}\label{subsec:lemeigenreal}
% Note
% \begin{equation}
% \pmb{Z} = \begin{bmatrix}
%     1 & \langle \hat\v_1, \hat\v_2 \rangle^{d-1} & \langle \hat\v_1, \hat\v_3 \rangle^{d-1} & \dots  & \langle \hat\v_1, \hat\v_r \rangle^{d-1} \\
%     \langle \hat\v_1, \hat\v_2 \rangle^{d-1} &  & \langle \hat\v_2, \hat\v_3 \rangle^{d-1} & \dots  & \langle \hat\v_2, \hat\v_r \rangle^{d-1} \\
%     \vdots & \vdots & \vdots & \ddots & \vdots \\
%     \langle \hat\v_1, \hat\v_r \rangle^{d-1} & \langle \hat\v_2, \hat\v_r \rangle^{d-1} & \langle \hat\v_3, \hat\v_r \rangle^{d-1} & \dots  & 1
% \end{bmatrix}^{-1}
% \end{equation}
In the proof of Proposition~\ref{prop:lowrkappopt2} we have seen that the matrix $\pmb{\hat{R}}^{\odot (d-1)}_{vv,r}$ is symmetric positive-definite. Thus its square root $\Big(\pmb{\hat{R}}^{\odot (d-1)}_{vv,r}\Big)^{1/2}$ is well defined. Notice that the spectrum of $\flat(\sT)$ satisfies
\[
\begin{split}
\operatorname{Spec}\bigl(\flat(\sT)\bigr) &= \operatorname{Spec}\bigg(\pmb{\hat{R}}^{\odot (d-1)}_{vv,r} \begin{bmatrix}
\frac{\hat\gamma_r}{\hat\gamma_1} \langle \sT, \hat\v_1^{\otimes (d-2)}\rangle & 0 & 0 \\
0 & \ddots & 0 \\
0 & 0 & \langle \sT, \hat\v_r^{\otimes (d-2)}\rangle
\end{bmatrix}\bigg)\\
&= \operatorname{Spec}\bigg(\Big(\pmb{\hat{R}}^{\odot (d-1)}_{vv,r}\Big)^{1/2} \begin{bmatrix}
\frac{\hat\gamma_r}{\hat\gamma_1} \langle \sT, \hat\v_1^{\otimes (d-2)}\rangle & 0 & 0 \\
0 & \ddots & 0 \\
0 & 0 & \langle \sT, \hat\v_r^{\otimes (d-2)}\rangle
\end{bmatrix} \Big(\pmb{\hat{R}}^{\odot (d-1)}_{vv,r}\Big)^{1/2}\bigg).
\end{split}
\]
It is clear that $\M = \Big(\pmb{\hat{R}}^{\odot (d-1)}_{vv,r}\Big)^{1/2} \begin{bmatrix}
\frac{\hat\gamma_r}{\hat\gamma_1} \langle \sT, \hat\v_1^{\otimes (d-2)}\rangle & 0 & 0 \\
0 & \ddots & 0 \\
0 & 0 & \langle \sT, \hat\v_r^{\otimes (d-2)}\rangle
\end{bmatrix} \Big(\pmb{\hat{R}}^{\odot (d-1)}_{vv,r}\Big)^{1/2}$ is symmetric because each $\langle \sT, \hat\v_i^{\otimes (d-2)}\rangle$ is symmetric, which implies that the eigenvalues of $\flat(\sT)$ are real.

Assume that $\operatorname{Spec}\bigl(\flat(\sT)\bigr) = \{\rho_1\bigl(\flat(\sT)\bigr), \dots, \rho_{rN}\bigl(\flat(\sT)\bigr)\}$. Then
\begin{multline}
S_{\mu_{\flat(\sT)}} (z) = \frac{1}{rN} \sum_{i=1}^{rN} \int \frac{\delta_{\rho_i(\flat(\sT))}}{\rho - z} \,d\rho = \frac{1}{rN} \sum_{i=1}^{rN} \frac{1}{\rho_i\bigl(\flat(\sT)\bigr) - z} \\
= \frac{1}{rN} \sum_{i=1}^{rN} \frac{1}{\rho_i(\M) - z} = \frac{1}{rN} \operatorname{Tr} (\Q_{\pmb{M}}(z))= \frac{1}{rN} \operatorname{Tr} (\Q(z)),
\end{multline}
which completes the proof.
% \end{proof}

%======================================================================================
% Proof sepctrum of b (general case!!)
%======================================================================================

\section{Proof of Theorem~\ref{thm:generalized_measure}}\label{sec:prooftotal}

We first consider the case where $\sT=\frac{1}{\sqrt{N}}\sX$, i.e., the noise case.
In this section, we will focus on the case where the sequence $(\gamma_1,\dots,\gamma_r,\hat\v_1, \dots, \hat\v_r)$ is a sequence of critical points, i.e., that satisfies \eqref{eq:rkrapprox}. 
Note that, in order to lighten the notations, we will denote $\pmb{\hat W}_r$ by $\pmb{\hat W}$.

We will first establish the limit of $\frac{1}{N} \mathbb{E}[\Tr(\Q^{ij}(z))]$.
For this purpose, let us denote
\begin{equation}\label{eq:notationsA}
\A = \begin{bmatrix}\A^{11} & 0 & 0\\ 0 & \ddots & 0 \\ 0 &  0 &\A^{rr}\end{bmatrix} = \begin{bmatrix} \langle \sT, \hat\v_1^{\otimes (d-2)}\rangle & 0 & 0\\ 0 & \ddots & 0 \\ 0 &  0 & \langle \sT, \hat\v_r^{\otimes (d-2)}\rangle\end{bmatrix},
\end{equation}
and we start from the identity $\Q(z) (\flat(\sT) - z\pId) = \pId$ that gives us the following equations  for any $1\leq s,t\leq r$,
\[
\left\{\begin{array}{ll}
\frac{1}{N}  \sum_{u=1}^r \mathbb{E}\Tr[\hat{W}_{ut}\Q^{su}(z) \A^{tt}] - z \frac{1}{N} \mathbb{E}\Tr[\Q^{st}(z)]  = 1 &\text{if $s=t$}\\
\frac{1}{N}  \sum_{u=1}^r \mathbb{E}\Tr[\hat{W}_{ut}\Q^{su}(z) \A^{tt}] - z \frac{1}{N} \mathbb{E}\Tr[\Q^{st}(z)]   = 0 &\text{if $s\neq t$}
\end{array}\right..
\]
Since $\pmb{\hat W}\xrightarrow{a.s.} \pmb{W}$, we can replace $\pmb{\hat W}$ by $\pmb{W}$ in the expectations by noticing that
\[
\mathbb{E}\Tr[\hat{W}_{ut}\Q^{su}(z) \A^{tt}] = \mathbb{E}\Tr[W_{ut}\Q^{su}(z) \A^{tt}] + \mathbb{E}\Tr[(\hat{W}_{ut} - W_{ut})\Q^{su}(z) \A^{tt}],
\]
and controlling the second term by
\[
\begin{split}
\frac{1}{N} \left| \mathbb{E}\Tr[(\hat{W}_{ut} - W_{ut})\Q^{su}(z) \A^{tt}]\right| &\leq \max_{u,t} |W_{ut} - \hat{W}_{ut}|\, \lVert\pmb{\hat W}^{-1}\rVert_{\infty} \, \frac{1}{N} \left| \mathbb{E}\Tr[(\Q(z) (\pmb{\hat W}\otimes\pId)\A)_{st}] \right|\\
 &\leq \max_{u,t}|W_{ut} - \hat{W}_{ut}| \, \lVert\pmb{\hat W}^{-1}\rVert_{\infty} \, \Vert\Q(z) (\pmb{\hat W}\otimes\pId)\A\Vert_2\\
 &= \max_{u,t}|W_{ut} - \hat{W}_{ut}|\, \lVert\pmb{\hat W}^{-1}\rVert_{\infty} \, \Vert\pId + z\Q(z) \Vert_2 \xrightarrow{a.s} 0.
\end{split}
\]
Thus, for any $1\leq s,t\leq r$,
\begin{equation}\label{eq:allcrieqs_gal}
\left\{\begin{array}{ll}
\frac1N  \sum_{u=1}^r \mathbb{E}\Tr[W_{ut}\Q^{su}(z) \A^{tt}] - z \frac1N \mathbb{E}\Tr[\Q^{st}(z)]   \xrightarrow{a.s} 1 &\text{if $s=t$}\\
\frac1N  \sum_{u=1}^r \mathbb{E}\Tr[W_{ut}\Q^{su}(z) \A^{tt}] - z \frac1N \mathbb{E}\Tr[\Q^{st}(z)]  \xrightarrow{a.s} 0 &\text{if $s\neq t$}
\end{array}\right..
\end{equation}

Therefore, we need to control the term $\frac1N \sum_{k\neq i} \mathbb{E}\Tr[\Q^{su}(z) \A^{tt}]$ for any $1\leq s,t\leq r$.
For the sake of simplicity of exposition, we will consider in the sequel the case $s=u=1$ since the other terms can be controlled in a similar way.
Recall that
\[
A_{ji}^{11} = \frac{1}{\sqrt{N}}\sum_{\ell_1, \dots, \ell_{d-2}=1}^N X_{\ell_1\dots \ell_{d-2}ji} \hat{v}^1_{\ell_1}\cdots \hat{v}^1_{\ell_{d-2}} \enspace,
\]
where $\hat{v}^1_j$ denotes the $j$th entry of $\hat{v}_1$, thus, using Stein lemma,
\begin{equation}\label{eq:Q11Ark2}
\begin{split}
&\frac{1}{N} \E\bigl[\sum_{i,j=1}^N Q^{11}_{ij}A_{ji}^{11}\bigr] = \frac{1}{N\sqrt{N}} \sum_{i, j, \ell_1, \dots, \ell_{d-2}=1}^N \E\bigl[ Q^{11}_{ij} X_{\ell_1\dots \ell_{d-2}ji} \hat{v}^1_{\ell_1}\cdots \hat{v}^1_{\ell_{d-2}}\bigr] \\
= &\!\begin{multlined}[t]
\underbrace{\frac{1}{N\sqrt{N}} \sum_{i, j, \ell_1, \dots, \ell_{d-2}=1}^N \sigma^2_{\ell_1\dots\ell_{d-2}ji}\E \bigl[ \hat{v}^1_{\ell_1}\cdots \hat{v}^1_{\ell_{d-2}} \frac{\partial Q^{11}_{ij}}{\partial X_{\ell_1\dots \ell_{d-2}ji}}\bigr]}_{\sA_1} \\
+ \underbrace{\frac{1}{N\sqrt{N}} \sum_{i, j, \ell_1, \dots, \ell_{d-2}=1}^N \sum_{k=1}^{d-2} \sigma^2_{\ell_1\dots\ell_{d-2}ji}\E \bigl[ \hat{v}^1_{\ell_1}\cdots \hat{v}^1_{\ell_{k-1}} \frac{\partial \hat{v}^{1}_{\ell_k}}{\partial X_{\ell_1\dots \ell_{d-2}ji}} \hat{v}^1_{\ell_{k+1}}\cdots \hat{v}^1_{\ell_{d-2}}\, Q^{11}_{ij}\bigr]}_{\sA_2}.
\end{multlined}
\end{split}
\end{equation}
We will first prove that $\sA_2$ is negligible asymptotically. Then, we will focus on the evaluation of $\sA_1$.
Note that, in order to ensure that the derivative of $\hat\v_1$ with respect to $X_{\ell_1\dots \ell_{d-2}ji}$ exists, we can use the same arguments of the proof of \cite[Lemma~8]{GCC22}.

\subsection{Control of $\mathcal{A}_2$}\label{sec:A1}
In order to control $\mathcal{A}_2$, we need to quantify the dependency of $\hat\v_1$ with respect to $\sX$ using the fixed point equation satisfied by $\hat\v_1,\dots,\hat\v_r$.

Recall that the critical points need to satisfy $\sF(\{\hat\v_i\}, \{\hat\gamma_i\})=\mathbf{0}$ with
\[
\sF(\{\hat\v_i\}, \{\hat\gamma_i\})= \begin{bmatrix}
\langle \sT, \hat\v_1^{\otimes (d-1)}\rangle - \sum_{k=1}^r \hat\gamma_k \langle\hat\v_1,\hat\v_k\rangle^{d-1} \hat\v_k \\
\vdots\\
\langle \sT, \hat\v_r^{\otimes (d-1)}\rangle - \sum_{k=1}^r \hat\gamma_k \langle\hat\v_r,\hat\v_k\rangle^{d-1} \hat\v_k \\
\langle \hat\v_1, \hat\v_1 \rangle - 1 \\
\vdots \\
\langle \hat\v_r, \hat\v_r \rangle - 1
\end{bmatrix}.
\]
Then the Jacobian of $\sF$ with respect to $(\{\hat\v_i\}, \{\hat\gamma_i\})$ is given by
\begin{equation}\label{eq:Jacobian1}
\pmb{J} = \begin{bmatrix}\pmb{C} & \pmb{R}\\ \pmb{L} & \mathbf{0}\end{bmatrix}
\end{equation}
with
\begin{equation}\label{eq:defC_gal}
\pmb{C} = (d-1)\Q^{-\top}\left(\frac{\hat\gamma_r}{d-1}\right)\pmb{\hat W}^{-\top} - \pmb{\tilde C} \enspace
\end{equation}
and 
\[
\pmb{\tilde C} = (d-1) \begin{bmatrix}
\sum_{k\neq 1}\hat\gamma_k\langle \hat\v_1, \hat\v_k\rangle^{d-2}\hat\v_k\hat\v_k^{\top} & \hat\gamma_2\langle \hat\v_1, \hat\v_2\rangle^{d-2}\v_2\v_1^{\top} & \dots &  \hat\gamma_r\langle \hat\v_1, \hat\v_r\rangle^{d-2}\hat\v_r\hat\v_1^{\top}\\
\hat\gamma_1\langle \hat\v_2, \hat\v_1\rangle^{d-2}\hat\v_1\hat\v_2^{\top} & \sum_{k\neq 2}\hat\gamma_k\langle \hat\v_2, \hat\v_k\rangle^{d-2}\hat\v_k\hat\v_k^{\top}   &  \dots & \hat\gamma_r\langle \hat\v_2, \hat\v_r\rangle^{d-2}\hat\v_r\hat\v_2^{\top}\\
\vdots &\vdots  &  \ddots & \vdots\\
\hat\gamma_1\langle \hat\v_r, \hat\v_1\rangle^{d-2}\hat\v_1\hat\v_r^{\top} & \hat\gamma_2\langle \hat\v_r, \hat\v_2\rangle^{d-2}\hat\v_2\hat\v_r^{\top} & \dots& \sum_{k\neq 1}\hat\gamma_k\langle \hat\v_r, \hat\v_k\rangle^{d-2}\hat\v_k\hat\v_k^{\top} \\
\end{bmatrix}
\]
and
\[
\pmb{L} = \begin{bmatrix}2\hat\v_1^{\top} & 0 & 0 \\
0 & \ddots & 0\\ 0 & 0 & 2\hat\v_r^{\top}\end{bmatrix} \text{\quad and \quad} 
\pmb{R} = -\begin{bmatrix}
\hat\v_1 & \dots & \langle\hat\v_1,\hat\v_r\rangle^{d-1}\hat\v_r\\
\dots & \ddots & \dots\\
\langle\hat\v_r,\hat\v_1\rangle^{d-1} \hat\v_1 &\dots &  \hat\v_r
\end{bmatrix}.
\]
Since we assumed that $\hat\gamma_r\xrightarrow{a.s.}\gamma_r$ and that $\gamma_r$ lies outside the spectrum of $\Q$, with Assumption~\ref{assump:general}, $\Q\left(\frac{\hat\gamma_r}{d-1}\right)$ is invertible and has a bounded spectral norm.
Furthermore, the rank of $\pmb{\tilde C}$ is at most $r^2$. Therefore, $\pmb{C}$ is almost surely invertible, and by Woodbury identity,
\begin{equation}\label{eq:defC-1}
\pmb{C}^{-1} = \frac{1}{d-1}\pmb{\hat W}^{\top}\Q(\frac{\hat\gamma_r}{d-1})^{\top} +  \frac{1}{d-1}\pmb{\hat W}^{\top}\Q(\frac{\hat\gamma_r}{d-1})^{\top}\pmb{\tilde{C}}\left(\Q(\frac{\hat\gamma_r}{d-1})^{-\top}\pmb{\hat W}^{-\top} - \pmb{\tilde{C}}\right)^{-1}
\end{equation}
and by the Schur complement formula,
\begin{equation}\label{eq:defJ-1}
\pmb{J}^{-1} = \left(\begin{matrix}
\pmb{C}^{-1} + \pmb{C}^{-1}\pmb{R}\pmb{S}^{-1}\pmb{L}\pmb{C}^{-1} & -\pmb{C}^{-1}\pmb{R}\pmb{S}^{-1} \\
-\pmb{S}^{-1}\pmb{L}\pmb{C}^{-1} & \pmb{S}^{-1}
\end{matrix}
\right) \text{\quad with \quad}\pmb{S}=\pmb{L}\pmb{C}^{-1}\pmb{R}.
\end{equation}

On the other hand, denote by $X_{\alpha} = X_{i_1\dots i_d}$. 
By the implicit function theorem, we have that
\begin{equation}\label{eq:implicit}
\left(\begin{matrix}
\frac{\partial \hat\v^{\top}_1}{\partial X_{\alpha}} &
\dots &
\frac{\partial \hat\v^{\top}_r}{\partial X_{\alpha}} &
\frac{\partial \hat\gamma_1}{\partial X_{\alpha}} &
\dots &
\frac{\partial \hat\gamma_r}{\partial X_{\alpha}}
\end{matrix}
\right)^{\top} = - \pmb{J}^{-1}\, \frac{\partial \sF}{\partial X_{\alpha}}.
\end{equation}
In particular,
\begin{equation}\label{eq:partialvD}
\begin{bmatrix}
\frac{\partial \hat\v_1}{\partial X_{\alpha}} \\
\vdots\\
\frac{\partial \hat\v_r}{\partial X_{\alpha}}
\end{bmatrix} = - \bigl(\pmb{C}^{-1} + \pmb{C}^{-1}\pmb{R}\pmb{S}^{-1}\pmb{L}\pmb{C}^{-1}\bigr)\begin{bmatrix}
\frac{\partial \langle \sT, \hat\v_1^{\otimes(d-1)}\rangle }{\partial X_{\alpha}} \\
\vdots\\
\frac{\partial \langle \sT, \hat\v_r^{\otimes(d-1)}\rangle }{\partial X_{\alpha}}
\end{bmatrix}.
\end{equation}

Note that, if $j \notin\{i_1, \dots, i_d\}$, the derivative of the $j$th entry of $\langle \sT, \hat\v_1^{\otimes(d-1)}\rangle$ with respect to $X_{\alpha}$ is zero, i.e.,
\[
\frac{\partial \langle \sT, \hat\v_1^{\otimes(d-1)}\rangle_j }{\partial X_{\alpha}} = 0,
\]
where $X_{\alpha}$ denotes $X_{i_1\dots i_d}$. Otherwise, when $j \in\{i_1, \dots, i_d\}$, assume that
\[
\{i_1, \dots, i_d\}\setminus \{j\} = \{\ell_1, \dots, \ell_1, \dots, \ell_p, \dots, \ell_p\}
\]
as a multiset, where the multiplicity of $\ell_k$ is $m_k$. Then
\begin{equation}\label{eq:multTvd-1gal}
\frac{\partial \langle \sT, \hat\v_1^{\otimes(d-1)}\rangle_j }{\partial X_{\alpha}} = \frac{1}{\sqrt{N}} \binom{d-1}{m_1, \dots, m_p}(\hat{v}^1_{\ell_1})^{m_1}\cdots (\hat{v}^1_{\ell_p})^{m_p}.
\end{equation}
For convenience of notation, we use $Q^{11}_{ij}$ to denote the $(i, j)$th entry of $\Q^{11}(z)$.

In particular, for $\sA_2$, we focus on the following term appearing in $\sA_2$ as an example of the argument
\begin{equation}\label{eq:partialv1exm}
\frac{1}{N\sqrt{N}} \sum_{i, j, \ell_1, \dots, \ell_{d-2}=1}^N \sigma^2_{\ell_1\dots\ell_{d-2}ji}\E \bigg[ \frac{\partial \hat{v}^{1}_{\ell_1}}{\partial X_{\ell_1\dots \ell_{d-2}ji}} \hat{v}^1_{\ell_{2}}\cdots \hat{v}^1_{\ell_{d-2}} \, Q^{11}_{ij}\bigg] \enspace.
\end{equation}
Assume $\pmb{C}^{-1} + \pmb{C}^{-1}\pmb{R}\pmb{S}^{-1}\pmb{L}\pmb{C}^{-1}$ has the form
\begin{equation}\label{eq:Ddef}
\pmb{C}^{-1} + \pmb{C}^{-1}\pmb{R}\pmb{S}^{-1}\pmb{L}\pmb{C}^{-1} = \left(\begin{matrix}
\pmb{D}^{11} & \dots & \pmb{D}^{1r} \\
\vdots & \ddots & \vdots \\
\pmb{D}^{r1} &\dots& \pmb{D}^{rr}
\end{matrix}\right).
\end{equation}
Then by~\eqref{eq:var} and~\eqref{eq:multTvd-1gal}, we have the following estimation for the term~\eqref{eq:partialv1exm},
\begin{equation}\label{eq:crieq1rk2}
\begin{split}
&\left\lvert\frac{1}{N\sqrt{N}} \sum_{i, j, \ell_1, \dots, \ell_{d-2}=1}^N \sigma^2_{\ell_1\dots\ell_{d-2}ji}\E \bigg[\frac{\partial \hat{v}^{1}_{\ell_1}}{\partial X_{\ell_1\dots \ell_{d-2}ji}} \hat{v}^1_{\ell_{2}}\cdots \hat{v}^1_{\ell_{d-2}} \, Q^{11}_{ij}\bigg] \right\rvert\\
= &\frac{1}{N\sqrt{N}} \sum_{i, j, \ell_1, \dots, \ell_{d-2}=1}^N \sigma^2_{\ell_1\dots\ell_{d-2}ji} \left\lvert\E \bigg[\sum_{k=1}^N\sum_{s=1}^r D^{1s}_{\ell_1 k}\frac{\partial \langle \sT, \hat\v_s^{\otimes(d-1)}\rangle_k }{\partial X_{\alpha}} v^1_{\ell_{2}}\cdots \hat{v}^1_{\ell_{d-2}} \, Q^{11}_{ij}\bigg] \right\rvert\\
\le &\!\begin{multlined}[t]
\frac{1}{N^2} \E\bigg[\bigl(\left\lvert\Tr(\pmb{D}^{11})\right\rvert + (d-3) \left\lvert \hat\v_1^{\top}\pmb{D}^{11}\v_1 \right\rvert \bigr)\left\lvert \hat\v_1^{\top}\hat\nu\Q^{11}\hat\v_1 \right\rvert \bigg] + \\
\frac{1}{N^2}\bigg[\E\left\lvert \hat\v_1^{\top}\pmb{D}^{11} (\hat\nu\Q^{11})^{\top}\hat\v_1\right\rvert + \E\left\lvert \hat\v_1^{\top} \pmb{D}^{11} \hat\nu\Q^{11}\hat\v_1 \right\rvert\bigg] \\
+\sum_{s=2}^r \frac{1}{N^2} \E\bigg[ \bigl(\left\lvert\Tr(\pmb{D}^{1s})\right\rvert \Vert \hat\v_1\odot \hat\v_s\Vert_1 + (d-3) \left\lvert \hat\v_s^{\top}\pmb{D}^{1s}\hat\v_1\right\rvert \bigr)\bigl(\Vert \hat\v_1\odot \hat\v_s\Vert_1\bigr)^{d-4} \left\lvert \hat\v_s^{\top}\hat\nu\Q^{11}\hat\v_s\right\rvert\bigg] \\
+ \frac{1}{N^2} \E \bigg[ \bigl(\Vert \hat\v_1\odot \hat\v_s\Vert_1\bigr)^{d-3} \Big(\left\lvert \hat\v_s^{\top}\pmb{D}^{1s} (\hat\nu\Q^{11})^{\top}\hat\v_s \right\rvert + \left\lvert \hat\v_s^{\top}\pmb{D}^{1s} \hat\nu\Q^{11}\hat\v_s \right\rvert \Big)\bigg] \enspace,
\end{multlined}
\end{split}
\end{equation}
where $\odot$ denotes the Hadamard product. Since the spectral norms
\[
\lVert \Q^{11} \rVert_2 \quad \text{and} \quad \lVert \pmb{D}^{ij} \rVert_2
\]
are bounded, the right-hand side of Inequality~\eqref{eq:crieq1rk2} tends to $0$ when $N$ goes to infinity. This implies that $\sA_2$ can actually be ignored when computing~\eqref{eq:Q11Ark2}.

\subsection{Control of $\mathcal{A}_1$}\label{sec:estimateA1}

On the other hand, from $\flat(\sT)\Q - z\Q = \pId$, we deduce that
\[
\frac{\partial \Q}{\partial X_{\ell_1\dots \ell_{d-2}ji}} = - \Q \frac{\partial \flat(\sT)}{\partial X_{\ell_1\dots \ell_{d-2}ji}} \Q \enspace.
\]
In particular, using the notations of \eqref{eq:notationsA} and considering the first block of $\frac{\partial \Q}{\partial X_{\ell_1\dots \ell_{d-2}ji}}$,
\begin{equation}\label{eq:parQ}
\frac{\partial \Q^{11}}{\partial X_{\ell_1\dots \ell_{d-2}ji}} = - \sum_{s,t=1}^r \hat{W}_{st}\Q_{1s}\frac{\partial \A^{tt}}{\partial X_{\ell_1\dots \ell_{d-2}ji}}\Q_{t1}.
\end{equation}

Therefore we have
\begin{equation}\label{eq:parQ11ij}
\frac{\partial  Q^{11}_{ij}}{\partial X_{\ell_1\dots \ell_{d-2}ji}} = - \sum_{s,t=1}^r\sum_{k,l=1}^N  \hat{W}_{st}Q^{1s}_{ik}\frac{\partial A_{kl}^{tt}}{\partial X_{\ell_1\dots \ell_{d-2}ji}}Q^{t1}_{lj}.
\end{equation}
Due to the symmetry of $\sT$, the entry $\frac{\partial A_{kl}^{tt}}{\partial X_{\ell_1\dots \ell_{d-2}ji}}$ is nonzero if and only if $\{k, l\} \subseteq \{\ell_1, \dots, \ell_{d-2}, j, i\}$ as a submultiset (which means that we allow repetitions of elements). More precisely, let
\[
I = \{\ell_1, \dots, \ell_{d-2}, j, i\} \setminus \{k, l\} = \{i_1, \dots, i_{d-2}\}, \quad \text{and} \quad \hat{v}^1_I = \hat{v}^t_{i_1} \cdots \hat{v}^t_{i_{d-2}} \enspace.
\]
Assume that $I$ has the form $\{\xi_1, \dots, \xi_1, \dots, \xi_p, \dots, \xi_p\}$, where the multiplicity of each $\xi_h$ is $m_h$. Then
\[
\frac{\partial A_{kl}^{tt}}{\partial X_{\ell_1\dots\ell_{d-2}ji}} = \frac{1}{\sqrt{N}} \binom{d-2}{m_1, \dots, m_p}\hat{v}^t_I.
\]
We define
\[
\Omega_{I, k, l} = \sigma^2_{\ell_1,\dots,\ell_{d-2}ji}\, \frac{\partial A_{kl}^{tt}}{\partial X_{\ell_1\dots\ell_{d-2}ji}} \enspace.
\]
It is evident that $\Omega_{I, k, l}\le \frac{1}{\sqrt{N}}\hat{v}^t_I$. In particular, when $\{k, l\} \cap \{i_1, \dots, i_{d-2}\} = \emptyset$, we have
\begin{equation}\label{eq:Omega}
\Omega_{I, k, l} = \frac{1}{\sqrt{N}}\frac{1}{d(d-1)}\hat{v}^t_I.
\end{equation}
Since $\lVert \hat\v_t\rVert_2 = 1$, the cardinality of $\{i\in[N] \mid \hat{v}_i^t = O(1) \text{ as } N \to\infty\}$ is finite. In addition, for fixed indices $i_1,\dots, i_{d-2}$, the cardinality of 
\[
\Psi_I = \left\{k,l\in [N]\times [N] \,\middle\vert\, \Omega_{I,k,l}\ne \frac{1}{\sqrt{N}}\frac{1}{d(d-1)}\hat{v}^t_I \right\}
\]
is finite as $N \to \infty$. We first consider the first term ($t=1$) in $\mathcal{A}_1$ using \eqref{eq:parQ11ij}
\begin{equation*}\label{eq:11A11}
\frac{1}{N\sqrt{N}}\sum_{i,j,\ell_1,\dots,\ell_{d-2} = 1}^N \E\bigg[\hat{v}^t_{\ell_1}\cdots \hat{v}^t_{\ell_{d-2}}\sigma^2_{\ell_1\dots\ell_{d-2}ji} \sum_{s=1}^r \hat{W}_{s1}Q^{1s}_{ik}\frac{\partial A_{kl}^{11}}{\partial X_{\ell_1\dots \ell_{d-2}ji}}Q^{11}_{lj}\bigg] \enspace. \tag{*}
\end{equation*}

\begin{itemize}
\item When $k = j$ and $l = i$,
\[
\begin{split}
\eqref{eq:11A11} &\le \frac{1}{N^2}\, \sum_{\ell_1, \dots, \ell_{d-2}}\E\bigl[ (\hat{v}^1_{\ell_1}\cdots \hat{v}^1_{\ell_{d-2}})^2\, \sum_{i,j}\sum_{s=1}^r  \hat{W}_{s1}Q^{1s}_{ij} Q^{11}_{ij} \bigr] \\
&= \frac{1}{N^2}\, \E\bigl[ \Vert \hat\v_1^{\otimes (d-2)} \Vert^2_{\nF} \, \sum_{i,j} \sum_{s=1}^r \hat{W}_{s1} Q^{1s}_{ij} Q^{11}_{ij} \bigr] \\
&\le \frac{1}{N^2}\, \E \bigg[\left\Vert\sum_{s=1}^r \hat{W}_{s1} \Q^{1s} \right\Vert_{\nF} \, \Vert \Q^{11}\Vert_{\operatorname{F}}\bigg] \\
&\le \frac{r}{N}\, \E \bigg[\Vert \Q\Vert^2_2 \bigg] \xrightarrow{N \to \infty} 0 \enspace,
\end{split}
\]
since the spectral norm of $\Q$ is bounded and $\lVert\pmb{\hat W}\rVert_{\infty} \le 1$.
\item When $k = i$ and $l = j$, we first claim that
\[
\frac{1}{N\sqrt{N}}\sum_I \sum_{\{k, l\}\cap I \ne \emptyset} \hat{v}^1_I\, \left\lvert\Omega_{I,k,l} - \frac{1}{\sqrt{N}}\frac{1}{d(d-1)} \hat{v}^1_I \right\vert\, \E\bigl[\sum_{s=1}^r \hat{W}_{s1} Q^{1s}_{kk} Q^{11}_{ll}\bigr] \xrightarrow{N \to \infty} 0 \enspace.
\]
The proof of this claim splits into two cases, i.e., only one of $k$ and $l$ is contained in $I$, or both of them are in $I$. Since the spirits of the proofs are the same, here we only present the argument for the case $k\in I$ and $l\notin I$. In this case, since
\[
\,\left\vert\Omega_{I,k,l} - \frac{1}{\sqrt{N}}\frac{1}{d(d-1)} \hat{v}^1_I\right\vert\, \le \frac{1}{\sqrt{N}} \lvert\hat{v}^1_I \rvert \enspace,
\]
we have
\[
\begin{split}
&\frac{1}{N\sqrt{N}}\sum_I \sum_{\{k, l\}\cap I \ne \emptyset} \lvert\hat{v}^1_I\rvert \, \left\lvert\Omega_{I,k,l} - \frac{1}{\sqrt{N}}\frac{1}{d(d-1)} \hat{v}^1_I \right\vert\, \E\bigl[\sum_{s=1}^r \hat{W}_{s1} Q^{1s}_{kk} Q^{11}_{ll}\bigr] \\
\le &\frac{1}{N^2}\sum_{J \subseteq ([N]\setminus \{k\})^{d-1}} (\hat{v}^1_{J})^2\sum_{k\in [N]} \E\bigl[\sum_{s=1}^r \hat{W}_{s1}\, Q^{1s}_{kk}\, \lvert\hat{v}^1_k\rvert^{m_{k,I}+1}\bigr] \E\bigl[\sum_l Q^{11}_{ll}\bigr] \\
\le & \bigg[\E\bigl(\frac{\Tr \Q^{11}}{N}\bigr)\bigg] \frac{\E\lVert \sum_{s=1}^r \hat{W}_{s1}\Q^{1s}\rVert_2}{N} \xrightarrow{N \to \infty} 0
\end{split}
\]
because the spectral norm $\lVert\Q\rVert_2$ is upper bounded, where $m_{k,I}$ is the multiplicity of $k$ in $I$. This claim implies that when considering the limiting behavior of~\eqref{eq:11A11} we are safe to use~\eqref{eq:Omega} for every tuple $I, k, l$. Thus
\[
\begin{split}
\eqref{eq:11A11} &\xrightarrow{N \to \infty} \frac{1}{d(d-1)}\frac{1}{N^2} \, \sum_{\ell_1, \dots, \ell_{d-2}} (\hat{v}^1_{\ell_1}\cdots \hat{v}^1_{\ell_{d-2}})^2\, \E\bigg[\sum_i \sum_{s=1}^r \hat{W}_{s1} Q^{1s}_{ii}\bigg] \E\bigl[\sum_j Q^{11}_{jj}\bigr] \\
&= \frac{1}{d(d-1)}\frac{1}{N^2} \, \E\bigg[\sum_{s=1}^r \hat{W}_{s1}\bigl(\Tr \Q^{1s}\bigr) \bigg] \, \bigg[\E\bigl(\Tr \Q^{11}\bigr)\bigg] \\
&\xrightarrow{N \to \infty} \frac{1}{d(d-1)} \, \sum_{s=1}^r W_{s1} g_{1s}(z) g_{11}(z)
\end{split}
\]
\item When $k = i$ and $l\in \{\ell_1, \dots, \ell_{d-2}\}$,
\[
\begin{split}
\eqref{eq:11A11} &\le \frac{1}{N^2} \,  \sum_i \E\bigl[\sum_{s=1}^r \hat{W}_{s1} Q^{1s}_{ii}\bigr] \, \sum_{\ell_1, \dots, \ell_{d-2}} \, \sum_{l\in \{\ell_1, \dots, \ell_{d-2}\}} \E\bigl[ \hat{v}^1_l (\hat{v}^1_{\{\ell_1, \dots, \ell_{d-2}\}\setminus\{l\}})^2 \, \sum_j Q^{11}_{lj} \hat{v}^1_{j} \bigr]\\
&= (d-2)\, \frac{\E\bigl(\hat\v_1^{\top} \Q^{11}\hat\v_1\bigr)}{N}\, \sum_{s=1}^r \frac{\E\bigl(\hat{W}_{s1}\Tr \Q^{1s}\bigr)}{N} \\
&\le (d-2)\, \frac{\E\Vert \Q^{11}\Vert_2}{N} \, \sum_{s=1}^r \frac{\E\bigl(\hat{W}_{s1}\Tr \Q^{1s}\bigr)}{N} \xrightarrow{N \to \infty} 0
\end{split}
\]
\item When $k = j$ and $l\in \{\ell_1, \dots, \ell_{d-2}\}$,
\[
\begin{split}
\eqref{eq:11A11} &\le \frac{1}{N^2} \, \sum_j \sum_{\ell_1, \dots, \ell_{d-2}} \sum_{l\in \{\ell_1, \dots, \ell_{d-2}\}} \E\bigg[ (\hat{v}^1_{\{\ell_1, \dots, \ell_{d-2}\}\setminus\{l\}})^2 \,  \sum_i \sum_{s=1}^r \hat{W}_{s1} Q^{1s}_{ij}\hat{v}^1_i \bigg] \, \E\bigl[Q^{11}_{lj}\hat{v}^1_{l} \bigr]\\
&= (d-2)\, \bigg[\frac{\E\bigl( \sum_{s=1}^r \hat{W}_{s1} \hat\v_1^{\top}\Q^{11} \mathbf{1}\bigr)}{N}\bigg] \, \bigg[\frac{\E\bigl(\hat\v_1^{\top} \Q^{11} \mathbf{1}\bigr)}{N}\bigg]\\
&\le (d-2)\, \bigg[\sum_{s=1}^r \frac{\E  \max_s |\hat{W}_{s1}| \Vert\Q^{1s}\Vert_2}{\sqrt{N}}\bigg] \, \bigg[\frac{\E\Vert \Q^{11}\Vert_2}{\sqrt{N}}\bigg] \xrightarrow{N \to \infty} 0 \enspace,
\end{split}
\]
where $\mathbf{1}$ denotes the vector whose entries are all $1$.
\item When $k\in \{\ell_1, \dots, \ell_{d-2}\}$ and $l\in \{i,j\}$, similarly $\eqref{eq:11A11} \xrightarrow{N \to \infty} 0$.
\item When $\{k, l\} \in \{\ell_1, \dots, \ell_{d-2}\}$,
\[
\begin{split}
\eqref{eq:11A11} &\le \frac{1}{N^2} \, \sum_{i, j} \, \sum_{\ell_1, \dots, \ell_{d-2}}\, \sum_{k, l \in\{\ell_1, \dots, \ell_{d-2}\}} \E\bigg[ (\hat{v}^1_{\{\ell_1, \dots, \ell_{d-2}\}\setminus\{k, l\}})^2 \, \hat{v}^1_i \sum_{s=1}^r \hat{W}_{1s} Q^{1s}_{ik}  \hat{v}^1_k\bigg] \E\bigl[\hat{v}^1_l Q^{11}_{lj} \hat{v}^1_j \bigr] \\
&= \binom{d-2}{2} \, \bigg[\frac{\E\bigl(\sum_{s=1}^r \hat{W}_{s1}\hat\v_1^{\top}  \Q^{1s} \hat\v_1\bigr)}{N}\bigg] \, \bigg[\frac{\E\bigl(\hat\v_1^{\top} \Q^{11}\hat\v_1\bigr)}{N}\bigg] \\
&\le \binom{d-2}{2}\, \bigg[\sum_{s=1}^r \max_s |\hat{W}_{s1}| \frac{\E \Vert \Q^{1s}\Vert_2}{N}\bigg] \, \bigg[\frac{\E\Vert \Q^{11}\Vert_2}{N}\bigg] \xrightarrow{N \to \infty} 0
\end{split}
\]
\end{itemize}
Thus we have
\begin{multline}\label{eq:par1Q11}
 \frac{1}{N\sqrt{N}} \E\bigg[ \sum_{i,j,\ell_1,\dots,\ell_{d-2}} \hat{v}^1_{\ell_1} \cdots \hat{v}^1_{\ell_{d-2}}\sigma^2_{\ell_1\dots\ell_{d-2}ji} \sum_{s=1}^r \hat{W}_{s1}\Q^{1s} \frac{\partial \A^{11}}{\partial X_{\ell_1\dots \ell_{d-2}ji}}\Q^{11} \bigg]_{ij} \\
\xrightarrow{N \to \infty}  \frac{1}{d(d-1)} \, \sum_{s=1}^r W_{s1} g_{1s}(z) g_{11}(z)
\end{multline}
We proceed similarly for any $1\leq t\leq r$, which leads to
\begin{equation}\label{eq:par2Q11}
\mathcal{A}_1 \xrightarrow{N \to \infty} - \frac{1}{d(d-1)} \, \sum_{s,t=1}^r W_{st} g_{1s}(z) g_{t1}(z).
\end{equation}

\subsection{Expression of $\pmb{G}(z)$}

Combining the convergence of $\mathcal{A}_1$ and $\mathcal{A}_2$ gives us
\begin{equation}\label{eq:pardQ11}
\frac{1}{N} \E\bigl[\sum_{i,j=1}^N Q^{11}_{ij}A_{ji}^{11}\bigr] \xrightarrow{N \to \infty}- \frac{1}{d(d-1)} \sum_{s,t=1}^r W_{st} g_{1s}(z) g_{t1}(z)=\pmb{e}_1^{\top}\pmb{G}(z)\pmb{W}_r \,\pmb{G}(z)\pmb{e}_1
\end{equation}
with $\pmb{e}_1$ the vector whose entries are all zero except the first coordinate that equals $1$ and
\[
\pmb{G}(z) = \begin{bmatrix} g_{11}(z) & \dots & g_{1r}(z) \\ \vdots & \ddots & \vdots\\ g_{r1}(z) & \dots & g_{rr}(z) \end{bmatrix}.
\]
Similarly, we get
\begin{equation}
\frac{1}{N} \E\bigl[\sum_{i,j=1}^N Q^{su}_{ij}A_{ji}^{tt}\bigr] \xrightarrow{N \to \infty} -\frac{1}{d(d-1)} \pmb{e}_s^{\top}\pmb{G}(z)\pmb{W}_r \,\pmb{G}(z)\pmb{e}_u.
\end{equation}
Therefore the equality in Equation~\eqref{eq:allcrieqs_gal} when $s\neq t$ converges to
\begin{equation}\label{eq:conveq1}
\left\{\begin{array}{ll}
\frac{1}{d(d-1)}\sum_{u=1}^r W_{ut}\pmb{e}_s^{\top}\pmb{G}(z)\pmb{W}_r\,\pmb{G}(z)\pmb{e}_u - z g_{st}(z) = 1 & \text{if $s= t$}\\
\frac{1}{d(d-1)}\sum_{u=1}^r W_{ut}\pmb{e}_s^{\top}\pmb{G}(z)\pmb{W}_r\,\pmb{G}(z)\pmb{e}_u - z g_{st}(z) = 0 & \text{if $s\neq t$}
\end{array}\right..
\end{equation}
We can combine all equations in \eqref{eq:conveq1} in a matrix form and obtain
\begin{equation}\label{eq:fixedpoint_gal}
\frac{1}{d(d-1)} \pmb{G}(z)\pmb{W}_r\,\pmb{G}(z)\pmb{W}_r + z \pmb{G}(z) + \pId = \mathbf{0}.
\end{equation}
Finally we notice the eigendecomposition 
\[
\frac{2}{\sqrt{d(d-1)}}\pmb{W}_r = \pmb{P}_r\, \begin{bmatrix}
    \kappa_1^{(r)}  & 0  & 0 \\
     0 & \ddots & 0 \\
    0 & 0   & \kappa_r^{(r)}\end{bmatrix}\pmb{P}_r^{-1}.
\]
Since \eqref{eq:fixedpoint_gal} can be reformulated as
\[
\frac{1}{d(d-1)} \pmb{G}(z)\pmb{W}_r + \Big(\pmb{G}(z)\pmb{W}_r\Big)^{-1} = - z \pmb{W}_r,
\]
$\pmb{P}_r$ also diagonalizes $\pmb{G}(z)$ except for a finite number of $z$. It remains then to compute the eigenvalues of $\pmb{G}(z)$ which we denote by $\zeta_1(z), \dots, \zeta_2(z)$. Then~\eqref{eq:fixedpoint_gal} becomes
\begin{equation}\label{eq:zeta}
 \frac{1}{4} \begin{bmatrix} (\kappa_1^{(r)})^2\zeta^2_1(z) & \dots & 0 \\ 0  &\ddots & 0 \\ 0 & 0 & (\kappa_1^{(r)})^2\zeta^2_r(z) \end{bmatrix} +  \begin{bmatrix} z\zeta_1(z) & 0 & 0 \\ 0 & \ddots & 0\\ 0 & 0 & z\zeta_r(z) \end{bmatrix} +  \pId =\pmb{0} \enspace.
\end{equation}
Because the Stieltjes transform behaves as $O(z^{-1})$ at infinity, Equation~\eqref{eq:zeta} has a unique solution which gives rise to the Stieltjes transform of some measure.
Distinguishing between the case $\kappa = 0$ and the case $\kappa \neq 0$ provides the result.

\subsection{Resolvent study: concentration property}\label{subsub:concentration}

% Since $\lVert \Q(z) \rVert_2$ and $\lVert \Q^{ij}(z) \rVert_2$ are bounded, the limit
% \[
% \lim_{N\to\infty} \frac{1}{N} \E\bigl[\Tr \Q^{ij}(z)\bigr]
% \]
% exists for all $i, j$, which we denote by $g_{ij}(z)$. 
To show the almost surely convergence, notice that
\[
\lVert \bigl(\Q^{ij}(z)\bigr)^k \rVert_2 \le \lVert \Q^{ij}(z) \rVert^k_2 \le \frac{1}{\lvert\mathfrak{I}(z) \rvert^k} \quad \text{and} \quad \lVert \bigl(\Q(z)\bigr)^k \rVert_2 \le \frac{1}{\lvert\mathfrak{I}(z) \rvert^k}
\]
for all $z\in \mathbb{C}\setminus \mathbb{R}$ and positive integer $k$, which implies that all moments
\[
\frac{1}{N} \operatorname{Tr} \bigl(\Q^{ij}(z)\bigr)^k \quad \text{and} \quad \frac{1}{rN} \operatorname{Tr} \bigl(\Q(z)\bigr)^k
\]
are bounded on $\mathbb{C}\setminus \mathbb{R}$, and thus there exist convergent subsequences of
\[
\bigg(\frac{1}{N} \mathbb{E}[\operatorname{Tr}\Q^{ij}(z)] \bigg)_{N\ge 1} \quad \text{and} \quad \bigg(\frac{1}{rN} \mathbb{E}[\operatorname{Tr}\Q(z)] \bigg)_{N\ge 1}.
\]
In fact, all convergent subsequences must converge to a same limit, i.e.,
\[
\lim_{N\to \infty}\frac{1}{N} \mathbb{E}[\operatorname{Tr}\Q^{ij}(z)] = g_{ij}(z) \quad \text{and} \quad \lim_{N\to\infty}\frac{1}{rN} \mathbb{E}[\operatorname{Tr}\Q(z)] = g(z) = \frac1r \sum_{s=1}^r g_{ss}(z)
\]
for some $g_{ij}(z)$, because all these limits need to satisfy Equation~\eqref{eq:fixedpoint_gal} which has a unique solution that can give rise to a Stieltjes transform. Now we would like to show that $\frac{1}{rN} \operatorname{Tr} \bigl(\Q(z)\bigr)$ actually concentrates around $g(z)$ by controlling the variances. Indeed, by Lemma~\ref{lem:Poincare},
\begin{equation}\label{eq:vartrace}
\operatorname{\mathbb{V}ar} \bigg(\frac{1}{rN} \operatorname{Tr} \bigl(\Q(z)\bigr)\bigg) \le \frac{1}{N^2} \mathbb{E} \bigg[\sum_{\bm{i}} \sigma^2_{\bm{i}} \Big(\sum_j\sum_s \frac{\partial Q_{jj}^{ss}(z)}{\partial X_{\bm{i}}} \Big)^2 \bigg]\enspace.
\end{equation}
For simplicity of notation, we will use $\Q$ instead of $\Q(z)$ in the remaining of this paper. By Equation~\eqref{eq:parQ11ij} and Equation~\eqref{eq:Omega}, we have
\[
\begin{split}
\sum_{\bm{i}}\sum_j \bigg[\frac{\partial Q_{jj}^{11}}{\partial X_{\bm{i}}} \bigg]^2 \le & \frac{1}{N} \bigl((d-2)!\bigr)^2 \sum_{j, k, l}\sum_t \bigg[(\sum_s\hat{W}_{st} Q^{1s}_{jk})^2(Q^{t1}_{lj})^2 \bigg] \\
= &\frac{1}{N} \bigl((d-2)!\bigr)^2 \bigg[\sum_t\lVert  \sum_s \hat{W}_{st}\Q^{t1}\Q^{1s}\rVert_{\operatorname{F}}^2 \bigg] \\
\le & \bigl((d-2)!\bigr)^2 \bigg[\sum_t\Vert \Q^{t1}\Vert_2^2   \sum_s |\hat{W}_{st}|\lVert\Q^{1s}\rVert_2^2 \bigg],
\end{split}
\]
which implies that $\sum_{\bm{i}}\sum_j \bigg[\frac{\partial Q_{jj}^{11}}{\partial X_{\bm{i}}} \bigg]^2$ is $\mathcal{O}(1)$. Similarly, we can show that $\sum_{\bm{i}}\sum_j \bigg[\frac{\partial Q_{jj}^{ss}}{\partial X_{\bm{i}}} \bigg]^2$ is $\mathcal{O}(1)$. Hence $\operatorname{\mathbb{V}ar} \bigg(\frac{1}{rN} \operatorname{Tr} \bigl(\Q(z)\bigr)\bigg)$ is of $\mathcal{O}(\frac{1}{N^2})$. By Chebyshev's inequality and Borel-Cantelli lemma, $\frac{1}{N} \operatorname{Tr} \Q^{ij}(z)$ and $\frac{1}{rN} \operatorname{Tr} \Q(z)$ converge to $g_{ij}(z)$ and $g(z)$ almost surely respectively.

%Under Assumption~\ref{assump2}, by the proof of Theorem~\ref{thm:lowrank_Stieltjes} and using Lemma~\ref{lem:Poincare}, we can control the variances of the limiting alignments $\alpha_{ij}$, and then show that the alignments converge almost surely to deterministic limits. For example,
%\[
%\operatorname{\mathbb{V}ar} \bigl(\langle \u_1, \v_1 \rangle\bigr) \le \frac{1}{N}\E \bigl(\lvert \langle \u_1, \pmb{D}_{11}\w^1 \rangle\rvert^2 + \lvert \langle \u_1, \pmb{D}_{12}\w^2 \rangle\rvert^2\bigr) \le \frac{\lVert \D_{11}\rVert_2 + \lVert \D_{12}\rVert_2}{N},
%\]
%where $\lVert \D_{11}\rVert_2$ and $\lVert \D_{12}\rVert_2$ are bounded. By Chebyshev's inequality, for any $t > 0$, we have
%\[
%\mathbb{P}\bigl(\lvert \alpha_{11} - \E\alpha_{11}\rvert \ge t \bigr) \le \frac{\lVert \D_{11}\rVert_2 + \lVert \D_{12}\rVert_2}{Nt^2}.
%\]
%Thus it remains to estimate the expectations $\mathbb{E} \alpha_{ij}$.

%======================================================================================
% Rank-r case
%======================================================================================

\subsection{Rank-$r$ case}

% In this section, we will focus on the case where the sequence $(\v_1, \v_2)_*$ satisfies Assumption~\ref{assump:purenoise}. The concentration phenomenon is almost identical to the part given in~\ref{subsub:concentration}. Thus we will focus on the proof that Equation~\eqref{eq:measurefpeqmatfrk2} also holds in this situation.

In the case of $\flat(\sT)$, it suffices to show that the limiting Stieltjes transforms of $\mu_{\flat(\sT)}$ and $\mu_{\flat(\frac{1}{\sqrt{N}}\sX)}$ are equal, i.e.,
\[
\lim_{N\to \infty} S_{\mu_{\flat(\sT)}} (z) =\lim_{N\to \infty} S_{\mu_{\flat(\frac{1}{\sqrt{N}}\sX)}} (z) \enspace.
\]
To this end, let
\[
\Q(z) = \bigl(\frac{1}{\sqrt{N}} \flat(\sX) - z\pId \bigr)^{-1} \text{ and } \pmb{R}(z) = \bigl(\flat(\sT) - z\pId \bigr)^{-1}.
\]
Then by Woodbury matrix identity we have
\[
\begin{split}
\frac{1}{rN}\Tr \pmb{R}(z) &= \frac{1}{rN}\Tr \bigg[\Q^{-1}(z) + \flat(\sum_{j=1}^r \beta_j \u_j^{\otimes d}) \bigg]^{-1} \\
&= \frac{1}{rN}\Tr \Q(z) - \frac{1}{rN}\Tr \bigg[\Q(z) (\sum_{i=1}^r \pmb{U}_i \A_i \pmb{U}_i^{\top}) \bigl(\Q^{-1}(z) + \sum_{i=1}^r \pmb{U}_i \A_i \pmb{U}_i^{\top} \bigr)^{-1}\bigg] \\
&= \frac{1}{rN}\Tr \Q(z) - \sum_{i=1}^r \frac{1}{rN}\Tr \bigg[\A_i \pmb{U}_i^{\top} \pmb{R}(z) \Q(z) \pmb{U}_i\bigg],
\end{split}
\]
where we denote
\[
\pmb{U}_i = \begin{bmatrix}
\u_i & 0 & 0 \\
0 & \ddots & 0\\
0 & 0 & \u_i
\end{bmatrix} \text{\quad and \quad} \A_i = \beta_i \pmb{\hat W}_r \begin{bmatrix}
\langle \u_i, \hat\v_1 \rangle^{d-2} & 0 & 0 \\
0 & \ddots & 0\\
0 & 0 & \langle \u_i, \hat\v_r\rangle^{d-2}
\end{bmatrix}.
\]
Notice that each matrix
\[
\A_i \pmb{U}_i^{\top} \pmb{R}(z) \Q(z) \pmb{U}_i
\]
is a $r\times r$ matrix, and has bounded spectral norm thanks to Assumption~\ref{assump:general} and the fact that $\Q(z)$ and $\pmb{R}(z)$ are of bounded spectral norms. This implies that
\[
\frac{1}{rN}\Tr \bigg[\A_i \pmb{U}_i^{\top} \pmb{R}(z) \Q(z) \pmb{U}_i\bigg] \xrightarrow{N \to \infty} 0
\]
for all $i \in [r]$. Thus
\[
\lim_{N\to \infty} \frac{1}{rN} \Tr \pmb{R}(z) = \lim_{N\to \infty} \frac{1}{rN} \Tr \Q(z).
\]

%====================================================================================
%====================== Explicit mu =================================================
%====================================================================================

\section{Proof of Theorem~\ref{th:explicitmu_noise}}
\label{subsec:proofexplicitmu}
From Theorem~\ref{thm:pure_noise_Stieltjes}, the limit of the Stieltjes transform $S_{\mu_{\flat(\sT)}} (z) $ of the empirical spectral measure of $\pmb{\flat}(\sT)$ is equal to 
\begin{align}
\frac1r\sum_{i=1}^r g_{ii}(z) = \frac1r\operatorname{Tr}\pmb{G}(z) = \frac2r \sum_{i=1}^r \frac{-z+\sqrt{z^2-\kappa_i^2}}{\kappa_i^2},
\end{align}
when $\kappa_i\neq 0$. By the inverse formula of each term of the summand, we have
\begin{equation}
\mu_{\kappa_i}(dx) = \frac{2}{\pi\kappa_i^2}\sqrt{\bigg(\kappa_i^2 - x^2\bigg)_+}dx .
\end{equation}
We compute similarly the inverse Stieltjes transform in the case where $\kappa_i = 0$.

%======================================================================================
% Spectrum corollaries
%======================================================================================

\section{Proof of Theorem~\ref{thm:mostgeneral}}\label{subsec:proofmeasure_orth}

We will follow the same lines of proof as in Section~\ref{sec:prooftotal}.
For the sake of clarity of the exposition, we provide details in the case $s=1$, focusing only on $\hat\v_1$ and $\hat\gamma_1$.
Therefore, we consider the following two fixed point equations that the critical points need to satisfy
\begin{equation}\label{eq:cripteq_1}
\begin{dcases}
\langle \sT, \hat\v_1^{\otimes (d-1)}\rangle = \hat\gamma_1 \hat\v_1\\
\langle \hat\v_1, \hat\v_1 \rangle = 1
\end{dcases}
\end{equation}
We can do this choice because we assumed that $\hat\v_1$ is orthogonal to $\hat\v_2,\dots,\hat\v_r$, hence we can separate the fixed point equations in equations involving $\hat\v_1$ (see \eqref{eq:cripteq_1}) and the equations involving $\hat\v_2,\dots,\hat\v_r$.
Denote now by $X_{\alpha} = X_{i_1\dots i_d}$, and let
\[
\sF_1(\hat\v_1, \hat\gamma_1, \hat\v_2, \hat\gamma_2, \dots, \hat\v_r, \hat\gamma_r) = \begin{bmatrix}
\langle \sT, \hat\v_1^{\otimes (d-1)}\rangle - \hat\gamma_1 \hat\v_1\\
\langle \hat\v_1, \hat\v_1 \rangle - 1
\end{bmatrix}.
\]
Then the Jacobian of $\sF_1$ with respect to $(\hat\v_1,  \hat\gamma_1)$ is given by
\[
\scalemath{0.93}{\pmb{J}_1 = \begin{bmatrix}
(d-1)\langle \sT, \hat\v_1^{\otimes (d-2)}\rangle - \hat\gamma_1\pId & -\hat\v_1 \\
2\hat\v_1^{\top} & 0 
\end{bmatrix}}.
\]

Therefore, the Jacobian $\pmb{J}_1$ is invertible since $\frac{\hat\gamma_1}{d-1}$ is assumed to be large enough. The implicit function theorem then provides
\begin{equation}\label{eq:implicit_1}
\left(\begin{matrix}
\frac{\partial \hat\v^{\top}_1}{\partial X_{\alpha}} &
\frac{\partial \hat\gamma_1}{\partial X_{\alpha}}
\end{matrix}
\right)^{\top} = - \pmb{J}_1^{-1}\, \frac{\partial \sF_1}{\partial X_{\alpha}}.
\end{equation}
Noticing
\begin{equation}\label{eq:C_1}
\pmb{C}_1 = (d-1)\langle \sT, \hat\v_1^{\otimes (d-2)}\rangle - \hat\gamma_1\pId,
\end{equation}
we can explicitly write out the inverse of $\pmb{J}_1$
\begin{equation}
\pmb{J}_1^{-1} = \begin{bmatrix}
\pmb{C}_1^{-1} + \frac{\pmb{C}_1^{-1}\hat\v_1\hat\v_1^{\top}\pmb{C}_1^{-1}}{\hat\v_1^{\top}\pmb{C}_1^{-1}\hat\v_1} & \frac{\pmb{C}_1^{-1}\hat\v_1}{2\hat\v_1^{\top}\pmb{C}_1^{-1}\hat\v_1} \\
-\frac{\hat\v_1^{\top}\pmb{C}_1^{-1}}{\hat\v_1^{\top}\pmb{C}_1^{-1}\hat\v_1} & \frac{1}{2\hat\v_1^{\top}\pmb{C}_1^{-1}\hat\v_1} 
\end{bmatrix}.
\end{equation}
In particular, from \eqref{eq:implicit_1}, we get
\begin{equation}\label{eq:partialv_1}
\frac{\partial \hat\v_1}{\partial X_{\alpha}} = -\left(\pmb{C}_1^{-1} + \frac{\pmb{C}_1^{-1}\hat\v_1\hat\v_1^{\top}\pmb{C}_1^{-1}}{\hat\v_1^{\top}\pmb{C}_1^{-1}\hat\v_1}\right)\frac{\partial \langle \sT,\hat\v_1^{\otimes(d-1)}\rangle_j }{\partial X_{\alpha}}.
\end{equation}

Now recall that $\A = \langle \sT, \hat\v_1^{\otimes (d-2)} \rangle$ and $\pmb{Q}_1 = (\A-z\pId)^{-1}$. Thus, 
\begin{equation}\label{eq:small_trace}
\frac1N \Tr(\pmb{Q}_1\A) - \frac{z}{N} \Tr(\pmb{Q}_1) = 1.
\end{equation}
Let us denote $g_{1}(z)=\frac1N \Tr(\pmb{Q}_1)$. We are now interested in
\begin{equation}
\begin{split}
&\frac{1}{N} \E\bigl[\sum_{i,j=1}^N Q^{1}_{ij}A_{ji}\bigr] \\
= &\frac{1}{N\sqrt{N}} \sum_{i, j, \ell_1, \dots, \ell_{d-2}=1}^N \E\bigl[Q^{1}_{ij} X_{\ell_1\dots \ell_{d-2}ji} \hat{v}^1_{\ell_1}\cdots \hat{v}^1_{\ell_{d-2}}\bigr] \\
= &\!\begin{multlined}[t]
\underbrace{\frac{1}{N\sqrt{N}} \sum_{i, j, \ell_1, \dots, \ell_{d-2}=1}^N \sigma^2_{\ell_1\dots\ell_{d-2}ji}\E \bigl[\hat{v}^1_{\ell_1}\cdots \hat{v}^1_{\ell_{d-2}} \frac{\partial Q^{1}_{ij}}{\partial X_{\ell_1\dots \ell_{d-2}ji}}\bigr]}_{\sA_1} \\
+ \underbrace{\frac{1}{N\sqrt{N}} \sum_{i, j, \ell_1, \dots, \ell_{d-2}=1}^N \sum_{k=1}^{d-2} \sigma^2_{\ell_1\dots\ell_{d-2}ji}\E \bigl[\hat{v}^1_{\ell_1}\cdots \hat{v}^1_{\ell_{k-1}} \frac{\partial \hat{v}^{1}_{\ell_k}}{\partial X_{\ell_1\dots \ell_{d-2}ji}} \hat{v}^1_{\ell_{k+1}}\cdots \hat{v}^1_{\ell_{d-2}}\, Q^{1}_{ij}\bigr]}_{\sA_2}.
\end{multlined}
\end{split}
\end{equation}
Following the same lines of proof in Section~\ref{sec:prooftotal}, we can prove that $\sA_2$ converges to $0$ using the fact that the spectral norm of $\pmb{C}_1^{-1} + \frac{\pmb{C}_1^{-1}\hat\v_1\hat\v_1^{\top}\pmb{C}_1^{-1}}{\hat\v_1^{\top}\pmb{C}_1^{-1}\hat\v_1}$ in \eqref{eq:partialv_1} is bounded. Following again Section~\ref{sec:prooftotal}, we can prove that 
\begin{equation}
\sA_1\xrightarrow{N \to \infty} - \frac{g_1(z)^2}{d(d-1)}.
\end{equation}
Therefore considering the limit when $N$ tends to infinity in Equation~\eqref{eq:small_trace} yields
\begin{equation}
- \frac{g_1(z)^2}{d(d-1)} - zg_{1}(z) = 1.
\end{equation}
Since $g_1(z)$ is a Stieltjes transform behaving as $O(z^{-1})$ for large $z$, we can solve this polynomial equation
\begin{equation}
g_1(z) = \frac{d(d-1)}{2}\left(-z+\sqrt{z^2-\frac{4}{d(d-1)}}\right).
\end{equation}
Considering the inverse transform provides the result.

%======================================================================================
% Proof of limiting alignements
%======================================================================================

%====================================================================================
%====================== General alignments =================================================
%====================================================================================

\section{Proof of Theorem~\ref{thm:generalized_align}}\label{sec:generalized_align}
Consider a sequence of critical points, i.e., satisfying \eqref{eq:rkrapprox}. Thus, thanks to Assumption~\ref{assump:general0}, we have
\begin{equation}\label{eq:u1v1}
\langle \sT, \hat\v_1^{\otimes (d-1)} \u_1\rangle = \sum_{t=1}^s \hat\gamma_t \langle \hat\v_t, \u_1 \rangle \langle \hat\v_t, \v_1 \rangle^{d-1} .
\end{equation}

Since
\begin{equation}\label{eq:alignTu1}
\langle \sT, \hat\v_1^{\otimes (d-1)} \u_1\rangle = \sum_{t=1}^r \beta_t \langle \u_1, \u_t\rangle \langle \u_t, \hat\v_1\rangle^{d-1} + \frac{1}{\sqrt{N}}\langle \mathcal{X}, \hat\v_1^{\otimes (d-1)} \u_1\rangle \enspace,
\end{equation}
to characterize the alignment $\langle \hat\v_1, \u_1 \rangle$ we need to evaluate $\langle \sX, \hat\v_1^{\otimes (d-1)} \u_1\rangle$. By~\eqref{eq:partialvD} and~\eqref{eq:Ddef},
\begin{equation}\label{eq:alignu1}
\begin{split}
\frac{1}{\sqrt{N}}\E\bigg[\langle \sX, \hat\v_1^{\otimes (d-1)} \u_1\rangle\bigg] 
&= \frac{1}{\sqrt{N}} \sum_{i_1,\dots,i_d} \E \bigg[X_{i_1\dots i_d} \hat{v}^1_{i_1} \cdots \hat{v}^1_{i_{d-1}}\bigg] u^1_{i_d} \\
&= \frac{d-1}{\sqrt{N}} \sum_{i_1,\dots,i_d} \sigma^2_{i_1\dots i_d}\E \bigg[\frac{\partial \hat{v}^1_{i_1}}{\partial X_{i_1\dots i_d}} \hat{v}^1_{i_2}\cdots \hat{v}^1_{i_{d-1}}\bigg] u^1_{i_d} \\
&= - \frac{d-1}{\sqrt{N}} \sum_{i_1, \dots, i_d, j} \sigma^2_{i_1\dots i_d} \E \bigg[\sum_{s=1}^r D^{1s}_{i_1j} \frac{\partial \langle \sT, \hat\v_s^{\otimes(d-1)}\rangle_j }{\partial X_{\alpha}} \hat{v}^1_{i_2}\cdots \hat{v}^1_{i_{d-1}}\bigg] u^1_{i_d} \enspace.
\end{split}
\end{equation}
with the notations of \eqref{eq:Ddef}. For the purpose of computing~\eqref{eq:alignu1}, we start from the term
\begin{equation}\label{eq:D11w1}
\frac{d-1}{\sqrt{N}} \sum_{i_1, \dots, i_d, j} \sigma^2_{i_1\dots i_d} \E \bigg[D^{11}_{i_1j}\frac{\partial \langle \sT, \hat\v_1^{\otimes(d-1)}\rangle_j }{\partial X_{\alpha}}\hat{v}^1_{i_2}\cdots \hat{v}^1_{i_{d-1}}\bigg]u^1_{i_d}.
\end{equation}
To this end, thanks to~\eqref{eq:defC-1}, let us decompose $\pmb{D}_{11}$ as
\begin{equation}\label{eq:C-111}
\pmb{D}_{11} = \frac{1}{d-1} \left(\Q(\frac{\hat\gamma_s}{d-1})\pmb{\hat W}\right)_{11} + \pmb{M},
\end{equation}
where $\left(\Q(\frac{\hat\gamma_s}{d-1})\pmb{\hat W}\right)_{11}$ denotes the top-left $N\times N$ block of the matrix $\Q(\frac{\hat\gamma_s}{d-1})\pmb{\hat W}$ and  $\pmb{M}$ is a finite rank matrix of bounded spectral norm similarly as in Section~\ref{sec:A1}.
Therefore
\[
\begin{split}
\lim_{N\to \infty}\frac{\Tr(\pmb{D}_{11})}{N} &= \lim_{N\to\infty} \frac{1}{N(d-1)}\Tr\left(\left(\Q(\frac{\hat\gamma_s}{d-1})\pmb{\hat W}\right)_{11}\right)\\
&= \frac{1}{d-1}\bigg(\pmb{G}(\frac{\gamma_s}{d-1})\pmb{W}\bigg)_{11}
\end{split}
\]

Next we compute $\phi_j:=\frac{\partial \langle \sT, \hat\v_1^{\otimes(d-1)}\rangle_j }{\partial X_{\alpha}}$. Assume $\{i_1, \dots, i_{d-1}\} = \{\ell_1, \dots, \ell_1, \dots, \ell_k, \dots, \ell_k\}$ as a multiset, where the multiplicity of $\ell_j$ is $m_j$. By~\eqref{eq:multTvd-1gal}, $\phi_j\neq 0$ only when $j\in\{i_1, \dots, i_d\}$. Then, its calculation is splitted into the following two cases.
\begin{itemize}
\item When $i_d \in \{\ell_1, \dots, \ell_k\}$, without loss of generality, we assume $i_d = \ell_1$. Then
\begin{equation*}
\phi_{\ell_1}= \frac{1}{\sqrt{N}}\binom{d-1}{m_1, \dots, m_k} \bigl(\hat{v}^1_{\ell_1}\bigr)^{m_1} \cdots \bigl(\hat{v}^1_{\ell_k}\bigr)^{m_k} \enspace,
\end{equation*}
and when $j \in \{2, \dots, k\}$,
\begin{multline*}
\phi_{\ell_j} = \frac{1}{\sqrt{N}}\binom{d-1}{m_1+1, m_2, \dots, m_{j-1}, m_j-1, m_{j+1}, \dots, m_k} \\
\bigl(\hat{v}^1_{\ell_1}\bigr)^{m_1+1} \bigl(\hat{v}^1_{\ell_2}\bigr)^{m_2} \cdots \bigl(\hat{v}^1_{\ell_{j-1}}\bigr)^{m_{j-1}} \bigl(\hat{v}^1_{\ell_j}\bigr)^{m_j-1} \bigl(\hat{v}^1_{\ell_{j+1}}\bigr)^{m_{j+1}} \cdots \bigl(\hat{v}^1_{\ell_k}\bigr)^{m_k} \enspace.
\end{multline*}
\item When $i_d \notin \{\ell_1, \dots, \ell_k\}$,
\begin{equation*}
\phi_{i_d}= \frac{1}{\sqrt{N}}\binom{d-1}{m_1, \dots, m_k} \bigl(\hat{v}^1_{\ell_1}\bigr)^{m_1} \cdots \bigl(\hat{v}^1_{\ell_k}\bigr)^{m_k} \enspace,
\end{equation*}
and when $j \in [k]$,
\begin{multline*}
\phi_{\ell_j} = \frac{1}{\sqrt{N}}\binom{d-1}{m_1, \dots, m_{j-1}, m_j-1, m_{j+1}, \dots, m_k, 1} \\
\bigl(\hat{v}^1_{\ell_1}\bigr)^{m_1} \cdots \bigl(\hat{v}^1_{\ell_{j-1}}\bigr)^{m_{j-1}} \bigl(\hat{v}^1_{\ell_j}\bigr)^{m_j-1} \bigl(\hat{v}^1_{\ell_{j+1}}\bigr)^{m_{j+1}} \cdots \bigl(\hat{v}^1_{\ell_k}\bigr)^{m_k} \bigl(\hat{v}^1_{i_d}\bigr) \enspace.
\end{multline*}
\end{itemize}
Consequently we split the computation of~\eqref{eq:D11w1} into the following two cases.
\begin{itemize}
\item When $i_d \in \{\ell_1, \dots, \ell_k\}$, without loss of generality, we assume $i_d = \ell_1$. Then
\begin{equation}\label{eq:D11est}
\resizebox{0.92\textwidth}{!}{%
$\begin{aligned}
&\frac{d-1}{\sqrt{N}} \sum_{i_1, \dots, i_d, j} \sigma^2_{i_1\dots i_d} \E \bigg[D^{11}_{i_1j}\phi_{j} \hat{v}^1_{i_2}\cdots \hat{v}^1_{i_{d-1}}\bigg]u^1_{i_d} \\
= &\frac{1}{\sqrt{N}} \sum_{\ell_1, \dots, \ell_k, j, q} m_j \binom{d-1}{m_1, \dots, m_k} \sigma^2_{i_1\dots i_d} \E \bigg[D^{11}_{\ell_jq}\phi_q \bigl(\hat{v}^1_{\ell_1}\bigr)^{m_1} \cdots \bigl(\hat{v}^1_{\ell_{j-1}}\bigr)^{m_{j-1}} \bigl(\hat{v}^1_{\ell_j}\bigr)^{m_j-1} \bigl(\hat{v}^1_{\ell_{j+1}}\bigr)^{m_{j+1}} \cdots \bigl(\hat{v}^1_{\ell_k}\bigr)^{m_k}\bigg]u^1_{\ell_1}.
\end{aligned}$
}
\end{equation}
For simplicity of notation, here we only focus on the following term contained in~\eqref{eq:D11est}
\begin{equation}\label{eq:D11w1aell}
\frac{1}{\sqrt{N}} \sum_{\ell_1, \dots, \ell_k, q}m_1\binom{d-1}{m_1, \dots, m_k}\sigma^2_{i_1\dots i_d} \E \bigg[D^{11}_{\ell_1q}\phi_q \bigl(\hat{v}^1_{\ell_1}\bigr)^{m_1-1} \bigl(\hat{v}^1_{\ell_{2}}\bigr)^{m_{2}} \cdots \bigl(\hat{v}^1_{\ell_k}\bigr)^{m_k}\bigg]u^1_{\ell_1} \enspace.
\end{equation}
Thanks to the expression of $\phi_q$, Equation~\eqref{eq:D11w1aell} has the form
\begin{multline}\label{eq:D11w1aellest}
\frac{1}{N} \sum_{\ell_1, \dots, \ell_k} \frac{m_1(m_1+1)}{d} u^1_{\ell_1}\E \bigg[\binom{d-1}{m_1, \dots, m_k} D^{11}_{\ell_1\ell_1} \bigl(\hat{v}^1_{\ell_1}\bigr)^{2m_1-1} \bigl(\hat{v}^1_{\ell_{2}}\bigr)^{2m_{2}} \cdots \bigl(\hat{v}^1_{\ell_k}\bigr)^{2m_k} \\
+ \cdots + \binom{d-1}{m_1+1, \dots, m_k-1} D^{11}_{\ell_1\ell_k} \bigl(\hat{v}^1_{\ell_1}\bigr)^{2m_1} \cdots \bigl(\hat{v}^1_{\ell_{k-1}}\bigr)^{2m_{k-1}} \bigl(\hat{v}^1_{\ell_k} \bigr)^{2m_k-1}\bigg] \enspace.
\end{multline}
Since $\D_{11}$ has bounded spectral norm and
\[
\sum_{\ell_1} u^1_{\ell_1} D^{11}_{\ell_1\ell_1} \hat{v}^1_{\ell_1} \le \sqrt{N}\Vert \D_{11}\Vert_2 \quad \text{and} \quad \sum_{\ell_1, \ell_k} u^1_{\ell_1} D^{11}_{\ell_1\ell_k} \hat{v}^1_{\ell_k} \le \Vert \D_{11}\Vert_2 \enspace,
\]
Equation~\eqref{eq:D11w1aellest} is bounded by
\[
\frac{(d!)^2}{N}  \bigl(\sqrt{N} + 1\bigr) \Vert \D_{11}\Vert_2 \xrightarrow{N \to \infty} 0 \enspace.
\]
Hence in a similar way we have that Equation~\eqref{eq:D11est} tends to $0$ when $N \to \infty$.
\item When $i_d \notin \{\ell_1, \dots, \ell_k\}$,
\begin{equation}\label{eq:D11idest}
\resizebox{0.91\textwidth}{!}{%
$\begin{aligned}
&\frac{d-1}{\sqrt{N}} \sum_{i_1, \dots, i_d, j}\sigma^2_{i_1\dots i_d} \E \bigg[D^{11}_{i_1j}\phi^1_j \hat{v}^1_{i_2}\cdots \hat{v}^1_{i_{d-1}}\bigg]u^1_{i_d} \\
= &\frac{1}{\sqrt{N}} \sum_{\ell_1, \dots, \ell_k, i_d, j, q}m_j\binom{d-1}{m_1, \dots, m_k}\sigma^2_{i_1\dots i_d} \E \bigg[D^{11}_{\ell_jq}\phi^1_q \bigl(\hat{v}^1_{\ell_1}\bigr)^{m_1} \cdots \bigl(\hat{v}^1_{\ell_{j-1}}\bigr)^{m_{j-1}} \bigl(\hat{v}^1_{\ell_j}\bigr)^{m_j-1} \bigl(\hat{v}^1_{\ell_{j+1}}\bigr)^{m_{j+1}} \cdots \bigl(\hat{v}^1_{\ell_k}\bigr)^{m_k}\bigg]u^1_{i_d}.
\end{aligned}$
}
\end{equation}
As before, we focus on the term
\begin{equation}\label{eq:D11w1idell}
\begin{split}
&\frac{1}{\sqrt{N}} \sum_{\ell_2, \dots, \ell_k, i_d, q}m_1\binom{d-1}{m_1, \dots, m_k}\sigma^2_{i_1\dots i_d} \E \bigg[D^{11}_{\ell_1q}\phi^1_q \bigl(\hat{v}^1_{\ell_1}\bigr)^{m_1-1} \bigl(\hat{v}^1_{\ell_{2}}\bigr)^{m_{2}} \cdots \bigl(\hat{v}^1_{\ell_k}\bigr)^{m_k}\bigg]u^1_{i_d} \\
= &\frac{1}{N} \sum_{\ell_2, \dots, \ell_k, i_d} \frac{m_1}{d} \E \bigg[\binom{d-1}{m_1-1, m_2, \dots, m_k, 1} D^{11}_{\ell_1\ell_1} \bigl(\hat{v}^1_{\ell_1}\bigr)^{2m_1-2} \bigl(\hat{v}^1_{\ell_{2}}\bigr)^{2m_{2}} \cdots \bigl(\hat{v}^1_{\ell_k}\bigr)^{2m_k} \bigl(\hat{v}^1_{i_d}\bigr) \\
+& \cdots + \binom{d-1}{m_1, \dots, m_k-1, 1} D^{11}_{\ell_1\ell_k} \bigl(\hat{v}^1_{\ell_1}\bigr)^{2m_1-1} \bigl(\hat{v}^1_{\ell_2}\bigr)^{2m_2} \cdots \bigl(\hat{v}^1_{\ell_{k-1}}\bigr)^{2m_{k-1}} \bigl(\hat{v}^1_{\ell_k} \bigr)^{2m_k-1} \bigl(\hat{v}^1_{i_d}\bigr) \\
+& \binom{d-1}{m_1, \dots, m_k} D^{11}_{\ell_1 i_d} \bigl(\hat{v}^1_{\ell_1}\bigr)^{2m_1-1} \bigl(\hat{v}^1_{\ell_2}\bigr)^{2m_2} \cdots \bigl(\hat{v}^1_{\ell_k} \bigr)^{2m_k} \bigg]u^1_{i_d} \enspace.
\end{split}
\end{equation}
Arguing as above, we see that the only term in~\eqref{eq:D11w1idell}, which can be nonzero as $N \to \infty$, is
\[
\frac{1}{N} \sum_{\ell_2, \dots, \ell_k, i_d} \frac{m_1}{d} \E \bigg[\binom{d-1}{m_1-1, m_2, \dots, m_k, 1} D^{11}_{\ell_1\ell_1} \bigl(\hat{v}^1_{\ell_1}\bigr)^{2m_1-2} \bigl(\hat{v}^1_{\ell_{2}}\bigr)^{2m_{2}} \cdots \bigl(\hat{v}^1_{\ell_k}\bigr)^{2m_k} \hat{v}^1_{i_d} \bigg]u^1_{i_d}
\]
when $m_1 = 1$, i.e.,
\[
\begin{split}
&\frac{1}{dN} \sum_{\ell_2, \dots, \ell_k, i_d} \binom{d-1}{m_2, \dots, m_k, 1} \E \bigg[D^{11}_{\ell_1\ell_1} \bigl(\hat{v}^1_{\ell_{2}}\bigr)^{2m_{2}} \cdots \bigl(\hat{v}^1_{\ell_k}\bigr)^{2m_k} \hat{v}^1_{i_d}\bigg]u^1_{i_d} \\
= &\frac{d-1}{dN} \sum_{i_2, \dots, i_d} \E \bigg[D^{11}_{\ell_1 \ell_1} \bigl(\hat{v}^1_{i_{2}}\bigr)^{2} \cdots \bigl(\hat{v}^1_{i_{d-1}}\bigr)^{2} \hat{v}^1_{i_d} \bigg]u^1_{i_d} \\
= &\frac{d-1}{dN} \E\bigg[D^{11}_{\ell_1\ell_1} \langle \u_1, \hat\v_1\rangle\bigg] \enspace.
\end{split}
\]
Thus~\eqref{eq:D11idest} converges a.s. to
\[
\frac{d-1}{d} \lim_{N\to\infty}\E\bigg[\frac{\Tr(\D_{11})}{N} \langle \u_1, \hat\v_1\rangle\bigg] = \frac{R^{uv, s}_{11}}{d}\bigg(\pmb{G}(\frac{\gamma_s}{d-1})\pmb{W}\bigg)_{11}.
\]
\end{itemize}
Similarly we can show that for any $1\leq t\leq s$,
\[
\begin{split}
&\frac{d-1}{\sqrt{N}} \sum_{i_1, \dots, i_d, j} \sigma^2_{i_1\dots i_d} \E \bigg[\bigl(D^{1t}_{i_1j}\phi^s_j \bigr)\hat{v}^1_{i_2}\cdots \hat{v}^1_{i_{d-1}}\bigg] u^1_{i_d} \\
\xrightarrow{N \to \infty} &\frac{d-1}{d} \lim_{N\to\infty}\E\bigg[\langle \hat\v_1, \hat\v_t\rangle^{d-2}\langle \u_1, \hat\v_t\rangle \frac{\Tr\bigl(\pmb{D}_{1t}\bigr)}{N}\bigg] \\
\xrightarrow{N \to \infty} &\frac{R^{uv, s}_{1t}(R^{vv, s}_{1t})^{d-2}}{d} \bigg(\pmb{G}(\frac{\gamma_t}{d-1})\pmb{W}\bigg)_{t1} \enspace.
\end{split}
\]
Hence Equation~\eqref{eq:alignTu1} gives us
\begin{equation}
\sum_{t=1}^s \gamma_t  R^{uv, s}_{1t}(R^{vv, s}_{t1})^{d-1} =\sum_{t=1}^r \beta_t (R^{uv, s}_{1t})^{d-1}R^{uu, r}_{t1} - \sum_{t=1}^s \frac{R^{uv, s}_{1t}(R^{vv, s}_{t1})^{d-2}}{d} \bigg(\pmb{G}(\frac{\gamma_r}{d-1})\pmb{W}\bigg)_{t1}.
\end{equation}
which can be rewritten as
\begin{equation}
\big(\pmb{R}_{uu, r} \pmb{D}_{\beta}  \pmb{R}_{uv, s}^{\odot (d-1)}\big)_{11} = \bigg(\pmb{R}_{uv, s}\bigg( \pmb{D}_{\gamma}\pmb{R}_{vv, s}^{\odot (d-1)}+  \frac{1}{d}\pmb{R}_{vv, s}^{\odot (d-2)}\odot\left(\pmb{G}\left(\frac{\gamma_r}{d-1}\right)\pmb{W}\right)\bigg) \bigg)_{11}.
\end{equation}

Similarly,  for any $1\leq t_1\leq s$ and $1\leq t_2\leq r$, computing the convergence of 
\[
\langle \sT, \hat\v_{t_2}^{\otimes (d-1)} \u_{t_1}\rangle = \sum_{u=1}^s \hat\gamma_u \langle \hat\v_{t_2}, \hat\v_u \rangle^{d-1}  \langle \u_{t_1}, \hat\v_u \rangle 
\]
gives us the first matrix equation in \eqref{eq:bigsystems}:
\begin{equation}
\big(\pmb{R}_{uu, r} \pmb{D}_{\beta}  \pmb{R}_{uv, s}^{\odot (d-1)}\big)_{t_1t_2} = \bigg(\pmb{R}_{uv, s}\bigg( \pmb{D}_{\gamma}\pmb{R}_{vv, s}^{\odot (d-1)}+  \frac{1}{d}\pmb{R}_{vv, s}^{\odot (d-2)}\odot\left(\pmb{G}\left(\frac{\gamma_r}{d-1}\right)\pmb{W}\right)\bigg) \bigg)_{t_1t_2}.
\end{equation}

In a similar way we can show that for $1\leq t \leq s$,
\begin{equation*}
\E\langle \frac{\sX}{\sqrt{N}}, \hat\v_t^{\otimes d}\rangle \xrightarrow{N \to \infty} \lim_{N\to\infty} \sum_{u=1}^r (R^{vv, s}_{tu})^{d-1}\E\bigg[\frac{\Tr(\pmb{D}_{tu})}{N}\bigg],
\end{equation*}
additionally for $1\leq t_1,t_2\leq s$,
\begin{equation*}
\E\langle \frac{\sX}{\sqrt{N}}, \hat\v_{t_1}^{\otimes (d-1)}\hat\v_{t_2}\rangle \xrightarrow{N \to \infty} \lim_{N\to\infty} \sum_{u=1}^s  \bigg[\frac{d-1}{d} R^{vv, s}_{ut_2} (R^{vv, s}_{ut_1})^{d-2} \E\bigg[\frac{\Tr(\pmb{D}_{t_1u})}{N} \bigg] + \frac{1}{d} (R^{vv, s}_{t_1u})^{d-1} \E\bigg[\frac{\Tr(\pmb{D}_{t_2u})}{N}\bigg]\bigg]\enspace.
\end{equation*}

On the other hand, from~\eqref{eq:rkrapprox}, we have
\begin{equation*}
\langle \sT, \hat\v_{t_1}^{\otimes (d-1)}\hat\v_{t_2}\rangle = \sum_{u=1}^s \hat\gamma_u \langle \hat\v_u, \hat\v_{t_2}\rangle \langle \hat\v_{t_1}, \hat\v_u\rangle^{d-1}
\end{equation*}
which gives us the following equations
\begin{multline}
\sum_{u=1}^s \gamma_u R^{vv, s}_{t_2u}(R^{vv, s}_{ut_1})^{d-1} = \sum_{u=1}^r \beta_u R^{uv, s}_{ut_2}(R^{uv, s}_{ut_1})^{d-1}\\
 -  \sum_{u=1}^s  \bigg[\frac{1}{d} R^{vv, s}_{t_2u} (R^{vv, s}_{ut_1})^{d-2} \bigg(\pmb{G}(\frac{\gamma_s}{d-1})\pmb{W}\bigg)_{ut_1} + \frac{1}{d(d-1)} (R^{vv, s}_{t_1u})^{d-1} \bigg(\pmb{G}(\frac{\gamma_s}{d-1})\pmb{W}\bigg)_{ut_2}\bigg].
\end{multline}
Combining all these equations in a matrix form provides the second equation in \eqref{eq:bigsystems}.

%\section{Proof of Theorem~\ref{thm:orthogonal}}\label{subsec:orthogonal}
%We follow the same line of proofs as in Section~\ref{subsec:proofalignsequ} replacing $\pmb{C}$ by $\pmb{C}_1$ defined in \eqref{eq:C_1} and remarking that $\pmb{C}_1$ is a rank-1 deformation of $\pmb{Q}_1(\frac{\hat\gamma_1}{d-1})$. Taking the expectation of the three following equations,
%\begin{align}\label{eq:alignTuThree}
%\langle \sT, \hat\v_1^{\otimes (d-1)} \u_1\rangle = \beta_1 \langle \u_1, \hat\v_1\rangle^{d-1} + \beta_2\langle \u_1, \u_2\rangle \langle \u_2, \hat\v_1\rangle^{d-1} + \frac{1}{\sqrt{N}}\langle \mathcal{X}, \hat\v_1^{\otimes (d-1)} \u_1\rangle \enspace,\\
%\langle \sT, \hat\v_1^{\otimes (d-1)} \u_2\rangle = \beta_1 \langle \u_1, \u_2\rangle\langle \u_1, \hat\v_1\rangle^{d-1} + \beta_2 \langle \u_2, \hat\v_1\rangle^{d-1} + \frac{1}{\sqrt{N}}\langle \mathcal{X}, \hat\v_1^{\otimes (d-1)} \u_2\rangle \enspace,\\
%\langle \sT, \hat\v_1^{\otimes d} \rangle = \beta_1 \langle \u_1, \hat\v_1\rangle^{d} + \beta_2 \langle \u_2, \hat\v_1\rangle^{d} + \frac{1}{\sqrt{N}}\langle \mathcal{X}, \hat\v_1^{\otimes d} \rangle \enspace,
%\end{align}
%we get  Equation~\eqref{eq:bigsystem_orth}.

\section{Proof of Lemma~\ref{lemma:likelihood}}\label{sec:likelihood}
For the sake of light notations, let us note $ \mathcal{H} =\mathcal{H}(\hat\gamma_1, \dots,\hat\gamma_r, \hat\v_1, \dots,\hat\v_r) $. Then,
\begin{equation}\label{eq:23}
 \mathcal{H} = \left\| \sT \right\|^2 + \sum_{i,j=1}^r \hat\gamma_i\hat\gamma_j (\langle \hat\v_i,\hat\v_j\rangle)^d - 2\sum_j \hat\gamma_j\langle \sT,\hat\v_j^{\otimes d}\rangle - \frac{1}{N}\|\sX\|^2.
\end{equation}
Using the fixed point equation \eqref{eq:rkrapprox}, we have, for any $1\leq i\leq r$,
\begin{equation}\label{eq:24}
\langle\sT,\hat\v_i^{\otimes d}\rangle =  \sum_{j=1}^r \hat\gamma_j(\langle \hat\v_i,\hat\v_j\rangle)^d.
%\left\{\begin{array}{l}
%\langle\sT,\hat\v_1^{\otimes d}\rangle =  \hat\gamma_1 + \hat\gamma_2(\langle \hat\v_1,\hat\v_2\rangle)^d\\
%\langle\sT,\hat\v_2^{\otimes d}\rangle =  \hat\gamma_2 + \hat\gamma_1(\langle \hat\v_1,\hat\v_2\rangle)^d
%\end{array}\right..
\end{equation}
Furthermore, since the true signals $\u_i$ and the noise tensor $\sX$ are independent, $\frac{1}{\sqrt{N}}\langle\sX,\u_i^{\otimes d}\rangle=o(1)$, which yields 
\begin{equation}\label{eq:25}
 \left\| \sT \right\|^2 = \sum_{i,j=1}^r \beta_i\beta_j \langle\u_i,\u_j \rangle^d + \frac{1}{N} \|\sX\|^2 + o(1).
\end{equation}
Substituting \eqref{eq:23} and \eqref{eq:24} in \eqref{eq:25} yields
\begin{equation}
 \mathcal{H} =  \begin{bmatrix}\beta_1\\ \vdots\\ \beta_r\end{bmatrix}^{\top}\pmb{R}_{uu,r}^{\odot d}\begin{bmatrix}\beta_1\\ \vdots\\ \beta_r\end{bmatrix} -  \begin{bmatrix}\gamma_1\\ \vdots\\ \gamma_r\end{bmatrix}^{\top}\pmb{R}_{vv,r}^{\odot d}\begin{bmatrix}\gamma_1\\ \vdots\\ \gamma_r\end{bmatrix}  + o(1).
\end{equation}

\section{Proof of Theorem~\ref{thm:alignmentbeta2}}\label{subsec:alignmentbeta}

% To successfully detect at least one critical point of~\eqref{lowrkopt}, we need detect at least the largest eigenvalue of $\flat(\sT)$, also denoted by $\hat\gamma_r$. 
Let us consider a sequence of critical points satisfying Assumption~\ref{assump:general0} and Assumption~\ref{assump:general2}, where $\hat{\gamma}_r$ is the largest eigenvalue.
Recall from~\eqref{eq:rk2deformation} that
\begin{equation}
\flat(\sT) = \flat(\sum_{j=1}^r \beta_j \u_j^{\otimes d}) + \frac{1}{\sqrt{N}}\flat(\sX) = \sum_{i=1}^r \pmb{U}_i \A_i \pmb{U}_i^{\top} + \frac{1}{\sqrt{N}}\flat(\sX),
\end{equation}
where
\[
\pmb{U}_i = \begin{bmatrix}
\u_i & 0 & 0 \\
0 & \ddots & 0\\
0 & 0 & \u_i
\end{bmatrix} \text{\quad and \quad} \A_i = \beta_i \pmb{\hat W} \begin{bmatrix}
\langle \u_i, \hat\v_1 \rangle^{d-2} & 0 & 0 \\
0 & \ddots & 0\\
0 & 0 & \langle \u_i, \hat\v_r\rangle^{d-2}
\end{bmatrix}.
\]
When $\beta_1 \ge \cdots \ge \lvert \beta_r \rvert$ are sufficiently large, the largest eigenvalue of $\flat(\sT)$, namely $\hat\gamma_r$, is greater than $(d-1)\kappa^{(r)}_1$, where $\kappa^{(r)}_1$ is the upper limit of the support of the empirical spectral measure of $\flat(\frac{1}{\sqrt{N}}\sX)$. Then the conclusion follows from Theorem~\ref{thm:generalized_align}. In theory we can follow the standard procedure to give a more precise bound for such $\beta_r$. More concretely, by solving Equation~\eqref{eq:bigsystems}, we can express $\hat\gamma_r$ in terms of $\pmb{R}_{uu,r}, \beta_1, \dots, \beta_r$, i.e., by elimination theory, $\hat\gamma_r$ is a solution of some univariate equation
\begin{equation}\label{eq:cribeta2}
\mathcal{G}(X, \pmb{R}_{uu,r}, \beta_1, \dots, \beta_r) = 0,
\end{equation}
where $X$ is the variable, and $\pmb{R}_{uu,r}, \beta_1, \dots, \beta_r$ are parameters. To find the critical value of $\beta_r$, let $\beta_1 = \cdots = \beta_r$. Then from Equation~\eqref{eq:cribeta2} we can express $\beta_r$ by a function of $X$ where $\pmb{R}_{uu,r}$ is a parameter, say
\[
\beta_r = \widetilde{\mathcal{G}}_{\pmb{R}_{uu,r}}(X).
\]
Thus we can take $\beta_{\operatorname{cri}}(\pmb{R}_{uu,r})$ as
\[
\beta_{\operatorname{cri}}(\pmb{R}_{uu,r}) = \lim_{X\to \kappa_1^{(r)}} \widetilde{\mathcal{G}}_{\pmb{R}_{uu,r}}(X) \enspace.
\]
Of course we can also solve $\beta_{\operatorname{cri}}(\pmb{R}_{uu,r})$ by taking the limit in Equation~\eqref{eq:cribeta2}, or more directly by taking the limit in Equation~\eqref{eq:bigsystems}. Unlike the classical matrix case, it is difficult to establish an explicit threshold via solving the polynomial system of equations~\eqref{eq:bigsystems}. %But the above procedure still confirms the theoretical existence of a threshold above which we can distinguish $\flat(\sT)$ and $\frac{1}{\sqrt{N}}\flat(\sX)$.

%======================================================================================
% Proof for explicit solution of limit alignements
%======================================================================================

\section{Solution of system of polynomial equations in the case $r=2$}\label{sec:proofmlecorrection}

In order to compute the plugin estimator, we need to compute the solutions of \eqref{eq:system_poly} which is a system of polynomial equations. 
In the rank-two case, we can provide a characterization of these solutions in terms of solutions of a single polynomial as follows.
\begin{theorem}\label{prop:eqr2d3}
Assume that $\hat\C$ defined in \eqref{eq:Lagde} is invertible. Note $\hat\C = \begin{bmatrix} c_{11} & c_{12} \\ c_{21} & c_{22} \end{bmatrix}$. and consider the polynomial equation
\begin{equation}\label{eq:polyc}
(c_{11}X - c_{21})(c_{21}X^{d-1} - c_{22})^{d-1} + (c_{22} -c_{12}X) (c_{11}X^{d-1} - c_{12})^{d-1} = 0.
\end{equation}
Then $\eta$ is a solution to \eqref{eq:polyc} if and only if $(\beta_1(\eta),\beta_2(\eta),\rho(\eta))$ is solution of \eqref{eq:system_poly} if the following expressions are well defined: 
 $\beta_1(\eta)=|\widetilde{K}_{11}|^{\frac{d}{2}}$ and $\beta_2(\eta)=|\widetilde{K}_{22}|^{\frac{d}{2}}\operatorname{sign}(\widetilde{K}_{12})$ and $\rho(\eta)=\frac{|\widetilde{K}_{12}|}{\sqrt{|\widetilde{K}_{11}\widetilde{K}_{22}}|}$ with 
\begin{equation}
\widetilde{\pmb{K}} = \begin{bmatrix} \widetilde{K}_{11} & \widetilde{K}_{12} \\ \widetilde{K}_{21} & \widetilde{K}_{22}\end{bmatrix} := \begin{bmatrix} \xi_1 & \xi_1\theta \\ \eta\xi_2 & \xi_2\end{bmatrix} \pmb{\hat M}\hat\C^{-1} \begin{bmatrix} \xi_1 & \xi_1\theta \\ \eta\xi_2 & \xi_2\end{bmatrix}^{\top} 
\end{equation}
 and
\begin{equation*}
\theta = \frac{c_{11}{\eta}^{d-1} - c_{12}}{c_{21}{\eta}^{d-1} - c_{22}}, \quad \xi_1 = \left(\begin{bmatrix} 1 & \eta^{d-1} \end{bmatrix} \hat\C^{-1} \begin{bmatrix} 1 & {\eta} \end{bmatrix}^{\top}\right)^{-\frac{1}{d}}, \quad \xi_2 = \left(\begin{bmatrix} \theta^{d-1} & 1 \end{bmatrix} \hat\C^{-1} \begin{bmatrix} \theta & 1 \end{bmatrix}^{\top}\right)^{-\frac{1}{d}}.
\end{equation*}
\end{theorem}
\begin{proof}

%\begin{proof}
Let $\hat\beta_1,\hat\beta_2,\hat\rho$ be solutions satisfying \eqref{eq:system_poly} and recall that $\hat\beta_1 > 0$.
To further simplify the system of equations~\eqref{eq:system_poly}, we note
For simplicity of notation, we let
\begin{equation}\label{eq:bigone2}
 \N = \begin{bmatrix} \hat\alpha_{11} & \hat\alpha_{12} \\ \hat\alpha_{21} & \hat\alpha_{22} \end{bmatrix}, \quad \pmb{L} = \begin{bmatrix} \hat\beta_1 & 0 \\ 0 & \hat\beta_2 \end{bmatrix}, \text{\quad and \quad}  \pmb{K} = \begin{bmatrix} 1 & \hat\rho \\ \hat\rho & 1\end{bmatrix}.
\end{equation}
Denote by $ \pmb{N}_d = \N\odot \dots \odot \N$ the $d$th iterated Hadamard product of $\N$, i.e., the $(i, j)$ entry of $\N_d$ is $\alpha_{ij}^d$.
We perform the following change of variables
\begin{equation}\label{eq:changeofvarNtilde}
\widetilde{\N} = \begin{cases}
\pmb{L}^{1/d} \N & \text{when } d \text{ is odd, or when } d \text{ is even and } \hat\beta_2 > 0 \\
\begin{bmatrix}
\hat\beta_1^{1/d} & 0 \\
0 & - (-\hat\beta_2)^{1/d}
\end{bmatrix} \N & \text{when } d \text{ is even and } \beta_2 < 0
\end{cases}, 
\end{equation}
and
\begin{equation}\label{eq:changofvarKtilde}
\widetilde{\pmb{K}} = \begin{cases}
\pmb{L}^{1/d} \pmb{K} \pmb{L}^{1/d} & \text{when } d \text{ is odd, or when } d \text{ is even and } \hat\beta_2 > 0 \\
\begin{bmatrix}
\hat\beta_1^{1/d} & 0 \\
0 & - (-\hat\beta_2)^{1/d}
\end{bmatrix} \pmb{K} \begin{bmatrix}
\hat\beta_1^{1/d} & 0 \\
0 & - (-\hat\beta_2)^{1/d}
\end{bmatrix} & \text{when } d \text{ is even and } \beta_2 < 0
\end{cases}.
\end{equation}
Then System~\eqref{eq:system_poly} becomes
\begin{equation}\label{eq:bigsystem3}
\left\{\begin{array}{l}
\widetilde{\pmb{K}} = \widetilde{\N}\hat\D \widetilde{\N}_{d-1}^{-1} \\
\widetilde{\N}^{\top} \widetilde{\N}_{d-1} = \hat\C
\end{array}\right..
\end{equation}
Since $\widetilde{\pmb{K}}$ is a function of $\widetilde{\N}$ from the first equation of~\eqref{eq:bigsystem3}, it suffices to solve
\begin{equation}\label{eq:matQ}
\widetilde{\N}^{\top} \widetilde{\N}_{d-1} = \hat\C,
\end{equation}
i.e., the second equation of~\eqref{eq:bigsystem3}.

Let us parameterize $\widetilde{\N}$ by $\widetilde{\N} =\begin{bmatrix} \xi_1 & 0 \\ 0 & \xi_2\end{bmatrix} \begin{bmatrix} 1 & \eta \\ \theta & 1\end{bmatrix} $. 
Then
\[
\widetilde{\N}_{d-1} =  \begin{bmatrix} \xi_1^{d-1} & 0 \\ 0 & \xi_2^{d-1}\end{bmatrix} \begin{bmatrix} 1 & \eta^{d-1} \\ \theta^{d-1} & 1\end{bmatrix},
\]
and $\widetilde{\N}^{\top} \widetilde{\N}_{d-1} = \hat\C$ can be rewritten as
\begin{equation}\label{eq17}
\begin{bmatrix} 1 & \eta^{d-1} \\ \theta^{d-1} & 1 \end{bmatrix} \pmb{\hat C}^{-1} \begin{bmatrix} 1 & \eta \\ \theta & 1  \end{bmatrix}^{\top} = \begin{bmatrix} \xi_1^{-d} & 0 \\ 0 & \xi_2^{-d} \end{bmatrix}.
\end{equation}
Note that
\begin{equation*}
\pmb{\hat C}^{-1} := \begin{bmatrix} m_{11} & m_{12} \\ m_{21} & m_{22} \end{bmatrix} = \frac{1}{c_{11}c_{22} - c_{12}c_{21}} \begin{bmatrix} c_{22} & -c_{12} \\ -c_{21} & c_{11}\end{bmatrix}.
\end{equation*}
Hence the off-diagonal entries in \eqref{eq17} yield the following equations
\begin{equation*}
\left\{\begin{array}{l}
%\eta^{d-1}(m_{11} + m_{21}\theta) + (m_{12} + m_{22}\theta) = 0\\
%(m_{11}\eta + m_{21}) + \theta^{d-1}(m_{12} + m_{22}\eta) = 0
\eta^{d-1}(m_{22} + m_{21}\theta) + (m_{12} + m_{11}\theta) = 0\\
(m_{22}\eta + m_{21}) + \theta^{d-1}(m_{11} + m_{12}\eta) = 0
\end{array}\right.,
\end{equation*}
which imply that
\begin{equation*}
\left\{\begin{array}{l}
\theta = - \frac{\eta^{d-1}m_{22}+ m_{12}}{\eta^{d-1}m_{21} + m_{11}}\\
(m_{22}\eta + m_{21})(-m_{21}\eta^{d-1} - m_{11})^{d-1} + (m_{22}\eta^{d-1}+ m_{12})^{d-1}(m_{12}\eta + m_{11}) = 0
\end{array}\right.,
\end{equation*}
and the second equation is exactly Equation~\eqref{eq:polyc}.
From \eqref{eq17}, we get
\begin{equation*}
 \xi_1 = \left(\begin{bmatrix} 1 & \eta^{d-1} \end{bmatrix} \hat\C^{-1} \begin{bmatrix} 1 & {\eta} \end{bmatrix}^{\top}\right)^{-\frac{1}{d}}, \quad \xi_2 = \left(\begin{bmatrix} \theta^{d-1} & 1 \end{bmatrix} \hat\C^{-1} \begin{bmatrix} \theta & 1 \end{bmatrix}^{\top}\right)^{-\frac{1}{d}}.
\end{equation*}

Finally, by injecting the second equation in the first in \eqref{eq:bigsystem3}, we obtain
\begin{equation}
    \widetilde{\pmb{K}} = \widetilde{\N}\hat\D \hat\C^{-1}\widetilde{\N}^{\top}. 
\end{equation}
We can now invert \eqref{eq:changofvarKtilde}. Note that we choose to solve the sign ambiguity on $\hat\rho$ and $\hat\beta_2$ by considering $\hat\rho\geq 0$. 
It is without loss of generality since, if $\rho < 0$, we can consider $-\u_2$ instead of $\u_2$ (and $-\beta_2$ instead of $\beta_2$ if $d$ is odd) without changing the model \eqref{eq: symPCA}. Therefore, we can naturally choose $\rho\geq 0$.
\end{proof}

\bibliographystyle{alpha}
\bibliography{Spiked_tensor.bib}

\newcommand{\etalchar}[1]{$^{#1}$}
\begin{thebibliography}{CMDL{\etalchar{+}}15}

\bibitem[AGJ20]{arous2020algorithmic}
G{\'e}rard~Ben Arous, Reza Gheissari, and Aukosh Jagannath.
\newblock Algorithmic thresholds for tensor {PCA}.
\newblock {\em The Annals of Probability}, 48(4):2052--2087, 2020.

\bibitem[AGP24a]{arous2024high}
G{\'e}rard~Ben Arous, C{\'e}dric Gerbelot, and Vanessa Piccolo.
\newblock High-dimensional optimization for multi-spiked tensor {PCA}.
\newblock {\em arXiv preprint arXiv:2408.06401}, 2024.

\bibitem[AGP24b]{arous2024stochastic}
G{\'e}rard~Ben Arous, C{\'e}dric Gerbelot, and Vanessa Piccolo.
\newblock Stochastic gradient descent in high dimensions for multi-spiked
  tensor {PCA}.
\newblock {\em arXiv preprint arXiv:2410.18162}, 2024.

\bibitem[AMMN19]{arous2019landscape}
G{\'e}rard~Ben Arous, Song Mei, Andrea Montanari, and Mihai Nica.
\newblock The landscape of the spiked tensor model.
\newblock {\em Communications on Pure and Applied Mathematics},
  72(11):2282--2330, 2019.

\bibitem[BAHH23]{ben2023long}
G{\'e}rard Ben~Arous, Daniel~Zhengyu Huang, and Jiaoyang Huang.
\newblock Long random matrices and tensor unfolding.
\newblock {\em The Annals of Applied Probability}, 33(6B):5753--5780, 2023.

\bibitem[BBAP05]{BBP05}
Jinho Baik, G{\'e}rard Ben~Arous, and Sandrine P{\'e}ch{\'e}.
\newblock Phase transition of the largest eigenvalue for nonnull complex sample
  covariance matrices.
\newblock {\em The Annals of Probability}, 33(5):1643--1697, 2005.

\bibitem[BGN11]{BGN11}
Florent Benaych-Georges and Raj~Rao Nadakuditi.
\newblock The eigenvalues and eigenvectors of finite, low rank perturbations of
  large random matrices.
\newblock {\em Advances in Mathematics}, 227(1):494--521, 2011.

\bibitem[BGN12]{benaych2012singular}
Florent Benaych-Georges and Raj~Rao Nadakuditi.
\newblock The singular values and vectors of low rank perturbations of large
  rectangular random matrices.
\newblock {\em Journal of Multivariate Analysis}, 111:120--135, 2012.

\bibitem[BJNP13]{birnbaum2013minimax}
Aharon Birnbaum, Iain~M Johnstone, Boaz Nadler, and Debashis Paul.
\newblock Minimax bounds for sparse {PCA} with noisy high-dimensional data.
\newblock {\em Annals of statistics}, 41(3):1055, 2013.

\bibitem[BS06]{baik2006eigenvalues}
Jinho Baik and Jack~W Silverstein.
\newblock Eigenvalues of large sample covariance matrices of spiked population
  models.
\newblock {\em Journal of multivariate analysis}, 97(6):1382--1408, 2006.

\bibitem[BY12]{bai2012sample}
Zhidong Bai and Jianfeng Yao.
\newblock On sample eigenvalues in a generalized spiked population model.
\newblock {\em Journal of Multivariate Analysis}, 106:167--177, 2012.

\bibitem[CDMF09]{CDMF09}
Mireille Capitaine, Catherine Donati-Martin, and Delphine F{\'e}ral.
\newblock The largest eigenvalues of finite rank deformation of large {W}igner
  matrices: {C}onvergence and nonuniversality of the fluctuations.
\newblock {\em The Annals of Probability}, 37(1):1--47, 2009.

\bibitem[Che19]{chen2019phase}
Wei-Kuo Chen.
\newblock Phase transition in the spiked random tensor with {R}ademacher prior.
\newblock {\em The Annals of Statistics}, 47(5):2734--2756, 2019.

\bibitem[CHL21]{chen2021phase}
Wei-Kuo Chen, Madeline Handschy, and Gilad Lerman.
\newblock Phase transition in random tensors with multiple independent spikes.
\newblock {\em The Annals of Applied Probability}, 31(4):1868--1913, 2021.

\bibitem[CMDL{\etalchar{+}}15]{CMDZZCP15}
Andrzej Cichocki, Danilo Mandic, Lieven De~Lathauwer, Guoxu Zhou, Qibin Zhao,
  Cesar Caiafa, and Huy~Anh Phan.
\newblock Tensor decompositions for signal processing applications: From
  two-way to multiway component analysis.
\newblock {\em IEEE Signal Processing Magazine}, 32(2):145--163, 2015.

\bibitem[CMW13]{cai2013sparse}
T~Tony Cai, Zongming Ma, and Yihong Wu.
\newblock {Sparse PCA: Optimal rates and adaptive estimation}.
\newblock {\em The Annals of Statistics}, 41(6):3074--3110, 2013.

\bibitem[CMW15]{cai2015optimal}
Tony Cai, Zongming Ma, and Yihong Wu.
\newblock Optimal estimation and rank detection for sparse spiked covariance
  matrices.
\newblock {\em Probability theory and related fields}, 161(3):781--815, 2015.

\bibitem[Com14]{Comon14}
Pierre Comon.
\newblock Tensors : A brief introduction.
\newblock {\em IEEE Signal Processing Magazine}, 31(3):44--53, 2014.

\bibitem[DGJ18]{donoho2018optimal}
David~L Donoho, Matan Gavish, and Iain~M Johnstone.
\newblock Optimal shrinkage of eigenvalues in the spiked covariance model.
\newblock {\em Annals of statistics}, 46(4):1742, 2018.

\bibitem[DSL08]{deSilva2008tensor}
Vin De~Silva and Lek-Heng Lim.
\newblock Tensor rank and the ill-posedness of the best low-rank approximation
  problem.
\newblock {\em SIAM J. Matrix Anal. Appl.}, 30(3):1084--1127, 2008.

\bibitem[EK08]{el2008spectrum}
Noureddine El~Karoui.
\newblock {Spectrum estimation for large dimensional covariance matrices using
  random matrix theory}.
\newblock {\em The Annals of Statistics}, 36(6):2757--2790, 2008.

\bibitem[FP07]{feral2007largest}
Delphine F{\'e}ral and Sandrine P{\'e}ch{\'e}.
\newblock The largest eigenvalue of rank one deformation of large {W}igner
  matrices.
\newblock {\em Communications in mathematical physics}, 272:185--228, 2007.

\bibitem[GCC22]{GCC22}
Jos{\'e} Henrique de~Morais Goulart, Romain Couillet, and Pierre Comon.
\newblock A random matrix perspective on random tensors.
\newblock {\em Journal of Machine Learning Research}, 23(264):1--36, 2022.

\bibitem[HHYC22]{huang2022power}
Jiaoyang Huang, Daniel~Z. Huang, Qing Yang, and Guang Cheng.
\newblock {Power iteration for tensor PCA}.
\newblock {\em Journal of Machine Learning Research}, 23(128):1--47, 2022.

\bibitem[HL13]{HL13}
Christopher~J. Hillar and Lek-Heng Lim.
\newblock Most tensor problems are {NP}-hard.
\newblock {\em Journal of the ACM}, 60(6):1--39, nov 2013.

\bibitem[HSS15]{hopkins2015tensor}
Samuel~B Hopkins, Jonathan Shi, and David Steurer.
\newblock Tensor principal component analysis via sum-of-square proofs.
\newblock In {\em Conference on Learning Theory}, pages 956--1006. PMLR, 2015.

\bibitem[JL09]{johnstone2009consistency}
Iain~M Johnstone and Arthur~Yu Lu.
\newblock On consistency and sparsity for principal components analysis in high
  dimensions.
\newblock {\em Journal of the American Statistical Association},
  104(486):682--693, 2009.

\bibitem[JLM20]{jagannath2020statistical}
Aukosh Jagannath, Patrick Lopatto, and L{\'e}o Miolane.
\newblock Statistical thresholds for tensor {PCA}.
\newblock {\em The Annals of Applied Probability}, 30(4):1910--1933, 2020.

\bibitem[Joh01]{johnstone2001distribution}
Iain~M Johnstone.
\newblock On the distribution of the largest eigenvalue in principal components
  analysis.
\newblock {\em The Annals of statistics}, 29(2):295--327, 2001.

\bibitem[JP18]{johnstone2018pca}
Iain~M Johnstone and Debashis Paul.
\newblock {PCA in high dimensions: An orientation}.
\newblock {\em Proceedings of the IEEE}, 106(8):1277--1292, 2018.

\bibitem[LML{\etalchar{+}}17]{lesieur2017statistical}
Thibault Lesieur, L{\'e}o Miolane, Marc Lelarge, Florent Krzakala, and Lenka
  Zdeborov{\'a}.
\newblock Statistical and computational phase transitions in spiked tensor
  estimation.
\newblock In {\em 2017 IEEE International Symposium on Information Theory
  (ISIT)}, pages 511--515. IEEE, 2017.

\bibitem[LW12]{ledoit2012nonlinear}
Olivier Ledoit and Michael Wolf.
\newblock {Nonlinear shrinkage estimation of large-dimensional covariance
  matrices}.
\newblock {\em The Annals of Statistics}, 40(2):1024--1060, 2012.

\bibitem[Ma13]{ma2013sparse}
Zongming Ma.
\newblock {Sparse principal component analysis and iterative thresholding}.
\newblock {\em The Annals of Statistics}, 41(2):772--801, 2013.

\bibitem[MR14]{richard2014statistical}
Andrea Montanari and Emile Richard.
\newblock A statistical model for tensor {PCA}.
\newblock {\em Advances in neural information processing systems},
  27:2897--2905, 2014.

\bibitem[MRZ15]{montanari2015limitation}
Andrea Montanari, Daniel Reichman, and Ofer Zeitouni.
\newblock On the limitation of spectral methods: From the gaussian hidden
  clique problem to rank-one perturbations of gaussian tensors.
\newblock In C.~Cortes, N.~Lawrence, D.~Lee, M.~Sugiyama, and R.~Garnett,
  editors, {\em Advances in Neural Information Processing Systems}, volume~28.
  Curran Associates, Inc., 2015.

\bibitem[OMH13]{onatski2013asymptotic}
Alexei Onatski, Marcelo~J Moreira, and Marc Hallin.
\newblock {Asymptotic power of sphericity tests for high-dimensional data}.
\newblock {\em The Annals of Statistics}, 41(3):1204--1231, 2013.

\bibitem[Pau07]{paul2007asymptotics}
Debashis Paul.
\newblock Asymptotics of sample eigenstructure for a large dimensional spiked
  covariance model.
\newblock {\em Statistica Sinica}, pages 1617--1642, 2007.

\bibitem[P{\'e}c06]{Peche06}
Sandrine P{\'e}ch{\'e}.
\newblock The largest eigenvalue of small rank perturbations of {H}ermitian
  random matrices.
\newblock {\em Probability Theory and Related Fields}, 134(1):127--173, 2006.

\bibitem[PWB20]{PWB20}
Amelia Perry, Alexander~S. Wein, and Afonso~S. Bandeira.
\newblock Statistical limits of spiked tensor models.
\newblock {\em Annales de l'Institut Henri Poincar{\'e}, Probabilit{\'e}s et
  Statistiques}, 56(1):230--264, 2020.

\bibitem[PWBM18]{perry2018optimality}
Amelia Perry, Alexander~S Wein, Afonso~S Bandeira, and Ankur Moitra.
\newblock {Optimality and sub-optimality of PCA I: Spiked random matrix
  models}.
\newblock {\em The Annals of Statistics}, 46(5):2416--2451, 2018.

\bibitem[SC10]{SC10}
Alwin Stegeman and Pierre Comon.
\newblock Subtracting a best rank-1 approximation may increase tensor rank.
\newblock {\em Linear Algebra and its Applications}, 433(7):1276--1300, 2010.

\bibitem[SDLF{\etalchar{+}}17]{SDFHPF17}
Nicholas~D. Sidiropoulos, Lieven De~Lathauwer, Xiao Fu, Kejun Huang,
  Evangelos~E. Papalexakis, and Christos Faloutsos.
\newblock Tensor decomposition for signal processing and machine learning.
\newblock {\em IEEE Transactions on Signal Processing}, 65(13):3551--3582,
  2017.

\bibitem[SGC24]{SGC24}
Mohamed El~Amine Seddik, Maxime Guillaud, and Romain Couillet.
\newblock When random tensors meet random matrices.
\newblock {\em The Annals of Applied Probability}, 34(1A):203--248, 2024.

\bibitem[SGG23]{SGG23}
Mohamed El~Amine Seddik, Jos{\'e} Henrique de~Morais Goulart, and Maxime
  Guillaud.
\newblock Hotelling deflation on large symmetric spiked tensors.
\newblock {\em arXiv preprint arXiv:2304.10248}, 2023.

\bibitem[Tao12]{tao2012topics}
Terence Tao.
\newblock {\em Topics in random matrix theory}, volume 132.
\newblock American Mathematical Soc., 2012.

\bibitem[VL13]{vu2013minimax}
Vincent~Q Vu and Jing Lei.
\newblock {Minimax sparse principal subspace estimation in high dimensions}.
\newblock {\em The Annals of Statistics}, 41(6):2905--2947, 2013.

\bibitem[VNVM14]{VNVM14}
Nick Vannieuwenhoven, Johannes Nicaise, Raf Vandebril, and Karl Meerbergen.
\newblock On generic nonexistence of the {S}chmidt-{E}ckart-{Y}oung
  decomposition for complex tensors.
\newblock {\em SIAM Journal on Matrix Analysis and Applications},
  35(3):886--903, 2014.

\end{thebibliography}

\end{document}